\documentclass[12pt]{amsart}
\usepackage{amscd,amssymb,amsthm,amsmath,amssymb,textcomp,multirow,rotating,longtable,mathrsfs,fancyhdr,textgreek,mathtools,wasysym,pdflscape,imakeidx,makecell}
\usepackage{tikz}
\usepackage{multicol}
\usepackage{tabulary}
\usepackage{colortbl}
\usepackage{float}

\newcommand{\DC}{\mathbb{C}}

\newcommand{\DP}{\mathbb{P}}
\newcommand{\DA}{\mathbb{A}}

\makeatletter
\newcommand*{\da@rightarrow}{\mathchar"0\hexnumber@\symAMSa 4B }
\newcommand*{\da@leftarrow}{\mathchar"0\hexnumber@\symAMSa 4C }
\newcommand*{\xdashrightarrow}[2][]{%
  \mathrel{%
    \mathpalette{\da@xarrow{#1}{#2}{}\da@rightarrow{\,}{}}{}%
  }%
}
\newcommand{\xdashleftarrow}[2][]{%
  \mathrel{%
    \mathpalette{\da@xarrow{#1}{#2}\da@leftarrow{}{}{\,}}{}%
  }%
}
\newcommand*{\da@xarrow}[7]{%
  \sbox0{$\ifx#7\scriptstyle\scriptscriptstyle\else\scriptstyle\fi#5#1#6\m@th$}%
  \sbox2{$\ifx#7\scriptstyle\scriptscriptstyle\else\scriptstyle\fi#5#2#6\m@th$}%
  \sbox4{$#7\dabar@\m@th$}%
  \dimen@=\wd0 %
  \ifdim\wd2 >\dimen@
    \dimen@=\wd2 %
  \fi
  \count@=2 %
  \def\da@bars{\dabar@\dabar@}%
  \@whiledim\count@\wd4<\dimen@\do{%
    \advance\count@\@ne
    \expandafter\def\expandafter\da@bars\expandafter{%
      \da@bars
      \dabar@ 
    }%
  }%
  \mathrel{#3}%
  \mathrel{%
    \mathop{\da@bars}\limits
    \ifx\\#1\\%
    \else
      _{\copy0}%
    \fi
    \ifx\\#2\\%
    \else
      ^{\copy2}%
    \fi
  }%
  \mathrel{#4}%
}
\makeatother

\newtheorem{theorem}{Theorem}[section]

\newtheorem{lemma}[theorem]{Lemma}
\newtheorem{corollary}[theorem]{Corollary}
\newtheorem*{corollary*}{Corollary}
\newtheorem*{maincorollary*}{Main Corollary}

\newtheorem*{conjecture*}{Conjecture}

\newtheorem*{problem*}{Problem}
\newtheorem*{calabiproblem*}{Calabi Problem}
\newtheorem*{Main Theorem*}{Main Theorem}

\newtheorem*{theorem*}{Theorem}
\newtheorem*{maintheorem*}{Main Theorem}

\theoremstyle{definition}

\newtheorem*{example*}{Example}
\newtheorem{definition}[theorem]{Definition}

\theoremstyle{remark}
\newtheorem{remark}[theorem]{Remark}
\newtheorem*{remark*}{Remark}

\makeatletter\@addtoreset{equation}{subsection} \makeatother

\usepackage{graphicx}
\usepackage{hyperref}
\usepackage{listings}
\usepackage{mathtools}
\usepackage{textcomp}
\usepackage[matrix,arrow,curve]{xy}

\makeindex

\title{$\delta$-invariants of log Fano planes}

    \author{Elena Denisova}

\begin{document}

\maketitle

\begin{abstract}
We compute the $\delta$-invariant for pairs $(\mathbb{P}^2, \lambda C_d)$, where $C_d$ is a plane curve of degree $d \leq 4$. These computations provide new examples of $K$-stable and $K$-semistable log Fano pairs, and contribute to the study of $K$-stability of log Fano varieties via the Abban--Zhuang method, which reduces higher-dimensional problems to the surface case.
\end{abstract}

\section{Introduction}
\subsection{History and Results}
\noindent The first problems involving the concept of curves emerged very early in human history. These problems were practical in nature, aimed at tasks such as measuring distances for agriculture and construction. Later in antiquity, around 600 BC, Thales made one of the earliest attempts to define these concepts more rigorously. Over time, the theory evolved, and by the beginning of the 18th century, the modern notion of an algebraic plane curve had taken shape:

\begin{definition}
{\it A projective algebraic plane curve $C_d$ of degree $d$} is a curve in the projective plane $\mathbb{P}^2$ defined by a single polynomial equation $F(x_0,x_1,x_2)=0$, where $F$ is a homogeneous polynomial of degree $d$.
\end{definition}

\noindent A projective algebraic plane curve may be smooth or singular, and a natural question is how to classify such curves for each degree. To keep things simple, we first focus on reduced curves. When $d=1$, the only possibility is a line. For $d=2$, there are two possibilities: a smooth conic or a pair of distinct lines. When $d=3$, the classification becomes more intricate. A reduced projective plane cubic curve can be one of the following:
 \begin{itemize}
     \item an irreducible smooth curve
     \item an irreducible curve with a node
     \item an irreducible curve with a cusp
     \item a conic and a general line
    \item a conic and a tangent line
     \item three general lines
     \item a union of three different lines with intersection at one point
 \end{itemize}
 \begin{figure}[h!]
   \label{Cubic Curves}
\includegraphics[width=16cm]{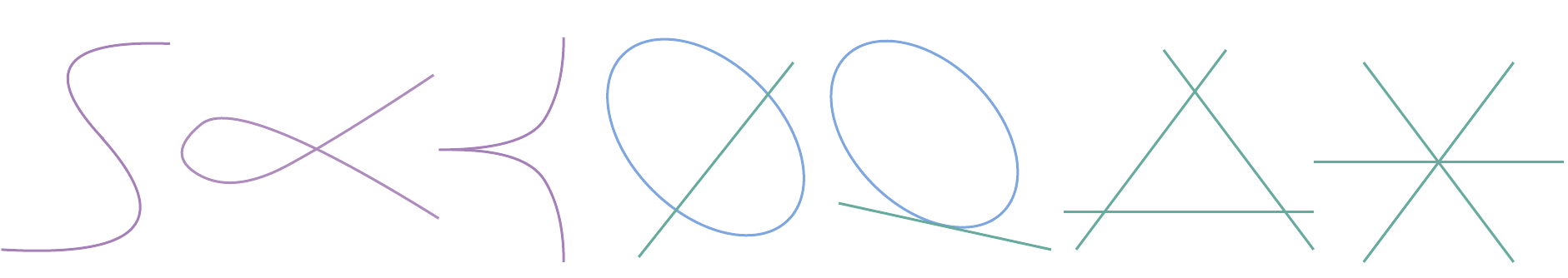}
 \caption{Cubic Curves}
 \end{figure}
 \newpage
\noindent Plane quartic curves were classified in \cite{Hui}, and we observe that the complexity of the classification problem increases dramatically with the degree. Specifically, there are $20$ types of irreducible quartics (see Figure~\ref{Irreducible Quartic Curves}), $13$ types consisting of an irreducible cubic and a line (Figure~\ref{Line and Irreducible cubic}), $5$ types splitting into two distinct conics (Figure~\ref{Two Conics}), $5$ types formed by a conic and two lines (Figure~\ref{Conic and Two Lines}), and $3$ types comprising four lines (Figure~\ref{Four Lines}). In \cite{Namba}, the classification of irreducible quintics was given by Makoto Namba, and in \cite{Yang}, Jin-Gen Yang provided a complete classification of sextics with simple singularities. For higher degrees, the classification problem becomes extremely intricate.

\vspace{1em}
\noindent Another natural question that arises is: \emph{How singular is $C_d$}? To address this, we consider pairs of the form $(\mathbb{P}^2, \lambda C_d)$.

\begin{definition}
Let $(X,D)$ be a pair where $X$ is a normal variety and $D = \sum a_i D_i$ is an effective $\mathbb{Q}$-linear combination of divisors $D_i$ on $X$, with coefficients $a_i > 0$, such that $K_X + D$ is $\mathbb{Q}$-Cartier. Fix a birational morphism $f: Y \to X$, and let $E \subset Y$ be an exceptional divisor over $X$. The \textit{discrepancy of the pair $(X, D)$ along $E$}, denoted $a(E, X, D)$, is defined by:
\begin{align*}
K_Y + f_*^{-1}D - f^*(K_X + D) \equiv \sum_i a(E_i, X, D) E_i,
\end{align*}
where $E_i \subset Y$ are the irreducible components of the exceptional locus of $f$. The \textit{discrepancy of the pair $(X,D)$} is then:
\begin{align*}
\mathrm{discrep}(X,D) = \inf_E \left\{ a(E,X,D) \mid E\ \text{exceptional divisor over } X \right\}.
\end{align*}
If $\dim X \geq 2$, the pair $(X,D)$ is called \textit{terminal} if and only if $\mathrm{discrep}(X,D) > 0$.
\\
Now, suppose $X$ has log terminal singularities and let $P \in X$ be a point. Then the \textit{log canonical threshold} at $P$ is defined by:
\begin{align*}
\mathrm{lct}_P(X,D) = \sup \left\{ \lambda \in \mathbb{Q} \mid (X, \lambda D)\ \text{is log canonical at } P \right\},
\end{align*}
and the global log canonical threshold is:
\begin{align*}
\mathrm{lct}(X,D) = \inf_{P \in X} \mathrm{lct}_P(X,D) = \sup \left\{ \lambda \in \mathbb{Q} \mid (X,\lambda D)\ \text{is log canonical} \right\}.
\end{align*}
\end{definition}

\noindent The values of log canonical thresholds for pairs $(\mathbb{P}^2, C_d)$ with $d \in \{3,4\}$ can be found, for example, in \cite{Cheltsov17}. This notion is closely related to the concept of $K$-stability via the Tian’s criterion:

\begin{theorem}[Tian's criterion]
Let $X$ be a $\mathbb{Q}$-Fano variety. If $\alpha(X) > \frac{n}{n+1}$, then $X$ is $K$-stable, where
$$
\alpha(X) = \inf \left\{ \mathrm{lct}(X, D) \mid D \sim_{\mathbb{Q}} -K_X,\ D\ \text{effective} \right\}.
$$
\end{theorem}
\noindent This shows that \textit{log Fano pairs}, i.e., pairs $(X, \lambda D)$ with $\lambda \in (0, \mathrm{lct}(X,D))$, are natural objects of study. These have been investigated extensively in the context of their $K$-moduli spaces, notably in \cite{ADVL19}. For instance, the authors proved that for sufficiently small $\lambda$, the $K$-moduli space parameterizing $K$-semi(stable) limits of pairs $(\mathbb{P}^n, \lambda V_d)$, where $V_d$ is a degree $d$ hypersurface, coincides with the GIT moduli space. This result was used to analyze $K$-moduli spaces of plane curves of degree up to 6.
\noindent What further properties can we study for such pairs? It is known (see \cite{Fujita}) that for a log Fano variety $(X,\Delta)$, one can define the $\delta$-invariant, which satisfies $\delta(X,\Delta) > 1$ if and only if $(X,\Delta)$ is $K$-stable. In some higher-dimensional cases, this reduces to computing the polarized $\delta$-invariant of log Fano pairs $(\mathbb{P}^2, \lambda C_d)$, where $C_d \subset \mathbb{P}^2$ is a plane curve of degree $d$.
\noindent In this article, we compute the exact value of $\delta(\mathbb{P}^2, \lambda C_d)$ for $d \leq 4$ over a suitable interval. More precisely, we prove that:
\newpage 
\begin{maintheorem*}
Let  $(\DP^2,\lambda C_d)$ be a pair where $C_d\subset \DP^2$ is a curve of degree $d\le 4$. If the worst singularity has type $A_5$ we denote $L$ be the tangent line at this point.  Then 
\begin{align*}
\delta(\DP^2,\lambda C_1)&= \frac{3(1-\lambda)}{3- \lambda}\text{ for }\lambda\in\big[0,1\big].\\
\delta(\DP^2,\lambda C_2)&=
 \begin{cases}
  1\text{ if $C_2$ is smooth  for }\lambda\in\big[0,\frac{3}{4}\big],\\
\frac{3-3\lambda}{3-2\lambda} \text{ if $C_2$ is a union of two different lines for }\lambda\in\big[0,1\big].
\end{cases}
\\\delta(\DP^2,\lambda C_3)&=
 \begin{cases}
  \frac{3-2\lambda}{3- 3\lambda }\text{ if $C_3$ is smooth and there's no 3-tangent to }C_3\text{ for }\lambda\in\big[0,\frac{3}{4}\big],\\
\frac{4-3\lambda}{4- 4\lambda }\text{ if $C_3$ is smooth and there is   a 3-tangent to }C_3\text{ for }\lambda\in\big[0,\frac{8}{9}\big],\\
1 \text{ if the worst singularity of $C_3$ has type $A_1$ for }\lambda\in\big[0,1\big],\\
 \frac{5-6\lambda}{5- 5\lambda }\text{ if the worst singularity of $C_3$ has type $A_2$ for }\lambda\in\big[0,\frac{5}{6}\big],\\
 \frac{3-4\lambda}{3-3\lambda}\text{ if the worst singularity of $C_3$ has type $A_3$ for }\lambda\in\big[0,\frac{3}{4}\big],\\
 \frac{2-3\lambda}{2- 2\lambda }\text{ if the worst singularity of $C_3$ has type $D_4$ for }\lambda\in\big[0,\frac{2}{3}\big].
 \end{cases}
 \\\delta(\DP^2,\lambda C_4)&=
 \begin{cases}
 \frac{3}{4}\cdot\frac{4-3\lambda}{3- 4\lambda }\text{ if $C_4$ is smooth and there is no  4-tangent to }C_4\text{ for }\lambda\in\big[0,\frac{3}{4}\big]\\
 \frac{3}{5}\cdot\frac{5-4\lambda}{3- 4\lambda }\text{ if $C_4$ is smooth and there exists a 4-tangent to }C_4\text{ for }\lambda\in\big[0,\frac{3}{4}\big],\\
 \frac{3}{2}\cdot\frac{2-2\lambda}{3- 4\lambda }\text{ if the worst singularity of $C_4$ has type $A_1$ for }\lambda\in\big[0,\frac{3}{4}\big],\\
 \frac{3}{5}\cdot\frac{5-6\lambda}{3- 4\lambda }\text{ if the worst singularity of $C_4$ has type $A_2$ for }\lambda\in\big[0,\frac{3}{4}\big],\\
 1\text{ if the worst singularity of $C_4$ has type $A_3$ for }\lambda\in\big[0,\frac{3}{4}\big],\\
 \frac{6}{13}\cdot \frac{7-10\lambda}{3- 4\lambda }\text{ if the worst singularity of $C_4$ has type $A_4$ for }\lambda\in \big[\frac{3}{8}, \frac{7}{10}\big],\\
 \frac{6}{7}\cdot \frac{4-6\lambda}{3- 4\lambda }\text{  if the worst singularity of $C_4$ has type $A_5$, and $L\not \subset C_4$ for }\lambda\in \big[\frac{3}{8}, \frac{2}{3}\big],\\
 \frac{3}{4}\cdot \frac{4-6\lambda}{3- 4\lambda }\text{  if the worst singularity of $C_4$ has type $A_5$, and $L\subset C_4$ for }\lambda\in \big[0, \frac{2}{3}\big],\\
 \frac{2}{5}\cdot \frac{9-14\lambda}{3- 4\lambda }\text{ if the worst singularity of $C_4$ has type $A_6$ for }\lambda\in \big[\frac{3}{8}, \frac{7}{14}\big],\\
 \frac{3}{4}\cdot\frac{5-8\lambda}{3- 4\lambda }\text{ if the worst singularity of $C_4$ has type $A_7$ for }\lambda\in\big[\frac{3}{8};\frac{5}{8}\big],\\
 \frac{3}{2}\cdot\frac{2-3\lambda}{3- 4\lambda }\text{ if the worst singularity of $C_4$ has type $D_4$ for }\lambda\in\big[0,\frac{2}{3}\big],\\
 \frac{3}{5}\cdot\frac{5-8\lambda}{3- 4\lambda }\text{ if the worst singularity of $C_4$ has type $D_5$ for }\lambda\in\big[0,\frac{5}{8}\big],\\
 \frac{3-5\lambda}{3- 4\lambda }\text{ if the worst singularity of $C_4$ has type $D_6$ for }\lambda\in\big[0,\frac{3}{5}\big],\\
 \frac{3}{7}\cdot\frac{7-12\lambda}{3- 4\lambda }\text{ if the worst singularity of $C_4$ has type $E_6$ for }\lambda\in\big[0,\frac{7}{12}\big],\\
 \frac{3}{5}\cdot\frac{5-9\lambda}{3- 4\lambda }\text{ if the worst singularity of $C_4$ has type $E_7$ for }\lambda\in \big[0, \frac{5}{9}\big],\\
  \frac{3-6\lambda}{3- 4\lambda }\text{ if  $C_4$ has a ``four-line" singularity for }\lambda\in\big[0,\frac{1}{2}\big].
 \end{cases}
 \end{align*}
If $C_d$ contains a conic component of multiplicity $c=2$ or a line component of multiplicity $l\ge 2$. Then:
 $$\delta(\DP^2,\lambda C_d)=
 \begin{cases}
     \frac{3(1-l\cdot \lambda)}{3-d\lambda}\text{ for }\lambda\in \big[0,\frac{1}{l}\big]\text{ if }l\ge 2,\\
     1\text{ for }\lambda\in\big [0,\frac{3}{8}\big ] \text{ if } c=2.
 \end{cases}$$
  \end{maintheorem*}
  \begin{remark}
  Let $C\subset \DP^2$ be a plane curve of degree $4$ with $A_4$, $A_5(1)$, $A_6$ or $A_7$ singularities at point $P\in C$  on $C$. Then
$$\delta_P(\DP^2,\lambda C)\ge\frac{3}{2(3-d \lambda )}\text{ for }\lambda\in \Big[0,\frac{3}{8}\Big].$$
  \end{remark}
  \noindent {\bf Acknowledgments:} I am grateful to my supervisor Professor Ivan Cheltsov for the introduction to the topic and continuous support.
  \subsection{Applications.} 
   We assume that $\lambda$ belongs to the respective interval from the main theorem unless stated otherwise.
  \subsubsection{Du Val del Pezzo surfaces of degree $2$}
    Let $X$ be a del Pezzo surface of degree $2$ with Du Val singularities.  The system $|-K_{X}|$ gives a double cover
$\iota: X \to \DP^2$ ramified in a curve $C$ of degree $4$ on $\DP^2$. 
We have $\iota^*(K_{\DP^2}+\frac{1}{2}C)=K_X$. Thus, by \cite{Dervan16} $X$ is $K$-stable if and only inf $(\DP^2,\frac{1}{2}C)$ is $K$-stable. In fact, \cite{DenisovadP2} showed that  $\delta(X)=\delta(\DP^2,\frac{1}{2}C)$.
   \subsubsection{Projective Space} Another possible application is the following.  Consider log Fano pairs $(\DP^3,\lambda S_s)$ where $S_s$ is a reduced surface in $\DP^3$ of degree $s$ and $\lambda\in \big(0,\min\big\{ \frac{4}{s}, \mathrm{lct}(\DP^3, S_s)\big\}\big)$.  
   \begin{lemma}
       Let $P\in S_s$ be a smooth point on $S_s$ and $s\ge 3$. If $P$ is a smooth point on $S_s$. Then $\delta_P(\DP^3, \lambda S)\ge 1$.
   \end{lemma}
   \begin{proof}
       Let $T$ be a general plane containing a point $P$. Let $P(u)$ and $N(u)$ be a positive and negative part of the Zariski decomposition of the divisor $-(K_{\DP^3}+\lambda S)-uT$. Then:
  $$P(u)=-(K_{\DP^3}+\lambda S)-uT\text{ and }N(u)=0\text{ for }u\in [0,4-\lambda s].$$
  We have $A_{(\DP^3,\lambda S)}(T)=1$ and 
  $$S_{(\DP^3,\lambda S)}(T)= \frac{1}{(4-\lambda s)^3}\int_0^{4-\lambda s}(4-\lambda s-u)^3du=\frac{4-\lambda s}{4}.$$
  Moreover, for  any divisor $F$ such that $P\in C_{T}(F)$    we get
{\allowdisplaybreaks \begin{align*}
S\big(W^{T}_{\bullet,\bullet}&;F\big)=\\
&=\frac{3}{(-K_{\DP^3}-\lambda S)^3}\Bigg(\int_0^{4-\lambda s}\big(P(u)^{2}\cdot T\big)\cdot\mathrm{ord}_{P}\Big(N(u)\big\vert_{T}\Big)du+\int_0^{4-\lambda s}\int_0^\infty \mathrm{vol}\big(P(u)\big\vert_{T}-vF\big)dvdu\Bigg)=\\
&= \frac{3}{(4-\lambda s)^3}\int_0^\tau\int_0^\infty \mathrm{vol}\big(P(u)\big\vert_{T}-vF\big)dvdu=\\
&= \frac{3}{(4-\lambda s)^3}\Bigg(\int_0^{4-\lambda s}\int_0^\infty \mathrm{vol}\big(-(K_{\DP^3}+\lambda S)|_T-uT|_T-vF\big)dvdu\Bigg)=\\
&= \frac{3}{(4-\lambda s)^3}\Bigg(\int_0^{4-\lambda s}\int_0^\infty \mathrm{vol}\big(-((K_{\DP^3}+T)+\lambda S)|_T-(u-1)T|_T-vF\big)dvdu\Bigg)=\\
&= \frac{3}{(4-\lambda s)^3}\Bigg(\int_0^{4-\lambda s}\int_0^\infty \mathrm{vol}\big(-(K_T+\lambda C_s)-(u-1)T|_T-vF\big)dvdu\Bigg)=\\
&= \frac{3}{(4-\lambda s)^3}\Bigg(\int_0^{4-\lambda s}\int_0^\infty \mathrm{vol}
\big((4-\lambda s -u )T|_T-vF\big)dvdu\Bigg)=\\
&= \frac{3}{(4-\lambda s)^3}\Bigg(\int_0^{4-\lambda s}\int_0^\infty \mathrm{vol}\Big(\frac{4-\lambda s -u }{3-\lambda s} (-K_T-\lambda C_s)-vF\Big)dvdu\Bigg)=\\
&= \frac{3}{(4-\lambda s)^3}\Bigg(\int_0^{4-\lambda s} \frac{(4-\lambda s -u )^3}{(3-\lambda s)^3} \int_0^\infty \mathrm{vol}\big(-(K_T+\lambda C_s)-vF\big)dvdu\Bigg)=\\
&= \frac{3}{(4-\lambda s)^3}\Bigg(\int_0^{4-\lambda s} \frac{(4-\lambda s -u )^3}{3-\lambda s} S_{(T,\lambda C_s)}(F)du\Bigg)= \\
& = \frac{3(4-\lambda s)}{4(3-\lambda s)} S_{(T,\lambda C_s)}(F)\le  \frac{3(4-\lambda s)}{4(3-\lambda s)}\cdot \frac{A_{(T,\lambda C_s)}(F)}{\delta_P(T,\lambda C_s)} =  \frac{3(4-\lambda s)}{4(3-\lambda s)}\cdot \frac{A_{(\DP^2,\lambda C_s)}(F)}{\delta_P(\DP^2,\lambda C_s)} 
\end{align*}}
Then $A_{(\DP^3,\lambda S)}(T)=1$, $S_{(\DP^3,\lambda S)}(T)=\frac{4-\lambda s}{4}$ and:
 {\allowdisplaybreaks 
 \begin{align*}
\delta_P(\DP^3&,\lambda S)
\ge \min\Bigg\{\frac{A_{(\DP^3,\lambda S)}(T)}{S_{(\DP^3,\lambda S)}(T)},\text{ }
\delta_P(T,W^T_{\bullet,\bullet})\Bigg\}\ge\\
& \ge \min \Bigg\{\frac{A_{(\DP^3,\lambda S)}(T)}{S_{(\DP^3,\lambda S)}(T)},\text{ }\inf_{\substack{F/T,\\ P\in C_{\DP^3}(F)}} 
\frac{A_{(T,\lambda C_s)}(F)}{S(W^T; F)}\Bigg\}=\min \Bigg\{\frac{A_{(\DP^3,\lambda S)}(T)}{S_{(\DP^3,\lambda S)}(T)},\text{ }\inf_{\substack{F/T,\\ P\in C_{\DP^3}(F)}} 
\frac{A_{(\DP^2,\lambda C_s)}(F)}{S(W^{\DP^2}; F)}\Bigg\}\ge\\
&\ge \min\Bigg\{\frac{A_{(\DP^3,\lambda S)}(T)}{S_{(\DP^3,\lambda S)}(T)},\text{ } 
 \delta_P(\DP^2,\lambda C_s)\cdot \frac{4(3-\lambda s)}{3(4-\lambda s)}\Bigg\}\ge 1\text{ for }s\ge 3.
    \end{align*}}
   \end{proof}
\noindent If $P$ is a singular point on $S_s$ we consider the  blowup of this point $\pi:\widetilde{\DP}^3\to \DP^3$ at this point with the exceptional divisor $E\cong\DP^2$ and let $m=\mathrm{mult}_PS_s$. Then
  \begin{lemma}
     $$\delta_P(\DP^3,\lambda S_s) \ge \min\Bigg\{\frac{4(3-\lambda m)}{3(4-\lambda s)},\text{ } 
 \delta(\DP^2,\lambda C_m)\cdot \frac{4(3-\lambda m)}{3(4-\lambda s)}\Bigg\}$$
  \end{lemma}
  \begin{proof}
   Let $P(u)$ and $N(u)$ be a positive and negative part of the Zariski decomposition of the divisor $-\pi^*(K_{\DP^3}+\lambda S)-uE$. Then:
  $$P(u)=-\pi^*(K_{\DP^3}+\lambda S)-uE\text{ and }N(u)=0\text{ for }u\in [0,4-\lambda s].$$
  We have: 
  $$S_{(\DP^3,\lambda S)}(E)= \frac{1}{(4-\lambda s)^3}\int_0^{4-\lambda s}\big((4-\lambda s)^3-u^3\big)du=\frac{3(4-\lambda s)}{4}.$$
  Moreover, for any $O\in E$ and for any divisor $F$ such that $O\in C_{E}(F)$    we get
{\allowdisplaybreaks \begin{align*}
S\big(W^{E}_{\bullet,\bullet}&;F\big)=\\
&=\frac{3}{(-K_{\DP^3}-\lambda S)^3}\Bigg(\int_0^{4-\lambda s}\big(P(u)^{2}\cdot E\big)\cdot\mathrm{ord}_{O}\Big(N(u)\big\vert_{E}\Big)du+\int_0^{4-\lambda s}\int_0^\infty \mathrm{vol}\big(P(u)\big\vert_{E}-vF\big)dvdu\Bigg)=\\
&= \frac{3}{(4-\lambda s)^3}\int_0^\tau\int_0^\infty \mathrm{vol}\big(P(u)\big\vert_{E}-vF\big)dvdu=\\
&= \frac{3}{(4-\lambda s)^3}\Bigg(\int_0^{4-\lambda s}\int_0^\infty \mathrm{vol}\big(-(K_{\widetilde{\DP}^3}+\lambda \widetilde{S}-(2-\lambda m)E)|_E-uE|_E-vF\big)dvdu\Bigg)=\\
&= \frac{3}{(4-\lambda s)^3}\Bigg(\int_0^{4-\lambda s}\int_0^\infty \mathrm{vol}\big(-((K_{\widetilde{\DP}^3}+E)+\lambda \widetilde{S})|_E+(3-\lambda m)E|_E-uE|_E-vF\big)dvdu\Bigg)=\\
&= \frac{3}{(4-\lambda s)^3}\Bigg(\int_0^{4-\lambda s}\int_0^\infty \mathrm{vol}\big(-(K_E+\lambda C_m)+(3-\lambda m)E|_E-uE|_E-vF\big)dvdu\Bigg)=\\
&= \frac{3}{(4-\lambda s)^3}\Bigg(\int_0^{4-\lambda s}\int_0^\infty \mathrm{vol}\big(-uE|_E-vF\big)dvdu\Bigg)=\\
&= \frac{3}{(4-\lambda s)^3}\Bigg(\int_0^{4-\lambda s}\int_0^\infty \mathrm{vol}\Big(\frac{u}{3-\lambda m} (-K_E-\lambda C_m)-vF\Big)dvdu\Bigg)=\\
&= \frac{3}{(4-\lambda s)^3}\Bigg(\int_0^{4-\lambda s} \frac{u^3}{(3-\lambda m)^3} \int_0^\infty \mathrm{vol}\big(-(K_E+\lambda C_m)-vF\big)dvdu\Bigg)=\\
&= \frac{3}{(4-\lambda s)^3}\Bigg(\int_0^{4-\lambda s} \frac{u^3}{3-\lambda m} S_{(E,\lambda C_m)}(F)du\Bigg)= \\
& = \frac{3(4-\lambda s)}{4(3-\lambda m)} S_{(E,\lambda C_m)}(F)\le  \frac{3(4-\lambda s)}{4(3-\lambda m)}\cdot \frac{A_{(E,\lambda C_m)}(F)}{\delta_O(E,\lambda C_m)} =  \frac{3(4-\lambda s)}{4(3-\lambda m)}\cdot \frac{A_{(\DP^2,\lambda C_m)}(F)}{\delta_O(\DP^2,\lambda C_m)} 
\end{align*}}
Then $A_{(\DP^3,\lambda S)}(E)=3-\lambda m $, $S_{(\DP^3,\lambda S)}(E)=\frac{3(4-\lambda s)}{4}$ and:
 {\allowdisplaybreaks 
 \begin{align*}
\delta_P&(\DP^3,\lambda S)
\ge \min\Bigg\{\frac{A_{(\DP^3,\lambda S)}(E)}{S_{(\DP^3,\lambda S)}(E)},\text{ }\inf_{O\in E}
\delta_O(E,W^E_{\bullet,\bullet})\Bigg\}
\ge \min \Bigg\{\frac{A_{(\DP^3,\lambda S)}(E)}{S_{(\DP^3,\lambda S)}(E)},\text{ }\inf_{\substack{F/E,\\ O\in C_{\widetilde{\DP}^3}(F)}} \inf_{O\in E}
\frac{A_{(E,\lambda C_m)}(F)}{S(W^E; F)}\Bigg\}=\\
& = \min \Bigg\{\frac{A_{(\DP^3,\lambda S)}(E)}{S_{(\DP^3,\lambda S)}(E)},\text{ }\inf_{\substack{F/E,\\ O\in C_{\widetilde{\DP}^3}(F)}} \inf_{O\in E}
\frac{A_{(\DP^2,\lambda C_m)}(F)}{S(W^{\DP^2}; F)}\Bigg\}
\ge \min\Bigg\{\frac{A_{(\DP^3,\lambda S)}(E)}{S_{(\DP^3,\lambda S)}(E)},\text{ } \inf_{O\in E}
 \delta_O(\DP^2,\lambda C_m)\cdot \frac{4(3-\lambda m)}{3(4-\lambda s)}\Bigg\}=\\
& =\min\Bigg\{\frac{4(3-\lambda m)}{3(4-\lambda s)},\text{ } 
 \delta(\DP^2,\lambda C_m)\cdot \frac{4(3-\lambda m)}{3(4-\lambda s)}\Bigg\}
    \end{align*}}
\end{proof}
\noindent We reduced the problem of computing the $\delta$-invariant on the threefold to the problem of computing $\delta$-invariant on a surface. 
\begin{corollary}
        If $S_3$ has  ordinary double points ($\DA_1$-singularities) then $(\DP^3,\lambda S_3)$  is $K$-stable.
\end{corollary}
\begin{corollary}
        If $S_4$ has $\DA_n$-singularities and  ordinary triple points then $(\DP^3,\lambda S_4)$  is $K$-stable.
\end{corollary}
\begin{corollary}
        If $S_5$ has $\DA_n$-singularities and  ordinary triple points then $(\DP^3,\lambda S_5)$  is $K$-stable.
\end{corollary}
\begin{corollary}
        If $S_6$ has $\DA_n$-singularities,   ordinary triple points, or ordinary quadruple points then $(\DP^3,\lambda S_6)$  is $K$-stable.
\end{corollary}

\begin{remark}
    In fact the main theorem covers different types of singularities as well, but due to the lack of the specific name we do not include them.
\end{remark}
\subsubsection{Cubic Threefolds}
  Let  $S_3$ be a surface in $\DP^3$ defined by a cubic polynomial $f(x,y,z)=0$. Consider the equation $w^3=f(x,y,z)$.
       It defines a cubic threefold $X$ together with a triple cover $$\iota:X\xrightarrow{3:1} \DP^3$$
       ramified in $S_3$.  We have $\iota^*(K_{\DP^3}+\frac{2}{3}S_3)=K_X$. Thus, by \cite{Dervan16} $X$ is $K$-stable if and only if $(\DP^3,\frac{2}{3}S_3)$ is $K$-stable. 
      \begin{corollary}
        If $S_3$ has ordinary double points then $X$  is $K$-stable.
\end{corollary}
       \begin{remark}
          This corollary also follows from the stronger statement proved in  \cite{XuLiu}. The authors described the $K$-moduli of Cubic Threefolds by proving that they coincide with $GIT$-moduli. Then using the classification in \cite{Allcock03} they obtained the explicit list of $K$-stable cubic threefolds which included all smooth cubic threefolds, cubic threefolds with isolated $\DA_n$-singularities ($n\le 4$) and $2$ more types of cubic threefolds which can be looked up in their work \cite[Corollary 1.2]{XuLiu}.
       \end{remark}

\subsubsection{Quartic Double  Solids}
    Consider a  quartic double solid $X$ which is a double cover $\DP^3$ branched along a  quartic surface $S_4$:
$$\iota: X\xrightarrow{2:1}\DP^3$$
We have $\iota^*(K_{\DP^3}+\frac{1}{2}S_4)=K_X$. Thus, by \cite{Dervan16} $X$ is $K$-stable if and only if $(\DP^3,\frac{1}{2}S_4)$ is $K$-stable.
    \begin{corollary}
        If $S_4$ has $\DA_n$-singularities and  ordinary triple points then $X$  is $K$-stable.
    \end{corollary}
    \begin{remark}
           If $S_4$ is smooth then $X$ is a smooth Quartic Double  Solid (Fano threefold in Family 1-12). Every smooth element in this family is known to be $K$-stable  \cite{Fano21, Dervan16}. This corollary  follows from the stronger statement proved in \cite{ADVL23} where singular Del Pezzo Threefolds of degree $2$ were studied.  The authors of this paper described $K$-moduli of Quartic Double  Solids. It follows from  \cite{ADVL23, Shah81} that $X$ is $K$-polystable if and only if the quartic surface $S_4$ is $GIT$-polystable with respect to natural action $PGL(4)$  except for those of the form $(x_0x_2+x_1^2+x_3^2)^2+a\cdot x_3^4$ for $a\in\DC$.
       \end{remark}

\subsubsection{Sextic Double  Solids}
    Consider a  sextic double solid $X$ which is a double cover $\DP^3$ branched along a  sextic surface $S_6$:
$$\iota: X\xrightarrow{2:1}\DP^3$$
We have $\iota^*(K_{\DP^3}+\frac{1}{2}S_6)=K_X$. Thus,  by \cite{Dervan16} $X$ is $K$-stable if and only if $(\DP^3,\frac{1}{2}S_6)$ is $K$-stable.
     \begin{corollary}
        If $S_6$ has $\DA_n$-singularities,   ordinary triple points, or ordinary quadruple points then $X$  is $K$-stable.
\end{corollary}

 \subsubsection{Complete intersection of a quadric and a cubic in $\DP^5$} Another possible application is the following.  Let $X$ be the complete intersection of a quadric $X_2$ and a cubic $X_3$ given by a cubic polynomial $w^3=f_3(x,y,z,t,u)$.  Then $X$ is a  Fano threefold in Family 1-3 in a Mori-Mukai Classification, which is a triple cover of the quadric threefold $Y$ in $\DP^4$ branched over an anticanonical surface $S$:
 $$\iota: X\xrightarrow{3:1} Y$$
 We assume $Y$ to be smooth and $S$ to be reduced. By \cite{Dervan16} $X$ is $K$-stable if and only if $(Y,\frac{2}{3}S)$ is $K$-stable. Consider log Fano pairs $(Y,\lambda S)$.  Let $P\in S$ be the point on $S$. Consider the point blowup $\pi:\widetilde{Y}\to Y$ at this point with the exceptional divisor $E\cong\DP^2$ and let $m=\mathrm{mult}_PS$. Then
  \begin{lemma}
     $$\delta_P(X,\lambda S) \ge \min\Bigg\{\frac{3-\lambda m}{3(1-\lambda )},\text{ } \frac{4(3-\lambda m)}{15 - 9 \lambda - 2 \lambda m},\text{ }
 \delta(\DP^2,\lambda C_m)\cdot \frac{4(3-\lambda m)}{9(1-\lambda )}\Bigg\}$$
  \end{lemma}
  \begin{proof}
   Let $P(u)$ and $N(u)$ be a positive and negative part of the Zariski decomposition of the divisor $-\pi^*(K_{X}+\lambda S)-uE$. Let $H$ be a hyperplane in $Y$. Then:
  \begin{align*}
      &P(u)=
\begin{cases}
  3(1-\lambda)\pi^*(H)-uE    \text{ for }u\in [0,3(1-\lambda)],\\
  (6(1-\lambda)-u)(\pi^*(H)-E)    \text{ for }u\in [3(1-\lambda),6(1-\lambda)].
  \end{cases}\\
  &N(u)=
\begin{cases}
 0    \text{ for }u\in [0,3(1-\lambda)],\\
 (u-3(1-\lambda))(\pi^*(H)-2E)    \text{ for }u\in [3(1-\lambda),6(1-\lambda)].
  \end{cases}
  \end{align*}
  We have: 
$$\big(P(u)\big)^2=
\begin{cases}
   54 - u^3 - 162\lambda + 162\lambda^2 -54\lambda^3 \text{ for }u\in [0,3(1-\lambda)],\\
  (6(1-\lambda)-u)^3    \text{ for }u\in [3(1-\lambda),6(1-\lambda)].
  \end{cases}$$
  Thus
  $$S_{(Y,\lambda S)}(E)= \frac{1}{54(1 - \lambda)^3}\Bigg(\int_0^{3(1-\lambda)}\big(54 - u^3 - 162\lambda + 162\lambda^2 -54\lambda^3\big)du+\int_{3(1-\lambda)}^{6(1-\lambda)}  (6(1-\lambda)-u)^3   du\Bigg)=3 - 3\lambda.$$
  Moreover, for any $O\in E$ and for any divisor $F$ such that $O\in C_{E}(F)$    we get
{\allowdisplaybreaks \begin{align*}
S&\big(W^{E}_{\bullet,\bullet};F\big)=\\
&=\frac{3}{(-K_{Y}-\lambda S)^3}\Bigg(\int_0^{6(1-\lambda)}\big(P(u)^{2}\cdot E\big)\cdot\mathrm{ord}_{O}\Big(N(u)\big\vert_{E}\Big)du+\int_0^{6(1-\lambda)}\int_0^\infty \mathrm{vol}\big(P(u)\big\vert_{E}-vF\big)dvdu\Bigg)\le\\
&\le \frac{3}{54(1 - \lambda)^3}\Bigg( 
\int_{3(1-\lambda)}^{6(1-\lambda)}\big(P(u)^{2}\cdot E\big)du + \int_0^\tau\int_0^\infty \mathrm{vol}\big(P(u)\big\vert_{E}-vF\big)dvdu\Bigg)=\\
&= \frac{3}{54(1 - \lambda)^3}\Bigg( 9(1-\lambda)^3 +\int_0^{3(1-\lambda)}\int_0^\infty \mathrm{vol}\big(P(u)\big\vert_{E}-vF\big)dvdu+\int_{3(1-\lambda)}^{6(1-\lambda)}\int_0^\infty \mathrm{vol}\big(P(u)\big\vert_{E}-vF\big)dvdu\Bigg)=\\
&= \frac{3}{54(1 - \lambda)^3}\Bigg(9(1-\lambda)^3 + \int_0^{3(1-\lambda)}\int_0^\infty \mathrm{vol}\big(-(K_{\widetilde{Y}}+\lambda \widetilde{S}-(2-\lambda m)E)|_E-uE|_E-vF\big)dvdu+\\
&  \text{$ $ $ $ $ $ $ $ $ $ $ $ $ $ $ $ $ $ $ $ $ $  }\text{$ $ $ $ $ $ $ $ $ $ $ $ $ $ $ $ $ $ $ $ $ $  }\text{$ $ $ $ $ $ $ $ $ $ $ $ $ $ $ $ $ $ $ $ $ $  } + \int_{3(1-\lambda)}^{6(1-\lambda)}\int_0^\infty \mathrm{vol}\big((6(1-\lambda)-u)(\pi^*(H)-E))|_E-vF\big)dvdu\Bigg)=\\
&= \frac{3}{54(1 - \lambda)^3}\Bigg(9(1-\lambda)^3 + \int_0^{3(1 - \lambda)}\int_0^\infty \mathrm{vol}\big(-((K_{\widetilde{Y}}+E)+\lambda \widetilde{S})|_E+(3-\lambda m)E|_E-uE|_E-vF\big)dvdu+\\
&  \text{$ $ $ $ $ $ $ $ $ $ $ $ $ $ $ $ $ $ $ $ $ $  }\text{$ $ $ $ $ $ $ $ $ $ $ $ $ $ $ $ $ $ $ $ $ $  }\text{$ $ $ $ $ $ $ $ $ $ $ $ $ $ $ $ $ $ $ $ $ $  } + \int_{3(1-\lambda)}^{6(1-\lambda)}\int_0^\infty \mathrm{vol}\big(-(6(1-\lambda)-u)\Big(\frac{1}{3}\pi^*(K_Y)+E\Big)|_E-vF\big)dvdu\Bigg)=\\
&= \frac{3}{54(1 - \lambda)^3}\Bigg(9(1-\lambda)^3 + \int_0^{3(1-\lambda)}\int_0^\infty \mathrm{vol}\big(-(K_E+\lambda C_m)+(3-\lambda m)E|_E-uE|_E-vF\big)dvdu+\\
& \text{$ $ $ $ $ $ $ $ $ $ $ $ $ $ $ $ $ $ $ $ $ $  }\text{$ $ $ $ $ $ $ $ $ $ $ $ $ $ $ $ $ $ $ $ $ $  }\text{$ $ $ $ $ $ $ $ $ $ $ $ $ $ $ $ $ $ $ $ $ $  } + \int_{3(1-\lambda)}^{6(1-\lambda)}\int_0^\infty \mathrm{vol}\big(-(6(1-\lambda)-u)\Big(\frac{1}{3}K_{\widetilde{Y}}+\frac{1}{3}E\Big)|_E-vF\big)dvdu\Bigg)=\\
&= \frac{3}{54(1 - \lambda)^3}\Bigg(9(1-\lambda)^3 + \int_0^{3(1-\lambda)}\int_0^\infty \mathrm{vol}\big(-(K_E+\lambda C_m)+(3-\lambda m)E|_E-uE|_E-vF\big)dvdu+\\
& \text{$ $ $ $ $ $ $ $ $ $ $ $ $ $ $ $ $ $ $ $ $ $  }\text{$ $ $ $ $ $ $ $ $ $ $ $ $ $ $ $ $ $ $ $ $ $  }\text{$ $ $ $ $ $ $ $ $ $ $ $ $ $ $ $ $ $ $ $ $ $  } + \int_{3(1-\lambda)}^{6(1-\lambda)}\int_0^\infty \mathrm{vol}\big(-\frac{1}{3}(6(1-\lambda)-u)K_E-uE|_E-vF\big)dvdu\Bigg)=\\
&= \frac{3}{54(1 - \lambda)^3}\Bigg(9(1-\lambda)^3 + \int_0^{3(1-\lambda)}\int_0^\infty \mathrm{vol}\big(-uE|_E-vF\big)dvdu+\\
&  \text{$ $ $ $ $ $ $ $ $ $ $ $ $ $ $ $ $ $ $ $ $ $  }\text{$ $ $ $ $ $ $ $ $ $ $ $ $ $ $ $ $ $ $ $ $ $  }\text{$ $ $ $ $ $ $ $ $ $ $ $ $ $ $ $ $ $ $ $ $ $  } + \int_{3(1-\lambda)}^{6(1-\lambda)}\int_0^\infty \mathrm{vol}\big( -(6(1-\lambda)-u)E|_E-vF\big)dvdu\Bigg)=\\
&= \frac{3}{54(1 - \lambda)^3}\Bigg(
9(1-\lambda)^3 +
\int_0^{3(1-\lambda)}\int_0^\infty \mathrm{vol}\Big(\frac{u}{3-\lambda m} (-K_E-\lambda C_m)-vF\Big)dvdu+\\
&\text{$ $ $ $ $ $ $ $ $ $ $ $ $ $ $ $ $ $ $ $ $ $  }\text{$ $ $ $ $ $ $ $ $ $ $ $ $ $ $ $ $ $ $ $ $ $  }\text{$ $ $ $ $ $ $ $ $ $ $ $ $ $ $ $ $ $ $ $ $ $  }
+ \int_{3(1-\lambda)}^{6(1-\lambda)}\int_0^\infty \mathrm{vol}\Big(\frac{(6(1-\lambda)-u)}{3-\lambda m} (-K_E-\lambda C_m)-vF\Big)dvdu\Bigg)=\\
&= \frac{3}{54(1 - \lambda)^3}\Bigg(
9(1-\lambda)^3 + \int_0^{3(1-\lambda)} \frac{u^3}{(3-\lambda m)^3} \int_0^\infty \mathrm{vol}\big(-(K_E+\lambda C_m)-vF\big)dvdu+\\
&\text{$ $ $ $ $ $ $ $ $ $ $ $ $ $ $ $ $ $ $ $ $ $  }\text{$ $ $ $ $ $ $ $ $ $ $ $ $ $ $ $ $ $ $ $ $ $  }\text{$ $ $ $ $ $ $ $ $ $ $ $ $ $ $ $ $ $ $ $ $ $  }+\int_{3(1-\lambda)}^{6(1-\lambda)} \frac{(6(1-\lambda)-u)^3}{(3-\lambda m)^3} \int_0^\infty \mathrm{vol}\big(-(K_E+\lambda C_m)-vF\big)dvdu\Bigg)=\\
&= \frac{3}{54(1 - \lambda)^3}
\Bigg(9(1-\lambda)^3 + \int_0^{3(1-\lambda)} \frac{u^3}{3-\lambda m} S_{(E,\lambda C_m)}(E)du + \int_{3(1-\lambda)}^{6(1-\lambda)} \frac{(6(1-\lambda)-u)^3}{3-\lambda m} S_{(E,\lambda C_m)}(E)du \Bigg)= \\
&= \frac{3}{54(1 - \lambda)^3}
\Bigg(9(1-\lambda)^3 + \frac{3^4(1-\lambda)^4}{2(3-\lambda m)} S_{(E,\lambda C_m)}(F)\Bigg)= \frac{1}{2}+ \frac{9(1-\lambda)}{4(3-\lambda m)}S_{(E,\lambda C_m)}(F)\le \\
&   \le  \frac{1}{2}+ \frac{9(1-\lambda)}{4(3-\lambda m)}\cdot \frac{A_{(E,\lambda C_m)}(F)}{\delta_O(E,\lambda C_m)} 
\end{align*}}
\noindent Note that $C:=E\cap  (\pi^*(H)-2E) $ is a smooth curve which does not intersect $C_m$ and $\delta_O(E,\lambda C_m)=1$ for $O\in C$. Thus since in this case we have 
$$\frac{A_{(E,\lambda C_m)}(F)}{2}+ \frac{9(1-\lambda)}{4(3-\lambda m)}\cdot \frac{A_{(E,\lambda C_m)}(F)}{\delta_O(E,\lambda C_m)} = \frac{15 - 9 \lambda - 2 \lambda m}{4(3-\lambda m)} A_{(E,\lambda C_m)}(F) $$
Then
$$
S\big(W^{E}_{\bullet,\bullet};F\big)
\le
\begin{cases}
  \frac{9(1-\lambda)}{4(3-\lambda m)} \frac{A_{(E,\lambda C_m)}(F)}{\delta_O(E,\lambda C_m)} \text{ for }O\in E\backslash C,\\
  \frac{15 - 9 \lambda - 2 \lambda m}{4(3-\lambda m)}A_{(E,\lambda C_m)}(F) \text{ for }O\in C.
\end{cases}
$$
Then $A_{(Y,\lambda S)}(E)=3-m \lambda$, $S_{(Y,\lambda S)}(E)=3-3 \lambda $ then:
 {\allowdisplaybreaks 
 \begin{align*}
\delta_P(Y,\lambda S)&
\ge \min\Bigg\{\frac{A_{(Y,\lambda S)}(E)}{S_{(Y,\lambda S)}(E)},\text{ }\inf_{O\in E}
\delta_O(E,W^E_{\bullet,\bullet})\Bigg\}
\ge \min \Bigg\{\frac{A_{(Y,\lambda S)}(E)}{S_{(Y,\lambda S)}(E)},\text{ }\inf_{\substack{F/E,\\ O\in C_{\widetilde{Y}}(F)}} \inf_{O\in E}
\frac{A_{(E,\lambda C_m)}(F)}{S(W^E; F)}\Bigg\}=\\
& = \min \Bigg\{\frac{A_{(Y,\lambda S)}(E)}{S_{(Y,\lambda S)}(E)},\text{ }\inf_{\substack{F/E,\\ O\in C_{\widetilde{Y}}(F)}} \inf_{O\in E}
\frac{A_{(\DP^2,\lambda C_m)}(F)}{S(W^{\DP^2}; F)}\Bigg\}\ge\\
&\ge \min\Bigg\{\frac{A_{(Y,\lambda S)}(E)}{S_{(Y,\lambda S)}(E)},\text{ } \frac{4(3-\lambda m)}{15 - 9 \lambda - 2 \lambda m},\text{ } \inf_{O\in E}
 \delta_O(\DP^2,\lambda C_m)\cdot \frac{4(3-\lambda m)}{9(1-\lambda )}\Bigg\}=\\
& =\min\Bigg\{\frac{3- \lambda m}{3(1-\lambda )},\text{ } \frac{4(3-\lambda m)}{15 - 9 \lambda - 2 \lambda m},\text{ }
 \delta(\DP^2,\lambda C_m)\cdot \frac{4(3-\lambda m)}{9(1-\lambda )}\Bigg\}
    \end{align*}}
\end{proof}
\noindent We reduced the problem of computing the $\delta$-invariant on the threefold to the problem of computing $\delta$-invariant on a surface. Note that we are interested in case when $\frac{(3-\lambda m)}{3(1-\lambda)}\ge 1$ and $ \frac{4(3-\lambda m)}{15 - 9 \lambda - 2 \lambda m}\ge 1$.
For $\lambda=\frac{2}{3}$ we obtain
$$\delta_P(Y,\lambda S) \ge \min\Bigg\{3-\frac{2m}{3},\text{ } 1+\frac{9 - 4 m}{27 - 4 m},\text{ }
 \delta(\DP^2,\lambda C_m)\cdot 4\big(1-\frac{2m}{9}\big)\Bigg\}$$
 Which proves the following result: 
 \begin{corollary}
        If $S$ has at most ordinary double points then $\delta_P(Y,\frac{2}{3}S)>1$ for $P\in S$.
    \end{corollary}

\section{Proof of Main Theorem via  Kento Fujita’s formulas}
\noindent Let $C\subset X=\DP^2$ be a plane curve of degree $d$, $\lambda\in \big(0,\min\big\{\frac{3}{d},\mathrm{lct}(\DP^2,C)\big\}\big)$.  Let  $f\colon\widetilde{X}\to X$ be a birational morphism,
let $E$ be a prime divisor in $\widetilde{X}$.
We say that $E$ is a prime divisor \emph{over} $X$.
If~$E$~is~\mbox{$f$-exceptional}, we say that $E$ is an~exceptional invariant prime divisor \emph{over}~$X$.
We will denote the~subvariety $f(E)$ by $C_X(E)$. 
Let \index{$S_{(X,\lambda C)}(E)$}
$$
S_{(X,\lambda C)}(E)=\frac{1}{(-K_X-\lambda C)^2}\int_{0}^{\tau}\mathrm{vol}(f^*(-K_X-\lambda C)-vE)dv,
$$
where $\tau=\tau_{(X,\lambda C)}(E)$ is the~pseudo-effective threshold of $E$ with respect to $-K_X-\lambda C$, i.e. we have \index{pseudo-effective threshold}
$$
\tau_{(X,\lambda C)}(E)=\mathrm{sup}\Big\{ v \in \mathbb{Q}_{>0}\ \big\vert \ f^*(-K_X-\lambda C)-vE\ \text{is pseudo-effective}\Big\}.
$$
Let 
$$A_{(X,\lambda C)} (E) = 1 + \mathrm{ord}_E(K_{\widetilde{X}} - f^*(K_X+\lambda C)).$$
\index{$\beta$-invariant}
Let $P$ be a point in $X$. We can define a local $\delta$-invariant now
$$
\delta_P(X,\lambda C)=\inf_{\substack{E/X\\ P\in C_X(E)}}\frac{A_{(X,\lambda C)}(E)}{S_{(X,\lambda C)}(E)},
$$
where the~infimum runs over all prime divisors $E$ over the surface $X$ such that $P\in C_X(E)$. Moreover a global $\delta$-invariant is given by:
$$\delta(X,\lambda C)=\inf_{P\in X}\delta_P(X,\lambda C)$$
Several results can help us to estimate $\delta$-invariants.
Let $\mathcal{C}$ be a smooth curve on $X$ containing $P$. 
Set
$$
\tau(\mathcal{C})=\mathrm{sup}\Big\{v\in\mathbb{R}_{\geqslant 0}\ \big\vert\ \text{the divisor  $-K_X-\lambda C-v\mathcal{C}$ is pseudo-effective}\Big\}.
$$
For~$v\in[0,\tau]$, let $P(v)$ be the~positive part of the~Zariski decomposition of the~divisor $-K_X-\lambda C-v\mathcal{C}$,
and let $N(v)$ be its negative part. 
Then we set $$
S\big(W^{\mathcal{C}}_{\bullet,\bullet};P\big)=\frac{2}{(-K_X-\lambda C)^2}\int_0^{\tau(\mathcal{C})} h(v) dv,
\text{ where }
h(v)=\big(P(v)\cdot \mathcal{C}\big)\times\big(N(v)\cdot \mathcal{C}\big)_P+\frac{\big(P(v)\cdot \mathcal{C}\big)^2}{2}.
$$
It follows from {\cite[Theorem 1.7.1]{Fano21}} that:
\begin{equation}\label{estimation1}
    \delta_P(S)\geqslant\mathrm{min}\Bigg\{\frac{1}{S_{(X,\lambda C)}(\mathcal{C})},\frac{1}{S(W_{\bullet,\bullet}^\mathcal{C},P)}\Bigg\}.
\end{equation}
Unfortunately, using this approach we do not always get a good estimation. In this case, we can try to apply the generalization of this method. Let $\pi\colon\widehat{X}\to X$ be a weighted blowup of the point $P$. Suppose, in addition, that $\widehat{X}$ is a Mori Dream space Then
\begin{itemize}
\item the~$\pi$-exceptional curve $E_P$ such that $\pi(E_P)=P$ is smooth and isomorphic to $\DP^1$,
\item the~log pair $(\widehat{X},E_P)$ has purely log terminal singularities.
\end{itemize}
Thus, the birational map $\pi$ a plt blowup of a point $P$.
Write
$$
K_{E_P}+\Delta_{E_P}=\big(K_{\widehat{X}}+\lambda \widehat{C}+E_P\big)\big\vert_{E_P},
$$
where $\Delta_{E_P}$ is an~effective $\mathbb{Q}$-divisor on $E_P$ known as the~different of the~log pair $(\widehat{X},E_P)$.
Note that the~log pair $(E_P,\Delta_{E_P})$ has at most Kawamata log terminal singularities, and the~divisor $-(K_{E_P}+\Delta_{E_P})$ is $\pi\vert_{E_P}$-ample.
\\Let $O$ be a point on $E_P$. 
Set
$$
\tau(E_P)=\mathrm{sup}\Big\{v\in\mathbb{R}_{\geqslant 0}\ \big\vert\ \text{the divisor  $\pi^*(-K_X-\lambda C)-vE_P$ is pseudo-effective}\Big\}.
$$
For~$v\in[0,\tau]$, let $\widehat{P}(v)$ be the~positive part of the~Zariski decomposition of the~divisor $\pi^*(-K_X-\lambda C)-vE_P$,
and let $\widehat{N}(v)$ be its negative part. 
Then we set $$
S\big(W^{E_P}_{\bullet,\bullet};O\big)=\frac{2}{K_{\widehat{X}}^2}\int_0^{\tau(E_P)} \widehat{h}(v) dv,
\text{ where }
\widehat{h}(v)=\big(\widehat{P}(v)\cdot E_P\big)\times\big(\widehat{N}(v)\cdot E_P\big)_O+\frac{\big(\widehat{P}(v)\cdot E_P\big)^2}{2}.
$$
Let
$$
A_{E_P,\Delta_{E_P}}(O)=1-\mathrm{ord}_{\Delta_{E_P}}(O),
$$
It follows from {\cite[Theorem 1.7.9]{Fano21}} and {\cite[Corollary 1.7.12]{Fano21}} that
\begin{equation}
\label{estimation2}
\delta_P(S)\geqslant\mathrm{min}\Bigg\{\frac{A_S(E_P)}{S_S(E_P)},\inf_{O\in E_P}\frac{A_{E_P,\Delta_{E_P}}(O)}{S\big(W^{E_P}_{\bullet,\bullet};O\big)}\Bigg\},
\end{equation}
where the~infimum is taken over all points $O\in E_P$.
 \section{Line}
  \begin{lemma} \label{line}
 Let $C\subset \DP^2$ be a reduced plane curve of degree $d$, and  $P\in C$ be a smooth point on a line $L$ which is a component of $C$. Then
 $$\delta_P(\DP^2,\lambda C)=\frac{3(1-\lambda)}{3- d \lambda }\text{ for }\lambda\in\Big[0,\min\Big\{1,\frac{3}{d}\Big\}\Big].$$
 \end{lemma}
 \begin{proof}
Note that $\delta_P(\DP^2,\lambda C)\le \frac{A_{(\DP^2,\lambda C)}(L)}{S_{(\DP^2,\lambda C)}(L)}=\frac{3(1-\lambda)}{3- d \lambda }$. Let $\pi:\widetilde{S}\to S$ be the~blow up of the~point $P$,
with the exceptional divisor $E$.
\begin{center}
 \includegraphics[width=6cm]{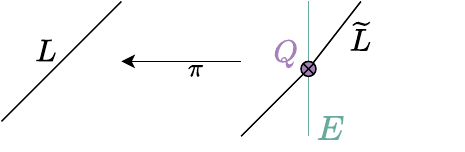}
 \end{center}
The Zariski decomposition of the divisor  $-\pi^*(K_{\DP^2}+\lambda C)-vE$ is given by:
\begin{align*}
P(v)=-\pi^*(K_{\DP^2}+\lambda C)-vE \text{ and }N(v)=0\text{ if }v\in[0,3- d \lambda ].
\end{align*}
Then
$$
P(v)^2=(d\lambda - 3- v )(v-(3- d\lambda))\text{ and }P(v)\cdot E=  v \text{ if }v\in[0,3- d \lambda ].
$$
For every $O\in L$ we get:
$$h(v) = \frac{v^2}{2}\text{ if }v\in[0,3-d \lambda ].
$$
So that
$$S\big(W^{E}_{\bullet,\bullet};O\big)= \frac{2}{(3-d \lambda )^2}\Big(\int_0^{3-d \lambda } \frac{v^2}{2} dv\Big)=\frac{3-d \lambda}{3}
$$
We have
$$
\delta_P(\DP^2,\lambda C)\geqslant\mathrm{min}\Bigg\{ \frac{3}{2}\cdot \frac{2-\lambda}{3-d \lambda } ,\inf_{O\in E}\frac{A_{E,\Delta_{E}}(O)}{S\big(W^{E}_{\bullet,\bullet};O\big)}\Bigg\},
$$
where $\Delta_{E}=\lambda Q$ where $Q=\widetilde{L}|_E$. 
So that
$$
\frac{A_{\overline{E},\Delta_{\overline{E}}}(O)}{S(W_{\bullet,\bullet}^{\overline{E}};O)}=
\left\{\aligned
&\frac{3(1-\lambda)}{3-d \lambda }\ \mathrm{if}\ O=Q,\\
&\frac{3}{3-d \lambda }\ \mathrm{otherwise}.
\endaligned
\right.
$$
Thus $\delta_P(\DP^2,\lambda C)= \frac{3(1-\lambda)}{3- d \lambda }$.     
 \end{proof}
 \section{Smooth conic}
 \begin{lemma}
 Let $C\subset \DP^2$ be a conic, $P\in C$ be a smooth point on $C$ such that the tangent line $L$ at this point. Then
 $$\delta_P(\DP^2,\lambda C)=1\text{ for }\lambda\in\Big[0,\frac{3}{4}\Big].$$
 \end{lemma}
 \begin{proof}
Let $\pi_1:S^1\to S$ be the~blow up of the~point $P$,
with the exceptional divisor $E_1^1$ and $L^1$, $C^1$ are strict transforms of $L$ and $C$ respectively; $\pi_2: S^2\to S^1$ be the~blow up of the~point $E_1^1\cap C^1\cap L^1$,
with the exceptional divisor $E$ and $L^2$, $C^2$ are strict transforms of $L^1$ and $C^1$ respectively; let $\theta:S^2\to\overline{S}$ be  the~contraction of  the~curve $E_1^2$,
and $\sigma$ is the~birational contraction of $\overline{E}=\theta(E)$. We denote the strict transforms of $C$ and $L$ on $\overline{S}$ by $\overline{C}$ and $\overline{L}$. Let $\overline{E}$ be the exceptional divisor of $\sigma$. Note that $\overline{E}$ contains one singular point $\overline{P}$  which is $\frac{1}{2}(1,1)$  singular point.
\begin{center}
 \includegraphics[width=15cm]{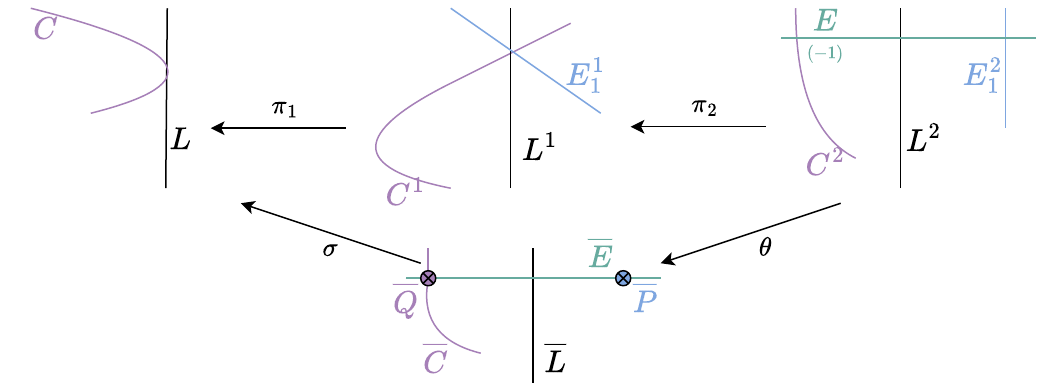}
 \end{center}
The intersections on $\overline{S}$ are given by:
\begin{center}
    \renewcommand{\arraystretch}{1.4}
    \begin{tabular}{|c|c|c|}
    \hline
         & $\overline{E}$ & $\overline{L}$ \\
    \hline
       $\overline{E}$  & $-\frac{1}{2}$ & $1$ \\
    \hline
       $\overline{L}$  & $1$ & $-1$ \\
    \hline
    \end{tabular}
\end{center}
 We have $\sigma^*(L)=\overline{L}+2\overline{E}$, $\sigma^*(C)=\overline{C}+2\overline{E}$, $\sigma^*(K_{\DP^2})=K_{\overline{S}}-2\overline{E}$. Thus, $A_{(\DP^2,\lambda C)}(\overline{E})=3-2\lambda$.
\noindent The Zariski decomposition of the divisor  $-\sigma^*(K_{\DP^2}+\lambda C)-v\overline{E}$ is given by:
\begin{align*}
&&P(v)=
\begin{cases}
-\sigma^*(K_{\DP^2}+\lambda C)-v\overline{E}\text{ if }v\in[0,3-2\lambda],\\
-\sigma^*(K_{\DP^2}+\lambda C)-v\overline{E}-\big(v-(3-2\lambda)\big)\overline{L}\text{ if }v\in[3-2\lambda, 2(3-2\lambda)],\\
\end{cases}\\&&
N(v)=
\begin{cases}
0\text{ if }v\in[0,3-2\lambda],\\
\big(v-(3-2\lambda)\big)\overline{L}\text{ if }v\in[3-2\lambda,2(3-2\lambda)].
\end{cases}
\end{align*}
Then
$$
P(v)^2=
\begin{cases}
(3-2\lambda)^2 -\frac{v^2}{2}\text{ if }v\in[0,3-2\lambda],\\
\frac{(v- 6 + 4\lambda )^2}{2}\text{ if }v\in[3-2\lambda,2(3-2\lambda)],
\end{cases}
\text{ and }
P(v)\cdot \overline{E}=
\begin{cases}
\frac{v}{2}\text{ if }v\in[0,3-2\lambda],\\
3-2\lambda -\frac{v}{2}\text{ if }v\in[3-2\lambda,2(3-2\lambda)],
\end{cases}
$$
Thus
$$S_S(\overline{E})=\frac{1}{(3-2\lambda)^2}\Big(\int_0^{3-2\lambda} (3-2\lambda)^2 -\frac{v^2}{2} dv+\int_{3-2\lambda}^{2(3-2\lambda)} \frac{(v- 6 + 4\lambda)^2}{2} dv\Big)=3-2\lambda$$
so that $\delta_P(\DP^2,\lambda C)\le \frac{3-2\lambda}{3-2\lambda}=1$. For every $O\in \overline{E}$, we get if $O\in \overline{E}\backslash \overline{L}$ or if $O\in \overline{E}\cap \overline{L}$:
$$h(v)\le 
\begin{cases}
\frac{v^2}{8}\text{ if }v\in[0,3-2\lambda],\\
\frac{(v- 6 + 4\lambda)^2}{8}\text{ if }v\in[3-2\lambda, 2(3-2\lambda)],
\end{cases}
\text{ or }
h(v)\le \begin{cases}
\frac{v^2}{8}\text{ if }v\in[0,3-2\lambda],\\
\frac{ (v- 6 + 4\lambda) (6 - 4\lambda - 3v)}{8}\text{ if }v\in[3-2\lambda, 2(3-2\lambda)].
\end{cases}
$$
So that
$$S\big(W^{\overline{E}}_{\bullet,\bullet};O\big)\le \frac{2}{(3-2\lambda)^2}\Big(\int_0^{3-2\lambda} \frac{v^2}{8} dv+\int_{3-2\lambda}^{2(3-2\lambda)} \frac{(v- 6 + 6\lambda)^2}{8} dv\Big)=\frac{3-2\lambda}{6}
$$
or
$$S\big(W^{\overline{E}}_{\bullet,\bullet};O\big)\le \frac{2}{(3-2\lambda)^2}\Big(\int_0^{3-2\lambda} \frac{v^2}{8} dv+\int_{3-2\lambda}^{2(3-2\lambda)} \frac{ (v- 6 + 4\lambda) (6 - 4\lambda - 3v)}{8} dv\Big)=\frac{3-2\lambda}{3}
$$
We have
$$
\delta_P(\DP^2,\lambda C)\geqslant\mathrm{min}\Bigg\{1,\inf_{O\in\overline{E}}\frac{A_{\overline{E},\Delta_{\overline{E}}}(O)}{S\big(W^{\overline{E}}_{\bullet,\bullet};O\big)}\Bigg\},
$$
where $\Delta_{\overline{E}}=\frac{1}{2}\overline{P}+\lambda \overline{Q}$, where $\overline{Q}=\overline{E}\cap \overline{C}$. 
So that
$$
\frac{A_{\overline{E},\Delta_{\overline{E}}}(O)}{S(W_{\bullet,\bullet}^{\overline{E}};O)}=
\left\{\aligned
&\frac{3}{3-2\lambda}\ \mathrm{if}\ O=\overline{E}\cap\overline{L} \text{ or if } O=\overline{P},\\
&\frac{6(1-\lambda)}{3-2\lambda}\ \mathrm{if}\ O=\overline{Q},\\
&\frac{6}{3-2\lambda}\ \mathrm{otherwise}.
\endaligned
\right.
$$
Thus $\delta_P(\DP^2,\lambda C)=1$ for $\lambda\in\big[0,\frac{3}{4}\big]$.
 \end{proof}
\section{Smooth cubic curves}
\begin{lemma}
 Let $C\subset \DP^2$ be a cubic curve, $P\in C$ be a smooth point on $C$ such that the tangent line $L$ at this point has multiplicity $2$. Then
 $$\delta_P(\DP^2,\lambda C)=\frac{3-2\lambda}{3-3 \lambda}\text{ for }\lambda\in\Big[0,1\Big].$$
 \end{lemma}
 \begin{proof}
Let $\pi_1:S^1\to S$ be the~blow up of the~point $P$,
with the exceptional divisor $E_1^1$ and $L^1$, $C^1$ are strict transforms of $L$ and $C$ respectively; $\pi_2: S^2\to S^1$ be the~blow up of the~point $E_1^1\cap C^1\cap L^1$,
with the exceptional divisor $E$ and $L^2$, $C^2$ are strict transforms of $L^1$ and $C^1$ respectively; let $\theta:S^2\to\overline{S}$ be  the~contraction of  the~curve $E_1^2$,
and $\sigma$ is the~birational contraction of $\overline{E}=\theta(E)$. We denote the strict transforms of $C$ and $L$ on $\overline{S}$ by $\overline{C}$ and $\overline{L}$. Let $\overline{E}$ be the exceptional divisor of $\sigma$. Note that $\overline{E}$ contains one singular point $\overline{P}$  which is $\frac{1}{2}(1,1)$  singular point.
\begin{center}
 \includegraphics[width=15cm]{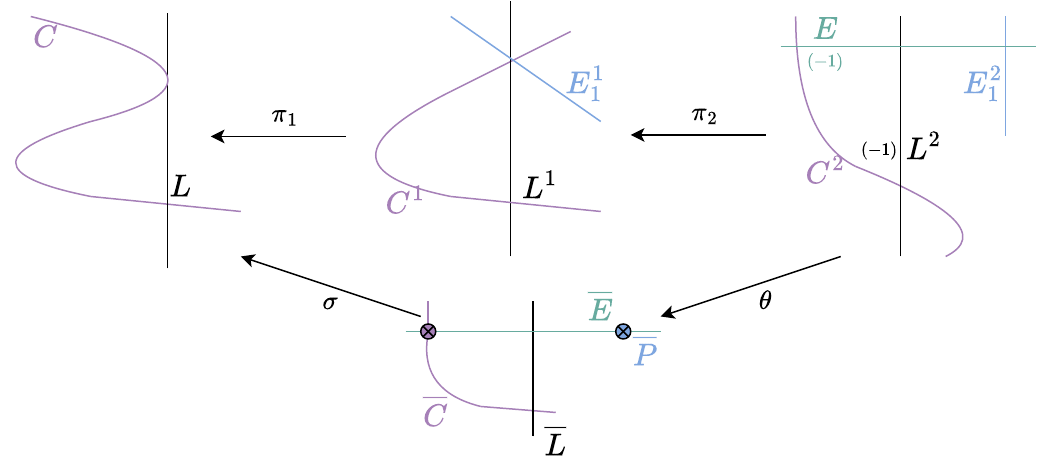}
 \end{center}
The intersections on $\overline{S}$ are given by:
\begin{center}
    \renewcommand{\arraystretch}{1.4}
    \begin{tabular}{|c|c|c|}
    \hline
         & $\overline{E}$ & $\overline{L}$ \\
    \hline
       $\overline{E}$  & $-\frac{1}{2}$ & $1$ \\
    \hline
       $\overline{L}$  & $1$ & $-1$ \\
    \hline
    \end{tabular}
\end{center}
 We have $\sigma^*(L)=\overline{L}+2\overline{E}$, $\sigma^*(C)=\overline{C}+2\overline{E}$, $\sigma^*(K_{\DP^2})=K_{\overline{S}}-2\overline{E}$. Thus, $A_{(\DP^2,\lambda C)}(\overline{E})=3-2\lambda$.
\noindent The Zariski decomposition of the divisor  $-\sigma^*(K_{\DP^2}+\lambda C)-v\overline{E}$ is given by:
\begin{align*}
&&P(v)=
\begin{cases}
-\sigma^*(K_{\DP^2}+\lambda C)-v\overline{E}\text{ if }v\in[0,3-3\lambda],\\
-\sigma^*(K_{\DP^2}+\lambda C)-v\overline{E}-\big(v-(3-3\lambda)\big)\overline{L}\text{ if }v\in[3-3\lambda, 2(3-3\lambda)],\\
\end{cases}\\&&
N(v)=
\begin{cases}
0\text{ if }v\in[0,3-3\lambda],\\
\big(v-(3-3\lambda)\big)\overline{L}\text{ if }v\in[3-3\lambda,2(3-3\lambda)].
\end{cases}
\end{align*}
Then
$$
P(v)^2=
\begin{cases}
(3-3\lambda)^2 -\frac{v^2}{2}\text{ if }v\in[0,3-3\lambda],\\
\frac{(v- 6 + 6\lambda )^2}{2}\text{ if }v\in[3-3\lambda,2(3-3\lambda)],
\end{cases}
\text{ and }
P(v)\cdot \overline{E}=
\begin{cases}
\frac{v}{2}\text{ if }v\in[0,3-3\lambda],\\
3-3\lambda -\frac{v}{2}\text{ if }v\in[3-3\lambda,2(3-3\lambda)],
\end{cases}
$$
Thus
$$S_S(\overline{E})=\frac{1}{(3-3\lambda)^2}\Big(\int_0^{3-3\lambda} (3-3\lambda)^2 -\frac{v^2}{2} dv+\int_{3-3\lambda}^{2(3-3\lambda)} \frac{(v- 6 + 6\lambda)^2}{2} dv\Big)=3-3\lambda$$
so that $\delta_P(\DP^2,\lambda C)\le \frac{3-2\lambda}{3-3\lambda}$. For every $O\in \overline{E}$, we get if $O\in \overline{E}\backslash \overline{L}$ or if $O\in \overline{E}\cap \overline{L}$:
$$h(v)\le 
\begin{cases}
\frac{v^2}{8}\text{ if }v\in[0,3-3\lambda],\\
\frac{(v- 6 + 6\lambda)^2}{8}\text{ if }v\in[3-3\lambda, 2(3-3\lambda)],
\end{cases}
\text{ or }
h(v)\le \begin{cases}
\frac{v^2}{8}\text{ if }v\in[0,3-3\lambda],\\
\frac{ (v- 6 + 6\lambda) (6 - 6\lambda - 3v)}{8}\text{ if }v\in[3-3\lambda, 2(3-3\lambda)].
\end{cases}
$$
So that
$$S\big(W^{\overline{E}}_{\bullet,\bullet};O\big)\le \frac{2}{(3-3\lambda)^2}\Big(\int_0^{3-3\lambda} \frac{v^2}{8} dv+\int_{3-3\lambda}^{2(3-3\lambda)} \frac{(v- 6 + 6\lambda)^2}{8} dv\Big)=\frac{3-3\lambda}{6}
$$
or
$$S\big(W^{\overline{E}}_{\bullet,\bullet};O\big)\le \frac{2}{(3-3\lambda)^2}\Big(\int_0^{3-3\lambda} \frac{v^2}{8} dv+\int_{3-3\lambda}^{2(3-3\lambda)} \frac{ (v- 6 + 6\lambda) (6 - 6\lambda - 3v)}{8} dv\Big)=\frac{3-3\lambda}{3}
$$
We have
$$
\delta_P(\DP^2,\lambda C)\geqslant\mathrm{min}\Bigg\{\frac{3-2\lambda}{3-3\lambda},\inf_{O\in\overline{E}}\frac{A_{\overline{E},\Delta_{\overline{E}}}(O)}{S\big(W^{\overline{E}}_{\bullet,\bullet};O\big)}\Bigg\},
$$
where $\Delta_{\overline{E}}=\frac{1}{2}\overline{P}+\lambda \overline{Q}$, where $\overline{Q}=\overline{E}\cap \overline{C}$. 
So that
$$
\frac{A_{\overline{E},\Delta_{\overline{E}}}(O)}{S(W_{\bullet,\bullet}^{\overline{E}};O)}=
\left\{\aligned
&\frac{3}{3-3\lambda}\ \mathrm{if}\ O=\overline{E}\cap\overline{L} \text{ or if } O=\overline{P},\\
&\frac{6(1-\lambda)}{3-3\lambda}\ \mathrm{if}\ O=\overline{Q},\\
&\frac{6}{3-3\lambda}\ \mathrm{otherwise}.
\endaligned
\right.
$$
Thus $\delta_P(\DP^2,\lambda C)=\frac{3-2\lambda}{3-3\lambda}$.
 \end{proof}

 \begin{lemma}
 Let $C\subset \DP^2$ be a cubic curve, $P\in C$ be a smooth point on $C$ such that the tangent line $L$ at this point has multiplicity $3$. Then
 $$\delta_P(\DP^2,\lambda C)=\frac{4-3\lambda}{4-4\lambda}\text{ for }\lambda\in\Big[0,\frac{3}{4}\Big].$$
 \end{lemma}
 \begin{proof}
Let $\pi_1:S^1\to S$ be the~blow up of the~point $P$,
with the exceptional divisor $E_1^1$ and $L^1$, $C^1$ are strict transforms of $L$ and $C$ respectively; $\pi_2: S^2\to S^1$ be the~blow up of the~point $E_1^1\cap C^1\cap L^1$,
with the exceptional divisor $E_2^2$ and $L^2$, $C^2$, $E_1^2$ are strict transforms of $L^1$, $C^1$, $E_1^1$ respectively; $\pi_3: S^3\to S^2$ be the~blow up of the~point $E_2^2\cap L^2$,
with the exceptional divisor $E$ and $L^3$, $C^3$, $E_1^3$, $E_2^3$ are strict transforms of $L^2$, $C^2$, $E_1^2$, $E_2^2$ respectively; let $\theta:S^3\to\overline{S}$ be  the~contraction of  the~curves  $E_1^3$, $E_2^3$
and $\sigma$ is the~birational contraction of $\overline{E}=\theta(E)$. We denote the strict transforms of $C$ and $L$ on $\overline{S}$ by $\overline{C}$ and $\overline{L}$. Let $\overline{E}$ be the exceptional divisor of $\sigma$. Note that $\overline{E}$ contains one singular point $\overline{P}$  which is $\frac{1}{3}(1,2)$  singular point.
\begin{center}
 \includegraphics[width=18cm]{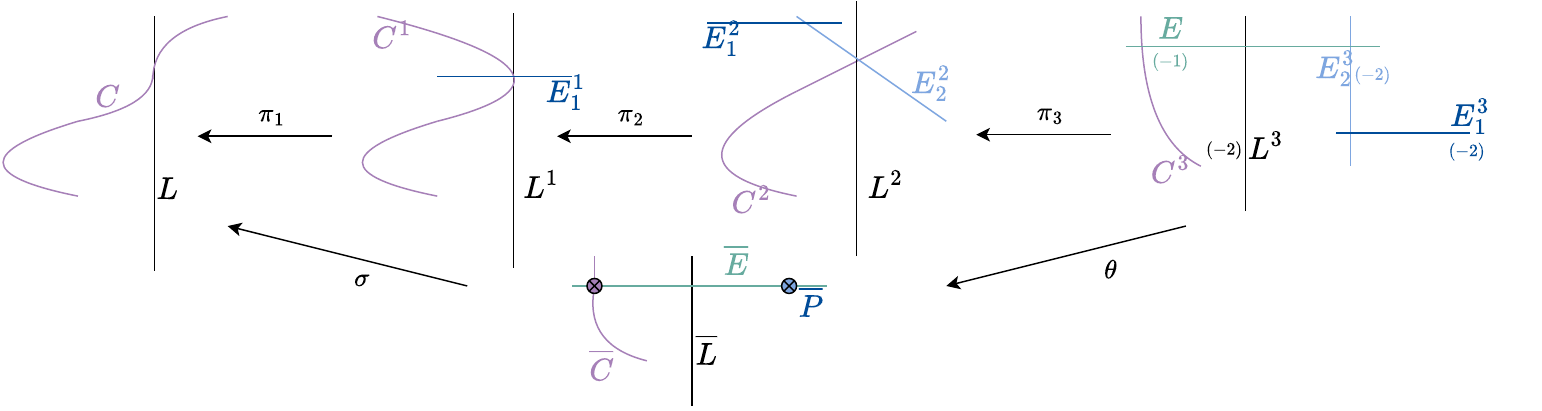}
 \end{center}
The intersections on $\overline{S}$ are given by:
\begin{center}
    \renewcommand{\arraystretch}{1.4}
    \begin{tabular}{|c|c|c|}
    \hline
         & $\overline{E}$ & $\overline{L}$ \\
    \hline
       $\overline{E}$  & $-\frac{1}{3}$ & $1$ \\
    \hline
       $\overline{L}$  & $1$ & $-2$ \\
    \hline
    \end{tabular}
\end{center}
 We have $\sigma^*(L)=\overline{L}+3\overline{E}$, $\sigma^*(C)=\overline{C}+3\overline{E}$, $\sigma^*(K_{\DP^2})=K_{\overline{S}}-3\overline{E}$. Thus, $A_{(\DP^2,\lambda C)}(\overline{E})=4-3\lambda$.
\noindent The Zariski decomposition of the divisor  $-\sigma^*(K_{\DP^2}+\lambda C)-v\overline{E}$ is given by:
\begin{align*}
&&P(v)=
\begin{cases}
-\sigma^*(K_{\DP^2}+\lambda C)-v\overline{E}\text{ if }v\in[0,3-3\lambda],\\
-\sigma^*(K_{\DP^2}+\lambda C)-v\overline{E}-\frac{1}{2}\big(v-(3-3\lambda)\big)\overline{L}\text{ if }v\in[3-3\lambda, 3(3-3\lambda)],\\
\end{cases}\\&&
N(v)=
\begin{cases}
0\text{ if }v\in[0,3-3\lambda],\\
\frac{1}{2}\big(v-(3-3\lambda)\big)\overline{L}\text{ if }v\in[3-3\lambda,3(3-3\lambda)].
\end{cases}
\end{align*}
Then
$$
P(v)^2=
\begin{cases}
(3-3\lambda)^2 -\frac{v^2}{3}\text{ if }v\in[0,3-3\lambda],\\
\frac{(v - 9 +9\lambda)^2}{6}\text{ if }v\in[3-3\lambda,3(3-3\lambda)],
\end{cases}
\text{ and }
P(v)\cdot \overline{E}=
\begin{cases}
\frac{v}{3}\text{ if }v\in[0,3-3\lambda],\\
\frac{1}{2}\big(3-3\lambda -\frac{v}{3}\big)\text{ if }v\in[3-3\lambda,3(3-3\lambda)],
\end{cases}
$$
Thus
$$S_S(\overline{E})=\frac{1}{(3-3\lambda)^2}\Big(\int_0^{3-3\lambda} (3-3\lambda)^2 -\frac{v^2}{3} dv+\int_{3-3\lambda}^{3(3-3\lambda)} \frac{(v - 9 +9\lambda)^2}{6} dv\Big)=\frac{4(3-3\lambda)}{3}$$
so that $\delta_P(\DP^2,\lambda C)\le \frac{3}{4}\cdot\frac{4-3\lambda}{3-3\lambda}=\frac{4-3\lambda}{4-4\lambda}$. For every $O\in \overline{E}$, we get if $O\in \overline{E}\backslash \overline{L}$ or if $O\in \overline{E}\cap \overline{L}$:
$$h(v)\le 
\begin{cases}
\frac{v^2}{18}\text{ if }v\in[0,3-3\lambda],\\
\frac{(v - 9 +9\lambda)^2}{72}\text{ if }v\in[3-3\lambda, 3(3-3\lambda)],
\end{cases}
\text{or }
h(v)\le \begin{cases}
\frac{v^2}{18}\text{ if }v\in[0,3-3\lambda],\\
\frac{ (v - 9 +9\lambda) (9 -5v - 9\lambda )}{72}\text{ if }v\in[3-2\lambda, 3(3-3\lambda)].
\end{cases}
$$
So that
$$S\big(W^{\overline{E}}_{\bullet,\bullet};O\big)\le \frac{2}{(3-3\lambda)^2}\Big(\int_0^{3-3\lambda} \frac{v^2}{18} dv+\int_{3-3\lambda}^{3(3-3\lambda)} \frac{(v - 9 +9\lambda)^2}{72} dv\Big)=\frac{3-3\lambda}{9}
$$
or
$$S\big(W^{\overline{E}}_{\bullet,\bullet};O\big)\le \frac{2}{(3-3\lambda)^2}\Big(\int_0^{3-3\lambda} \frac{v^2}{18} dv+\int_{3-3\lambda}^{3(3-3\lambda)} \frac{ (v - 9 +0\lambda) (9 -5v - 9\lambda )}{72} dv\Big)=\frac{3-3\lambda}{3}
$$
We have
$$
\delta_P(\DP^2,\lambda C)\geqslant\mathrm{min}\Bigg\{\frac{4-3\lambda}{4-4\lambda},\inf_{O\in\overline{E}}\frac{A_{\overline{E},\Delta_{\overline{E}}}(O)}{S\big(W^{\overline{E}}_{\bullet,\bullet};O\big)}\Bigg\},
$$
where $\Delta_{\overline{E}}=\frac{2}{3}\overline{P}+\lambda \overline{Q}$, where $\overline{Q}=\overline{E}\cap \overline{C}$. 
So that
$$
\frac{A_{\overline{E},\Delta_{\overline{E}}}(O)}{S(W_{\bullet,\bullet}^{\overline{E}};O)}=
\left\{\aligned
&\frac{3}{3-3\lambda}\ \mathrm{if}\ O=\overline{E}\cap\overline{L}\text{ or if } O=\overline{P},\\
&\frac{9(1-\lambda)}{3-3\lambda}\ \mathrm{if}\ O=\overline{Q},\\
&\frac{9}{3-3\lambda}\ \mathrm{otherwise}.
\endaligned
\right.
$$
Thus $\delta_P(\DP^2,\lambda C)=\frac{4-3\lambda}{4-4\lambda}$.
 \end{proof}

\section{Smooth quartic curves}
 \begin{lemma}\label{qc-smooth-2}
 Let $C\subset \DP^2$ be a  quartic curve, $P\in C$ be a smooth point on $C$ such that the tangent line $L$ at this point has multiplicity $2$. Then
 $$\delta_P(\DP^2,\lambda C)=\frac{3-2\lambda}{3-4 \lambda}\text{ for }\lambda\in\Big[0,\frac{3}{4}\Big].$$
 \end{lemma}
 \begin{proof}
Let $\pi_1:S^1\to S$ be the~blow up of the~point $P$,
with the exceptional divisor $E_1^1$ and $L^1$, $C^1$ are strict transforms of $L$ and $C$ respectively; $\pi_2: S^2\to S^1$ be the~blow up of the~point $E_1^1\cap C^1\cap L^1$,
with the exceptional divisor $E$ and $L^2$, $C^2$ are strict transforms of $L^1$ and $C^1$ respectively; let $\theta:S^2\to\overline{S}$ be  the~contraction of  the~curve $E_1^2$,
and $\sigma$ is the~birational contraction of $\overline{E}=\theta(E)$. We denote the strict transforms of $C$ and $L$ on $\overline{S}$ by $\overline{C}$ and $\overline{L}$. Let $\overline{E}$ be the exceptional divisor of $\sigma$. Note that $\overline{E}$ contains one singular point $\overline{P}$  which is $\frac{1}{2}(1,1)$  singular point.
\begin{center}
 \includegraphics[width=15cm]{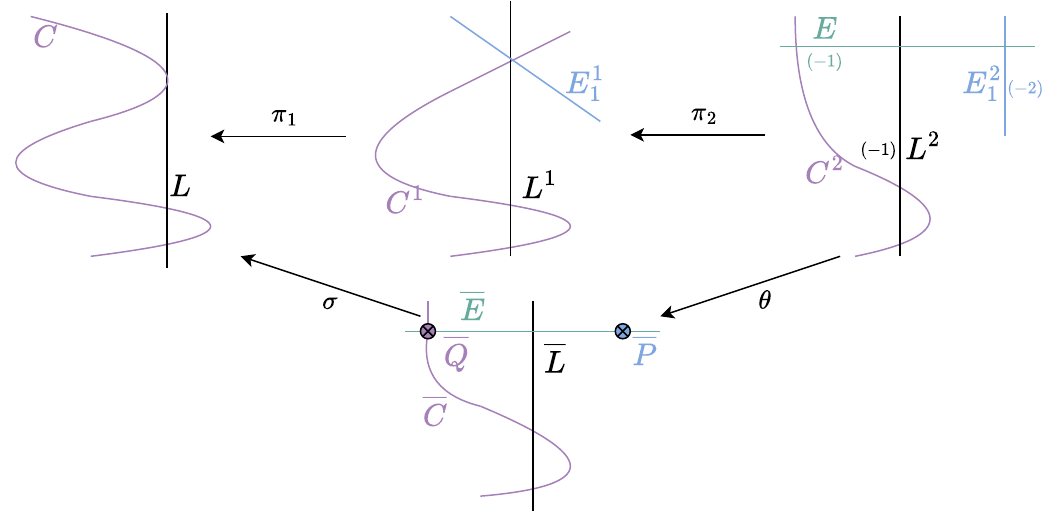}
 \end{center}
The intersections on $\overline{S}$ are given by:
\begin{center}
    \renewcommand{\arraystretch}{1.4}
    \begin{tabular}{|c|c|c|}
    \hline
         & $\overline{E}$ & $\overline{L}$ \\
    \hline
       $\overline{E}$  & $-\frac{1}{2}$ & $1$ \\
    \hline
       $\overline{L}$  & $1$ & $-1$ \\
    \hline
    \end{tabular}
\end{center}
 We have $\sigma^*(L)=\overline{L}+2\overline{E}$, $\sigma^*(C)=\overline{C}+2\overline{E}$, $\sigma^*(K_{\DP^2})=K_{\overline{S}}-2\overline{E}$. Thus, $A_{(\DP^2,\lambda C)}(\overline{E})=3-2\lambda$.
\noindent The Zariski decomposition of the divisor  $-\sigma^*(K_{\DP^2}+\lambda C)-v\overline{E}$ is given by:
\begin{align*}
&&P(v)=
\begin{cases}
-\sigma^*(K_{\DP^2}+\lambda C)-v\overline{E}\text{ if }v\in[0,3-4\lambda],\\
-\sigma^*(K_{\DP^2}+\lambda C)-v\overline{E}-\big(v-(3-4\lambda)\big)\overline{L}\text{ if }v\in[3-4\lambda, 2(3-4\lambda)],\\
\end{cases}\\&&
N(v)=
\begin{cases}
0\text{ if }v\in[0,3-4\lambda],\\
\big(v-(3-4\lambda)\big)\overline{L}\text{ if }v\in[3-4\lambda,2(3-4\lambda)].
\end{cases}
\end{align*}
Then
$$
P(v)^2=
\begin{cases}
(3-4\lambda)^2 -\frac{v^2}{2}\text{ if }v\in[0,3-4\lambda],\\
\frac{(v- 6 + 8\lambda )^2}{2}\text{ if }v\in[3-4\lambda,2(3-4\lambda)],
\end{cases}
\text{ and }
P(v)\cdot \overline{E}=
\begin{cases}
\frac{v}{2}\text{ if }v\in[0,3-4\lambda],\\
3-4\lambda -\frac{v}{2}\text{ if }v\in[3-4\lambda,2(3-4\lambda)],
\end{cases}
$$
Thus
$$S_S(\overline{E})=\frac{1}{(3-4\lambda)^2}\Big(\int_0^{3-4\lambda} (3-4\lambda)^2 -\frac{v^2}{2} dv+\int_{3-4\lambda}^{2(3-4\lambda)} \frac{(v- 6 + 8\lambda)^2}{2} dv\Big)=3-4\lambda$$
so that $\delta_P(\DP^2,\lambda C)\le \frac{3-2\lambda}{3-4\lambda}$. For every $O\in \overline{E}$, we get if $O\in \overline{E}\backslash \overline{L}$ or if $O\in \overline{E}\cap \overline{L}$:
$$h(v)\le 
\begin{cases}
\frac{v^2}{8}\text{ if }v\in[0,3-4\lambda],\\
\frac{(v- 6 + 8\lambda)^2}{8}\text{ if }v\in[3-4\lambda, 2(3-4\lambda)],
\end{cases}
\text{or }
h(v)\le \begin{cases}
\frac{v^2}{8}\text{ if }v\in[0,3-4\lambda],\\
\frac{ (v- 6 + 8\lambda) (6 - 8\lambda - 3v)}{8}\text{ if }v\in[3-4\lambda, 2(3-4\lambda)].
\end{cases}
$$
So that
$$S\big(W^{\overline{E}}_{\bullet,\bullet};O\big)\le \frac{2}{(3-4\lambda)^2}\Big(\int_0^{3-4\lambda} \frac{v^2}{8} dv+\int_{3-4\lambda}^{2(3-4\lambda)} \frac{(v- 6 + 8\lambda)^2}{8} dv\Big)=\frac{3-4\lambda}{6}
$$
or
$$S\big(W^{\overline{E}}_{\bullet,\bullet};O\big)\le \frac{2}{(3-4\lambda)^2}\Big(\int_0^{3-4\lambda} \frac{v^2}{8} dv+\int_{3-4\lambda}^{2(3-4\lambda)} \frac{ (v- 6 + 8\lambda) (6 - 8\lambda - 3v)}{8} dv\Big)=\frac{3-4\lambda}{3}
$$
We have
$$
\delta_P(\DP^2,\lambda C)\geqslant\mathrm{min}\Bigg\{\frac{3-2\lambda}{3-4\lambda},\inf_{O\in\overline{E}}\frac{A_{\overline{E},\Delta_{\overline{E}}}(O)}{S\big(W^{\overline{E}}_{\bullet,\bullet};O\big)}\Bigg\},
$$
where $\Delta_{\overline{E}}=\frac{1}{2}\overline{P}+\lambda\overline{Q}$, where $\overline{Q}=\overline{E}\cap \overline{C}$. 
So that
$$
\frac{A_{\overline{E},\Delta_{\overline{E}}}(O)}{S(W_{\bullet,\bullet}^{\overline{E}};O)}=
\left\{\aligned
&\frac{3}{3-4\lambda}\ \mathrm{if}\ O=\overline{E}\cap\overline{L} \text{ or if } O=\overline{P},\\
&\frac{6(1-\lambda)}{3-4\lambda}\ \mathrm{if}\ O=\overline{Q},\\
&\frac{6}{3-4\lambda}\ \mathrm{otherwise}.
\endaligned
\right.
$$
Thus $\delta_P(\DP^2,\lambda C)=\frac{3-2\lambda}{3-4\lambda}$.
 \end{proof}

 \begin{lemma}
 Let $C\subset \DP^2$ be a quartic curve, $P\in C$ be a smooth point on $C$ such that the tangent line $L$ at this point has multiplicity $3$. Then
 $$\delta_P(\DP^2,\lambda C)=\frac{3}{4}\cdot\frac{4-3\lambda}{3-4\lambda}\text{ for }\lambda\in\Big[0,\frac{3}{4}\Big].$$
 \end{lemma}
 \begin{proof}
Let $\pi_1:S^1\to S$ be the~blow up of the~point $P$,
with the exceptional divisor $E_1^1$ and $L^1$, $C^1$ are strict transforms of $L$ and $C$ respectively; $\pi_2: S^2\to S^1$ be the~blow up of the~point $E_1^1\cap C^1\cap L^1$,
with the exceptional divisor $E_2^2$ and $L^2$, $C^2$, $E_1^2$ are strict transforms of $L^1$, $C^1$, $E_1^1$ respectively; $\pi_3: S^3\to S^2$ be the~blow up of the~point $E_2^2\cap L^2$,
with the exceptional divisor $E$ and $L^3$, $C^3$, $E_1^3$, $E_2^3$ are strict transforms of $L^2$, $C^2$, $E_1^2$, $E_2^2$ respectively; let $\theta:S^3\to\overline{S}$ be  the~contraction of  the~curves  $E_1^3$, $E_2^3$
and $\sigma$ is the~birational contraction of $\overline{E}=\theta(E)$. We denote the strict transforms of $C$ and $L$ on $\overline{S}$ by $\overline{C}$ and $\overline{L}$. Let $\overline{E}$ be the exceptional divisor of $\sigma$. Note that $\overline{E}$ contains one singular point $\overline{P}$  which is $\frac{1}{3}(1,2)$  singular point.
\begin{center}
 \includegraphics[width=18cm]{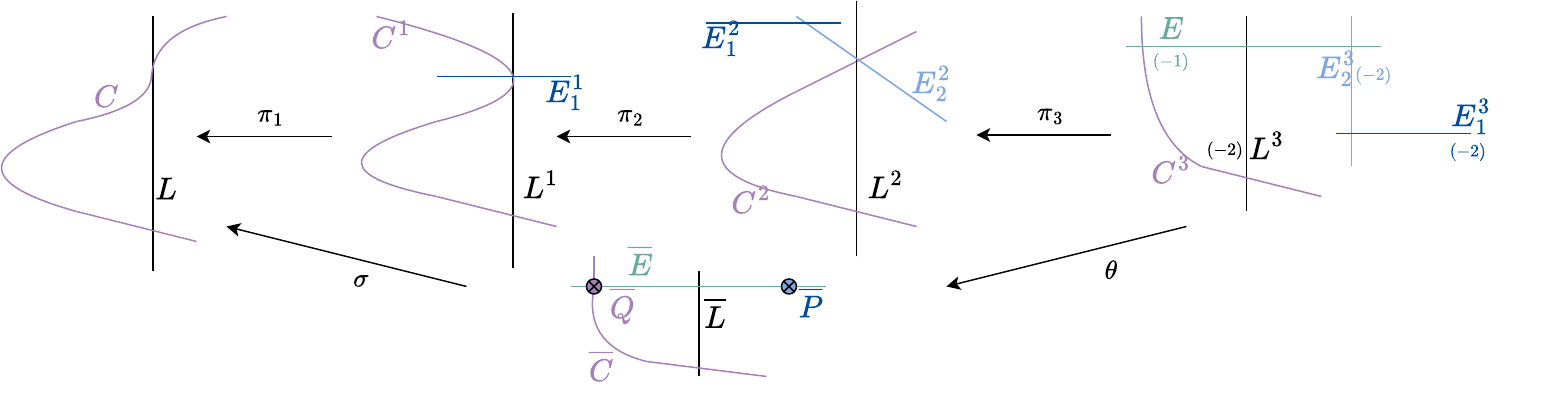}
 \end{center}
The intersections on $\overline{S}$ are given by:
\begin{center}
    \renewcommand{\arraystretch}{1.4}
    \begin{tabular}{|c|c|c|}
    \hline
         & $\overline{E}$ & $\overline{L}$ \\
    \hline
       $\overline{E}$  & $-\frac{1}{3}$ & $1$ \\
    \hline
       $\overline{L}$  & $1$ & $-2$ \\
    \hline
    \end{tabular}
\end{center}
 We have $\sigma^*(L)=\overline{L}+3\overline{E}$, $\sigma^*(C)=\overline{C}+3\overline{E}$, $\sigma^*(K_{\DP^2})=K_{\overline{S}}-3\overline{E}$. Thus, $A_{(\DP^2,\lambda C)}(\overline{E})=4-3\lambda$.
\noindent The Zariski decomposition of the divisor  $-\sigma^*(K_{\DP^2}+\lambda C)-v\overline{E}$ is given by:
\begin{align*}
&&P(v)=
\begin{cases}
-\sigma^*(K_{\DP^2}+\lambda C)-v\overline{E}\text{ if }v\in[0,3-4\lambda],\\
-\sigma^*(K_{\DP^2}+\lambda C)-v\overline{E}-\frac{1}{2}\big(v-(3-4\lambda)\big)\overline{L}\text{ if }v\in[3-4\lambda, 3(3-4\lambda)],\\
\end{cases}\\&&
N(v)=
\begin{cases}
0\text{ if }v\in[0,3-4\lambda],\\
\frac{1}{2}\big(v-(3-4\lambda)\big)\overline{L}\text{ if }v\in[3-4\lambda,3(3-4\lambda)].
\end{cases}
\end{align*}
Then
$$
P(v)^2=
\begin{cases}
(3-4\lambda)^2 -\frac{v^2}{3}\text{ if }v\in[0,3-4\lambda],\\
\frac{(v - 9 +12\lambda)^2}{6}\text{ if }v\in[3-4\lambda,3(3-4\lambda)],
\end{cases}
\text{ and }
P(v)\cdot \overline{E}=
\begin{cases}
\frac{v}{3}\text{ if }v\in[0,3-4\lambda],\\
\frac{1}{2}\big(3-4\lambda -\frac{v}{3}\big)\text{ if }v\in[3-4\lambda,3(3-4\lambda)],
\end{cases}
$$
Thus
$$S_S(\overline{E})=\frac{1}{(3-4\lambda)^2}\Big(\int_0^{3-4\lambda} (3-4\lambda)^2 -\frac{v^2}{3} dv+\int_{3-4\lambda}^{3(3-4\lambda)} \frac{(v - 9 +12\lambda)^2}{6} dv\Big)=\frac{4(3-4\lambda)}{3}$$
so that $\delta_P(\DP^2,\lambda C)\le \frac{3}{4}\cdot\frac{4-3\lambda}{3-4\lambda}$. For every $O\in \overline{E}$, we get if $O\in \overline{E}\backslash \overline{L}$ or if $O\in \overline{E}\cap \overline{L}$:
$$h(v)\le 
\begin{cases}
\frac{v^2}{18}\text{ if }v\in[0,3-4\lambda],\\
\frac{(v - 9 +12\lambda)^2}{72}\text{ if }v\in[3-4\lambda, 3(3-4\lambda)],
\end{cases}
\text{or }
h(v)\le \begin{cases}
\frac{v^2}{18}\text{ if }v\in[0,3-4\lambda],\\
\frac{ (v - 9 +12\lambda) (9 -5v - 12\lambda )}{72}\text{ if }v\in[3-4\lambda, 3(3-4\lambda)].
\end{cases}
$$
So that
$$S\big(W^{\overline{E}}_{\bullet,\bullet};O\big)\le \frac{2}{(3-4\lambda)^2}\Big(\int_0^{3-4\lambda} \frac{v^2}{18} dv+\int_{3-4\lambda}^{3(3-4\lambda)} \frac{(v - 9 +12\lambda)^2}{72} dv\Big)=\frac{3-4\lambda}{9}
$$
or
$$S\big(W^{\overline{E}}_{\bullet,\bullet};O\big)\le \frac{2}{(3-4\lambda)^2}\Big(\int_0^{3-4\lambda} \frac{v^2}{18} dv+\int_{3-4\lambda}^{3(3-4\lambda)} \frac{ (v - 9 +12\lambda) (9 -5v - 12\lambda )}{72} dv\Big)=\frac{3-4\lambda}{3}
$$
We have
$$
\delta_P(\DP^2,\lambda C)\geqslant\mathrm{min}\Bigg\{\frac{3}{4}\cdot\frac{4-3\lambda}{3-4\lambda},\inf_{O\in\overline{E}}\frac{A_{\overline{E},\Delta_{\overline{E}}}(O)}{S\big(W^{\overline{E}}_{\bullet,\bullet};O\big)}\Bigg\},
$$
where $\Delta_{\overline{E}}=\frac{2}{3}\overline{P}+\lambda \overline{Q}$, where $\overline{Q}=\overline{E}\cap \overline{C}$.  
So that
$$
\frac{A_{\overline{E},\Delta_{\overline{E}}}(O)}{S(W_{\bullet,\bullet}^{\overline{E}};O)}=
\left\{\aligned
&\frac{3}{3-4\lambda}\ \mathrm{if}\ O=\overline{E}\cap\overline{L}\text{ or if } O=\overline{P},\\
&\frac{9(1-\lambda)}{3-4\lambda}\ \mathrm{if}\ O=\overline{Q},\\
&\frac{9}{3-4\lambda}\ \mathrm{otherwise}.
\endaligned
\right.
$$
Thus $\delta_P(\DP^2,\lambda C)=\frac{3}{4}\cdot\frac{4-3\lambda}{3-4\lambda}$.
 \end{proof}

  \begin{lemma}
 Let $C\subset \DP^2$ be a  quartic curve, $P\in C$ be a smooth point on $C$ such that the tangent line $L$ at this point has multiplicity $3$. Then
 $$\delta_P(\DP^2,\lambda C)=\frac{3}{5}\cdot\frac{5-4\lambda}{3-4\lambda}\text{ for }\lambda\in\Big[0,\frac{3}{4}\Big].$$
 \end{lemma}
 \begin{proof}
Let $\pi_1:S^1\to S$ be the~blow up of the~point $P$,
with the exceptional divisor $E_1^1$ and $L^1$, $C^1$ are strict transforms of $L$ and $C$ respectively; $\pi_2: S^2\to S^1$ be the~blow up of the~point $E_1^1\cap C^1\cap L^1$,
with the exceptional divisor $E_2^2$ and $L^2$, $C^2$, $E_1^2$ are strict transforms of $L^1$, $C^1$, $E_1^1$ respectively; $\pi_3: S^3\to S^2$ be the~blow up of the~point $E_2^2\cap L^2$,
with the exceptional divisor $E_3^3$ and $L^3$, $C^3$, $E_1^3$, $E_2^3$ are strict transforms of $L^2$, $C^2$, $E_1^2$, $E_2^2$ respectively; $\pi_4: S^4\to S^3$ be the~blow up of the~point $E_2^2\cap L^2$,
with the exceptional divisor $E$ and $L^4$, $C^4$, $E_1^4$, $E_2^4$, $E_3^4$ are strict transforms of $L^3$, $C^3$, $E_1^3$, $E_2^3$, $E_3^3$  respectively; let $\theta:S^4\to\overline{S}$ be  the~contraction of  the~curves  $E_1^3$, $E_2^3$, $E_3^3$
and $\sigma$ is the~birational contraction of $\overline{E}=\theta(E)$. We denote the strict transforms of $C$ and $L$ on $\overline{S}$ by $\overline{C}$ and $\overline{L}$. Let $\overline{E}$ be the exceptional divisor of $\sigma$. Note that $\overline{E}$ contains one singular point $\overline{P}$  which is $\frac{1}{4}(1,3)$  singular point.
\begin{center}
 \includegraphics[width=18cm]{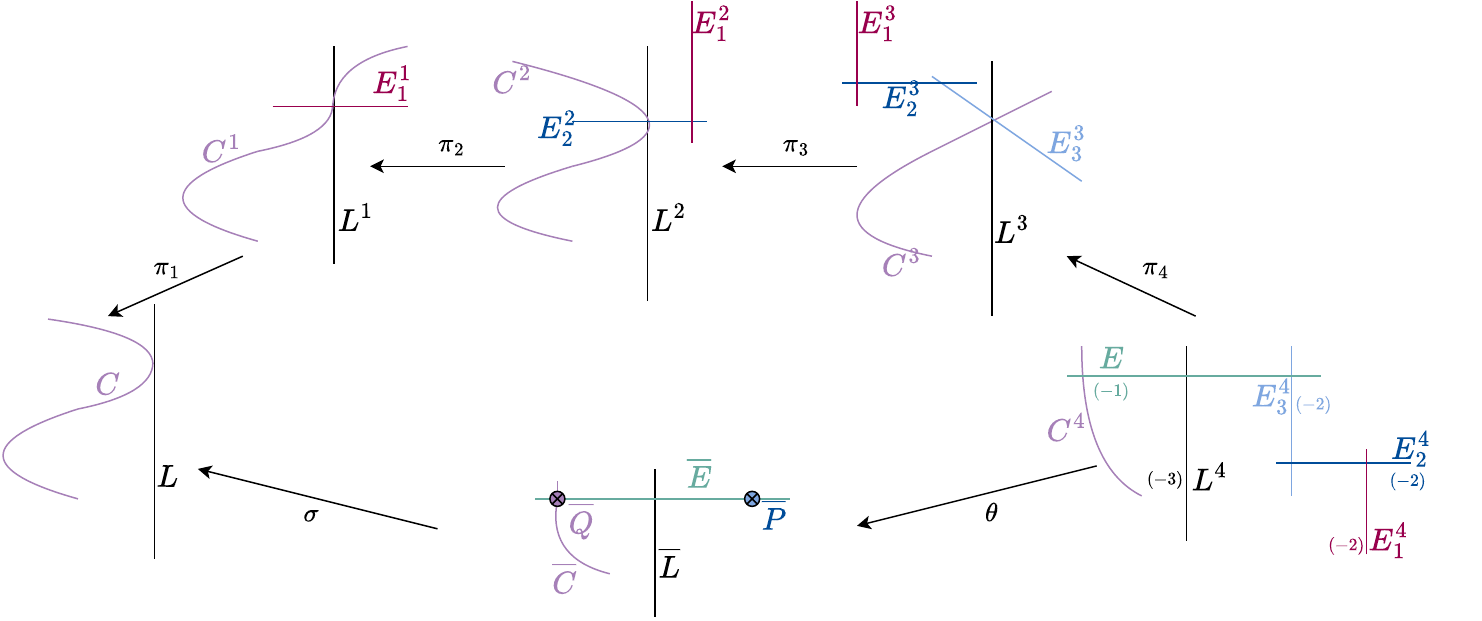}
 \end{center}
The intersections on $\overline{S}$ are given by:
\begin{center}
    \renewcommand{\arraystretch}{1.4}
    \begin{tabular}{|c|c|c|}
    \hline
         & $\overline{E}$ & $\overline{L}$ \\
    \hline
       $\overline{E}$  & $-\frac{1}{4}$ & $1$ \\
    \hline
       $\overline{L}$  & $1$ & $-3$ \\
    \hline
    \end{tabular}
\end{center}
 We have $\sigma^*(L)=\overline{L}+4\overline{E}$, $\sigma^*(C)=\overline{C}+4\overline{E}$, $\sigma^*(K_{\DP^2})=K_{\overline{S}}-4\overline{E}$. Thus, $A_{(\DP^2,\lambda C)}(\overline{E})=5-4\lambda$.
\noindent The Zariski decomposition of the divisor  $-\sigma^*(K_{\DP^2}+\lambda C)-v\overline{E}$ is given by:
\begin{align*}
&&P(v)=
\begin{cases}
-\sigma^*(K_{\DP^2}+\lambda C)-v\overline{E}\text{ if }v\in[0,3-4\lambda],\\
-\sigma^*(K_{\DP^2}+\lambda C)-v\overline{E}-\frac{1}{3}\big(v-(3-4\lambda)\big)\overline{L}\text{ if }v\in[3-4\lambda, 4(3-4\lambda)],\\
\end{cases}\\&&
N(v)=
\begin{cases}
0\text{ if }v\in[0,3-4\lambda],\\
\frac{1}{3}\big(v-(3-4\lambda)\big)\overline{L}\text{ if }v\in[3-4\lambda,4(3-4\lambda)].
\end{cases}
\end{align*}
Then
$$
P(v)^2=
\begin{cases}
(3-4\lambda)^2 -\frac{v^2}{4}\text{ if }v\in[0,3-4\lambda],\\
\frac{(v - 12 +16\lambda)^2}{12}\text{ if }v\in[3-4\lambda,4(3-4\lambda)],
\end{cases}
\text{ and }
P(v)\cdot \overline{E}=
\begin{cases}
\frac{v}{4}\text{ if }v\in[0,3-4\lambda],\\
\frac{1}{3}\big(3-4\lambda -\frac{v}{4}\big)\text{ if }v\in[3-4\lambda,4(3-4\lambda)],
\end{cases}
$$
Thus
$$S_S(\overline{E})=\frac{1}{(3-4\lambda)^2}\Big(\int_0^{3-4\lambda} (3-4\lambda)^2 -\frac{v^2}{4} dv+\int_{3-4\lambda}^{4(3-4\lambda)} \frac{(v - 12 +16\lambda)^2}{12} dv\Big)=\frac{5(3-4\lambda)}{3}$$
so that $\delta_P(\DP^2,\lambda C)\le \frac{3}{5}\frac{5-4\lambda}{3-4\lambda}$. For every $O\in \overline{E}$, we get if $O\in \overline{E}\backslash \overline{L}$ or if $O\in \overline{E}\cap \overline{L}$:
$$h(v)\le 
\begin{cases}
\frac{v^2}{32}\text{ if }v\in[0,3-4\lambda],\\
\frac{(v - 12 +16\lambda)^2}{72}\text{ if }v\in[3-4\lambda, 4(3-4\lambda)],
\end{cases}
\text{or }
h(v)\le \begin{cases}
\frac{v^2}{32}\text{ if }v\in[0,3-4\lambda],\\
\frac{ (v - 12 +16\lambda) (12 -7v - 16\lambda )}{288}\text{ if }v\in[3-4\lambda, 4(3-4\lambda)].
\end{cases}
$$
So that
$$S\big(W^{\overline{E}}_{\bullet,\bullet};O\big)\le \frac{2}{(3-4\lambda)^2}\Big(\int_0^{3-4\lambda} \frac{v^2}{32} dv+\int_{3-4\lambda}^{4(3-4\lambda)} \frac{(v - 12 +16\lambda)^2}{288} dv\Big)=\frac{3-4\lambda}{12}
$$
or
$$S\big(W^{\overline{E}}_{\bullet,\bullet};O\big)\le \frac{2}{(3-4\lambda)^2}\Big(\int_0^{3-4\lambda} \frac{v^2}{32} dv+\int_{3-4\lambda}^{4(3-4\lambda)} \frac{ (v - 12 +16\lambda) (12 -7v - 16\lambda )}{288} dv\Big)=\frac{3-4\lambda}{3}
$$
We have
$$
\delta_P(\DP^2,\lambda C)\geqslant\mathrm{min}\Bigg\{\frac{3}{5}\cdot\frac{5-4\lambda}{3-4\lambda},\inf_{O\in\overline{E}}\frac{A_{\overline{E},\Delta_{\overline{E}}}(O)}{S\big(W^{\overline{E}}_{\bullet,\bullet};O\big)}\Bigg\},
$$
where $\Delta_{\overline{E}}=\frac{3}{4}\overline{P}+\lambda\overline{Q}$, where $\overline{Q}=\overline{E}\cap \overline{C}$. 
So that
$$
\frac{A_{\overline{E},\Delta_{\overline{E}}}(O)}{S(W_{\bullet,\bullet}^{\overline{E}};O)}=
\left\{\aligned
&\frac{3}{3-4\lambda}\ \mathrm{if}\ O=\overline{E}\cap\overline{L}\text{ or if } O=\overline{P},\\
&\frac{12(1-\lambda)}{3-4\lambda}\ \mathrm{if}\ O=\overline{Q},\\
&\frac{12}{3-4\lambda}\ \mathrm{otherwise}.
\endaligned
\right.
$$
Thus $\delta_P(\DP^2,\lambda C)=\frac{3}{5}\cdot\frac{5-4\lambda}{3-4\lambda}$.
 \end{proof}
 \noindent This proves the following theorem:
 \begin{theorem}
  Let $C\subset \DP^2$ be a smooth quartic curve. Then  for $\lambda\in\big[0,\frac{3}{4}\big]$ we have
 $$\delta(\DP^2,\lambda C)=
 \begin{cases}
 \frac{3}{4}\cdot\frac{4-3\lambda}{3-4\lambda}\text{ if there is no  4-tangent to }C,\\
 \frac{3}{5}\cdot\frac{5-4\lambda}{3-4\lambda}\text{ if there exists a 4-tangent to }C.
 \end{cases}.$$
 \end{theorem}

 \section{Curves with $A_1$ singularities}
   \begin{lemma}\label{A1-points}
 Let $C\subset \DP^2$ be a plane curve of degree $d$ with an $A_1$ singularity at point $P\in C$ on $C$. Then
 $$\delta_P(\DP^2,\lambda C)=\frac{3-3\lambda}{3-d \lambda }\text{ for }\lambda\in\Big[0,\min\Big\{1, \frac{3}{d}\Big\}\Big].$$
 \end{lemma}
 \begin{proof}
Let $\pi:\widetilde{S}\to S$ be the~blow up of the~point $P$,
with the exceptional divisor $E$. We denote the strict transform of $C$  on $\widetilde{S}$ by $\widetilde{C}$.
\begin{center}
 \includegraphics[width=8cm]{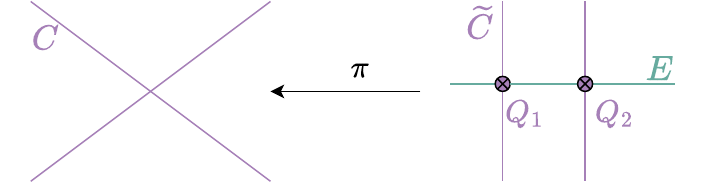}
 \end{center}
 We have $\pi^*(C)=\widetilde{C}+2E$, $\pi^*(K_{\DP^2})=K_{\widetilde{S}}-E$. Thus, $A_{(\DP^2,\lambda C)}(E)=2-2\lambda$.
\noindent The Zariski decomposition of the divisor  $-\pi^*(K_{\DP^2}+\lambda C)-vE$ is given by:
\begin{align*}
P(v)=-\pi^*(K_{\DP^2}+\lambda C)-vE \text{ and }N(v)=0\text{ if }v\in[0,3- d \lambda ].
\end{align*}
Then
$$
P(v)^2=\big(v-(3-d \lambda )\big)\big(v+(3-d \lambda )\big)\text{ and }P(v)\cdot E= v\text{ if }v\in[0,3- d \lambda ].
$$
Thus
$$S_{(\DP^2, \lambda C)}(E)=\frac{1}{(3- d \lambda )^2}\Big(\int_0^{3-d \lambda } \big(v-(3-d \lambda )\big)\big(v+(3-d \lambda )\big) dv\Big)=\frac{2(3-d \lambda )}{3}$$
so that $\delta_P(\DP^2,\lambda C)\le \frac{3}{2}\cdot \frac{2-2\lambda}{3-d \lambda }= \frac{3(1-\lambda)}{3-d \lambda } $. For every $O\in E$ we get:
$$h(v) = \frac{v^2}{2}\text{ if }v\in[0,3-d \lambda ].
$$
So that
$$S\big(W^{E}_{\bullet,\bullet};O\big)= \frac{2}{(3-d \lambda )^2}\Big(\int_0^{3-d \lambda } \frac{v^2}{2} dv\Big)=\frac{3-d \lambda }{3}$$
We have
$$
\delta_P(\DP^2,\lambda C)\geqslant\mathrm{min}\Bigg\{ \frac{3(1-\lambda)}{3-d \lambda } ,\inf_{O\in E}\frac{A_{E,\Delta_{E}}(O)}{S\big(W^{E}_{\bullet,\bullet};O\big)}\Bigg\},
$$
where $\Delta_{E}=\lambda Q_1+\lambda Q_2$ where $ Q_1+ Q_2=\widetilde{C}|_E$. 
So that
$$
\frac{A_{\overline{E},\Delta_{\overline{E}}}(O)}{S(W_{\bullet,\bullet}^{\overline{E}};O)}=
\left\{\aligned
&\frac{3(1-\lambda)}{3-d \lambda }\ \mathrm{if}\ O\in\{Q_1,Q_2\},\\
&\frac{3}{3-d \lambda }\ \mathrm{otherwise}.
\endaligned
\right.
$$
Thus $\delta_P(\DP^2,\lambda C)= \frac{3(1-\lambda)}{3-d \lambda }$.
 \end{proof}
 \noindent This proves the following theorems:
  \begin{theorem}
  Suppose $C_2\subset \DP^2$ is a curve of degree $2$ with  $A_1$ singularity, i.e. $C$ is a union of two different lines. Then  for $\lambda\in\big[0,1\big]$ we have:
 $$\delta(\DP^2,\lambda C_2)=\frac{3-3\lambda}{3- 2 \lambda}.$$
 \end{theorem}
  \begin{theorem}
  Suppose $C_3\subset \DP^2$ is a cubic curve with at most $A_1$ singularities and at least one $A_1$ singularity. Then  for $\lambda\in\big[0,1\big]$ we have:
 $$\delta(\DP^2,\lambda C_3)=1.$$
 \end{theorem}
 \begin{theorem}
  Suppose $C_4\subset \DP^2$ is a quartic curve with at most $A_1$ singularities and at least one $A_1$ singularity. Then  for $\lambda\in\big[0,\frac{3}{4}\big]$ we have:
 $$\delta(\DP^2,\lambda C_4)=\frac{3-3\lambda}{3- 4 \lambda}.$$
 \end{theorem}
 \section{Curves with $A_2$ singularities}
 \begin{lemma}
 Let $C\subset \DP^2$ be a plane curve of degree $d$ with an $A_2$ singularity at point $P\in C$  on $C$. Then
 $$\delta_P(\DP^2,\lambda C)=\frac{3}{5}\cdot\frac{5-6\lambda}{3-d \lambda }\text{ for }\lambda\in\Big[0,\min\Big\{\frac{5}{6}, \frac{3}{d}\Big\}\Big].$$
 \end{lemma}
 \begin{proof}
Let $L$ be a tangent line at point $P$.  Let $\pi_1:S^1\to S$ be the~blow up of the~point $P$,
with the exceptional divisor $E_1^1$ and $L^1$, $C^1$ are strict transforms of $L$ and $C$ respectively; $\pi_2: S^2\to S^1$ be the~blow up of the~point $E_1^1\cap C^1\cap L^1$,
with the exceptional divisor $E_2^2$ and $L^2$, $C^2$, $E_1^2$ are strict transforms of $L^1$, $C^1$, $E_1^1$ respectively; $\pi_3: S^3\to S^2$ be the~blow up of the~point $E_1^2\cap E_2^2\cap C^2$,
with the exceptional divisor $E$ and $L^3$, $C^3$, $E_1^3$, $E_2^3$ are strict transforms of $L^2$, $C^2$, $E_1^2$, $E_2^2$ respectively; let $\theta:S^3\to\overline{S}$ be  the~contraction of  the~curves  $E_1^3$, $E_2^3$
and $\sigma$ is the~birational contraction of $\overline{E}=\theta(E)$. We denote the strict transforms of $C$ and $L$ on $\overline{S}$ by $\overline{C}$ and $\overline{L}$. Let $\overline{E}$ be the exceptional divisor of $\sigma$. Note that $\overline{E}$ contains two singular points $\overline{P}_1$ and  $\overline{P}_2$   which are $\frac{1}{3}(1,1)$ and  $\frac{1}{2}(1,1)$ singular points respectively.
\begin{center}
 \includegraphics[width=16cm]{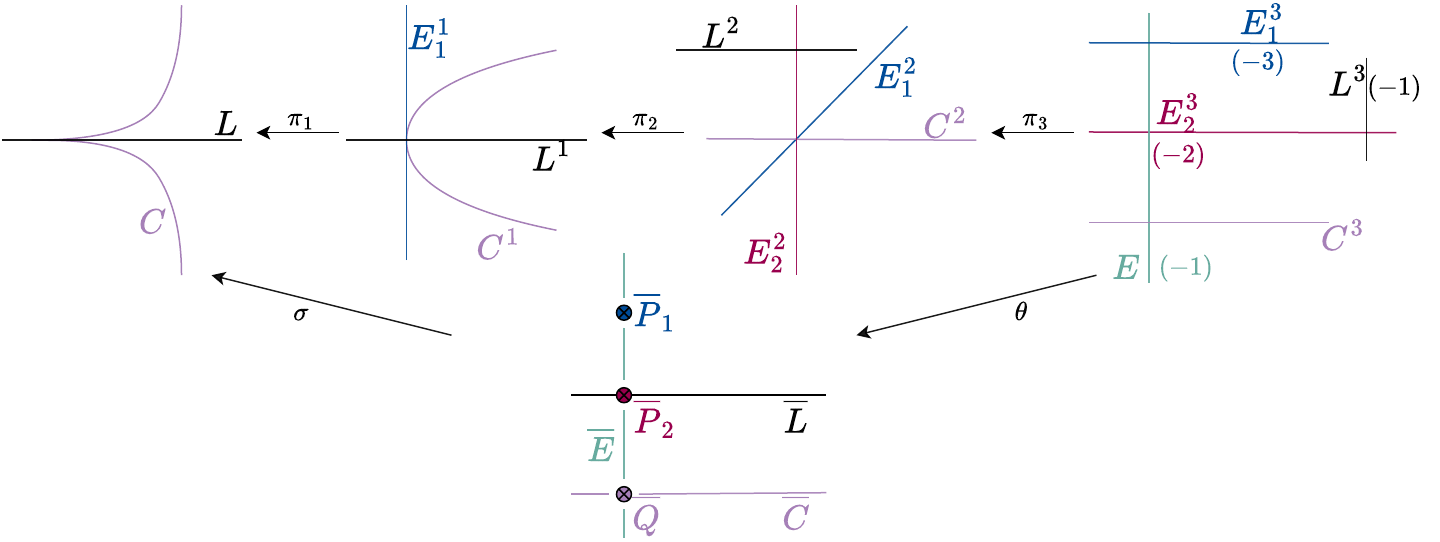}
 \end{center}
The intersections on $\overline{S}$ are given by:
\begin{center}
    \renewcommand{\arraystretch}{1.4}
    \begin{tabular}{|c|c|c|}
    \hline
         & $\overline{E}$ & $\overline{L}$ \\
    \hline
       $\overline{E}$  & $-\frac{1}{6}$ & $\frac{1}{2}$ \\
    \hline
       $\overline{L}$  & $\frac{1}{2}$ & $-\frac{1}{2}$ \\
    \hline
    \end{tabular}
\end{center}
 We have $\sigma^*(L)=\overline{L}+3\overline{E}$, $\sigma^*(C)=\overline{C}+6\overline{E}$, $\sigma^*(K_{\DP^2})=K_{\overline{S}}-4\overline{E}$. Thus, $A_{(\DP^2,\lambda C)}(\overline{E})=5-6\lambda$.
\noindent The Zariski decomposition of the divisor  $-\sigma^*(K_{\DP^2}+\lambda C)-v\overline{E}$ is given by:
\begin{align*}
&&P(v)=
\begin{cases}
-\sigma^*(K_{\DP^2}+\lambda C)-v\overline{E}\text{ if }v\in[0,2(3-d \lambda )],\\
-\sigma^*(K_{\DP^2}+\lambda C)-v\overline{E}-2\big(v-(3-d \lambda )\big)\overline{L}\text{ if }v\in[2(3-d \lambda ), 3(3-d \lambda )],\\
\end{cases}\\&&
N(v)=
\begin{cases}
0\text{ if }v\in[0,2(3-d \lambda )],\\
2\big(v-(3-d \lambda )\big)\overline{L}\text{ if }v\in[2(3-d \lambda ),3(3-d \lambda )].
\end{cases}
\end{align*}
Then
{\small $$
P(v)^2=
\begin{cases}
(3-d \lambda )^2 -\frac{v^2}{6}\text{ if }v\in[0,2(3-4)\lambda],\\
\frac{(v - 3(3-d\lambda))^2}{3}\text{ if }v\in[2(3-d \lambda ),3(3-d \lambda )],
\end{cases}
\text{and }
P(v)\cdot \overline{E}=
\begin{cases}
\frac{v}{6}\text{ if }v\in[0,2(3-d \lambda )],\\
3-d \lambda  -\frac{v}{3}\text{ if }v\in[2(3-d \lambda ),3(3-d \lambda )],
\end{cases}
$$}
Thus
$$S_{(\DP^2, \lambda C)}(\overline{E})=\frac{1}{(3-d \lambda )^2}\Big(\int_0^{2(3-d \lambda )} (3-d \lambda )^2 -\frac{v^2}{6} dv+\int_{2(3-d \lambda )}^{3(3-d \lambda )} \frac{(v - 3(3-d\lambda))^2}{3} dv\Big)=\frac{5(3-d \lambda )}{3}$$
so that $\delta_P(\DP^2,\lambda C)\le \frac{3}{5}\cdot\frac{5-6\lambda}{3-d \lambda }$. For every $O\in \overline{E}$, we get if $O\in \overline{E}\backslash \overline{L}$ or if $O\in \overline{E}\cap \overline{L}$:
\begin{align*}
&&h(v)\le 
\begin{cases}
\frac{v^2}{72}\text{ if }v\in[0,2(3-d \lambda )],\\
\frac{(v - 3(3-d\lambda))^2}{18}\text{ if }v\in[2(3-d \lambda ), 3(3-d \lambda )],
\end{cases}\\
&\text{or}&
h(v)\le \begin{cases}
\frac{v^2}{72}\text{ if }v\in[0,2(3-d \lambda )],\\
\frac{ (v - 3(3-d\lambda)) (3(3 -d\lambda) -2v)}{18}\text{ if }v\in[2(3-d \lambda ), 3(3-d \lambda )].
\end{cases}
\end{align*}
So that
$$S\big(W^{\overline{E}}_{\bullet,\bullet};O\big)\le \frac{2}{(3-d \lambda)^2}\Big(\int_0^{2(3-d \lambda )} \frac{v^2}{72} dv+\int_{2(3-d\lambda )}^{3(3-d \lambda )} \frac{(v - 3(3-d\lambda))^2}{18} dv\Big)=\frac{3-d \lambda }{9}
$$
or
$$S\big(W^{\overline{E}}_{\bullet,\bullet};O\big)\le \frac{2}{(3-d \lambda )^2}\Big(\int_0^{2(3-d \lambda )} \frac{v^2}{72} dv+\int_{2(3-d \lambda )}^{3(3-d \lambda )} \frac{ (v - 3(3-d\lambda)) (3(3 -d\lambda) -2v)}{18} dv\Big)=\frac{3-d \lambda }{6}
$$
We have
$$
\delta_P(\DP^2,\lambda C)\geqslant\mathrm{min}\Bigg\{\frac{3}{5}\cdot\frac{5-6\lambda}{3-d \lambda },\inf_{O\in\overline{E}}\frac{A_{\overline{E},\Delta_{\overline{E}}}(O)}{S\big(W^{\overline{E}}_{\bullet,\bullet};O\big)}\Bigg\},
$$
where $\Delta_{\overline{E}}=\frac{2}{3}\overline{P}_1+\frac{1}{2}\overline{P}_2+\lambda\overline{Q}$  where $\overline{Q}=\overline{C}|_{\overline{E}}$. 
So that
$$
\frac{A_{\overline{E},\Delta_{\overline{E}}}(O)}{S(W_{\bullet,\bullet}^{\overline{E}};O)}=
\left\{\aligned
&\frac{3}{3-d \lambda }\ \mathrm{if}\ O\in \{\overline{P}_1,\overline{P}_2\},\\
&\frac{9(1-\lambda)}{3-d \lambda }\ \mathrm{if}\ O= \overline{Q},\\
&\frac{9}{3-d \lambda }\ \mathrm{otherwise}.
\endaligned
\right.
$$
Thus $\delta_P(\DP^2,\lambda C)=\frac{3}{5}\cdot\frac{5-6\lambda}{3-d \lambda }$ \text{ for }$\lambda\in\big[0,\min\big\{\frac{5}{6}, \frac{3}{d}\big\}\big]$.
 \end{proof}

  \noindent This proves the following theorems:
   \begin{theorem}
  Suppose $C_3\subset \DP^2$ is a cubic curve with  $A_3$ singularity. Then  for $\lambda\in\big[0,\frac{5}{6}\big] $ we have:
 $$\delta(\DP^2,\lambda C_3)=\frac{5-6\lambda}{5-5 \lambda }.$$
 \end{theorem}
 \begin{theorem}
  Suppose $C_4\subset \DP^2$ is a quartic curve with at most $A_2$ singularities and at least one $A_2$ singularity. Then  for $\lambda\in\big[0,\frac{3}{4}\big]$ we have:
 $$\delta(\DP^2,\lambda C_4)=\frac{3}{5}\cdot\frac{5-6\lambda}{3-4 \lambda }.$$
 \end{theorem}
 \section{Curves with $A_3$ singularities}
 \begin{lemma}
 Let $C\subset \DP^2$ be a plane curve of degree $d$ with an $A_3$ singularity at point $P\in C$  on $C$. Then
 $$\delta_P(\DP^2,\lambda C)=\frac{3-4 \lambda }{3-d \lambda }\text{ for }\lambda\in\Big[0,\min\Big\{\frac{3}{4}, \frac{3}{d}\Big\}\Big].$$
 \end{lemma}
 \begin{proof}
We have two options:
\\{\it Option 1:} The tangent line $L$ to $C$ at point $P$ is not a component of $C$. Let $\pi_1:S^1\to S$ be the~blow up of the~point $P$,
with the exceptional divisor $E_1^1$ and $L^1$, $C^1$ are strict transforms of $L$ and $C$ respectively; $\pi_2: S^2\to S^1$ be the~blow up of the~point $E_1^1\cap C^1\cap L^1$,
with the exceptional divisor $E$ and $L^2$, $C^2$ are strict transforms of $L^1$ and $C^1$ respectively; let $\theta:S^2\to\overline{S}$ be  the~contraction of  the~curve $E_1^2$,
and $\sigma$ is the~birational contraction of $\overline{E}=\theta(E)$. We denote the strict transforms of $C$ and $L$ on $\overline{S}$ by $\overline{C}$ and $\overline{L}$. Let $\overline{E}$ be the exceptional divisor of $\sigma$. Note that $\overline{E}$ contains one singular point $\overline{P}$  which is $\frac{1}{2}(1,1)$  singular point.
\begin{center}
 \includegraphics[width=11cm]{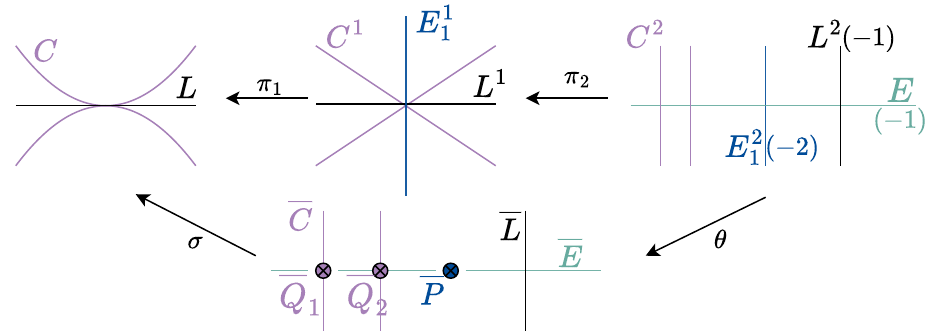}
 \end{center}
 {\it Option 2:} The tangent line $L$ to $C$ at point $P$ is a component of $C$. Let $C=\mathcal{C}+L$. Let $\pi_1:S^1\to S$ be the~blow up of the~point $P$,
with the exceptional divisor $E_1^1$ and $L^1$, $C^1$ are strict transforms of $L$ and $\mathcal{C}$ respectively; $\pi_2: S^2\to S^1$ be the~blow up of the~point $E_1^1\cap C^1\cap L^1$,
with the exceptional divisor $E$ and $L^2$, $C^2$ are strict transforms of $L^1$ and $C^1$ respectively; let $\theta:S^2\to\overline{S}$ be  the~contraction of  the~curve $E_1^2$,
and $\sigma$ is the~birational contraction of $\overline{E}=\theta(E)$. We denote the strict transforms of $C$ and $L$ on $\overline{S}$ by $\overline{C}$ and $\overline{L}$. Let $\overline{E}$ be the exceptional divisor of $\sigma$. Note that $\overline{E}$ contains one singular point $\overline{P}$  which is $\frac{1}{2}(1,1)$  singular point.
\begin{center}
 \includegraphics[width=12cm]{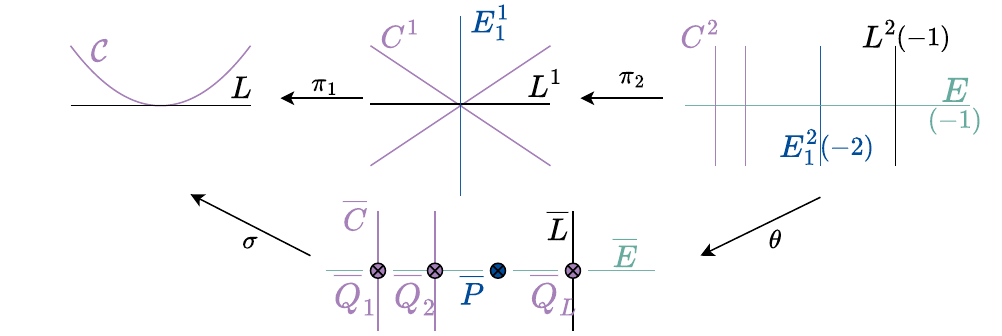}
 \end{center}
In both cases the intersections on $\overline{S}$ are given by:
\begin{center}
    \renewcommand{\arraystretch}{1.4}
    \begin{tabular}{|c|c|c|}
    \hline
         & $\overline{E}$ & $\overline{L}$ \\
    \hline
       $\overline{E}$  & $-\frac{1}{2}$ & $1$ \\
    \hline
       $\overline{L}$  & $1$ & $-1$ \\
    \hline
    \end{tabular}
\end{center}
 We have $\sigma^*(L)=\overline{L}+2\overline{E}$, $\sigma^*(C)=\overline{C}+4\overline{E}$, $\sigma^*(K_{\DP^2})=K_{\overline{S}}-2\overline{E}$. Thus, $A_{(\DP^2,\lambda C)}(\overline{E})=3-4 \lambda $.
\noindent The Zariski decomposition of the divisor  $-\sigma^*(K_{\DP^2}+\lambda C)-v\overline{E}$ is given by:
\begin{align*}
&&P(v)=
\begin{cases}
-\sigma^*(K_{\DP^2}+\lambda C)-v\overline{E}\text{ if }v\in[0,3-d \lambda ],\\
-\sigma^*(K_{\DP^2}+\lambda C)-v\overline{E}-\big(v-(3-d \lambda )\big)\overline{L}\text{ if }v\in[3-d \lambda , 2(3-d \lambda )],\\
\end{cases}\\&&
N(v)=
\begin{cases}
0\text{ if }v\in[0,3-d \lambda ],\\
\big(v-(3-d \lambda )\big)\overline{L}\text{ if }v\in[3-d \lambda ,2(3-d \lambda )].
\end{cases}
\end{align*}
Then
$$
P(v)^2=
\begin{cases}
(3-d \lambda )^2 -\frac{v^2}{2}\text{ if }v\in[0,3-d \lambda ],\\
\frac{(v- 2(3- 4\lambda) )^2}{2}\text{ if }v\in[3-d \lambda ,2(3-d \lambda )],
\end{cases}
\text{ and }
P(v)\cdot \overline{E}=
\begin{cases}
\frac{v}{2}\text{ if }v\in[0,3-d \lambda ],\\
3-d \lambda  -\frac{v}{2}\text{ if }v\in[3-d \lambda ,2(3-d \lambda )],
\end{cases}
$$
Thus
$$S_{(\DP^2, \lambda C)}(\overline{E})=\frac{1}{(3-d \lambda )^2}\Big(\int_0^{3-d \lambda } (3-d \lambda )^2 -\frac{v^2}{2} dv+\int_{3-d \lambda }^{2(3-d \lambda )} \frac{(v- 2(3- d\lambda))^2}{2} dv\Big)=3-d \lambda $$
so that $\delta_P(\DP^2,\lambda C)\le \frac{3-4 \lambda }{3-d \lambda }$. For every $O\in \overline{E}$, we get if $O\in \overline{E}\backslash \overline{L}$ or if $O\in \overline{E}\cap \overline{L}$:
{\small $$h(v)\le 
\begin{cases}
\frac{v^2}{8}\text{ if }v\in[0,3-d \lambda ],\\
\frac{(v- 2(3 - d\lambda))^2}{8}\text{ if }v\in[3-d \lambda , 2(3-d \lambda )],
\end{cases}
\text{or }
h(v)\le \begin{cases}
\frac{v^2}{8}\text{ if }v\in[0,3-d \lambda ],\\
\frac{ (v- 2(3 - d\lambda)) (2(3 - d\lambda) - 3v)}{8}\text{ if }v\in[3-d \lambda , 2(3-d \lambda )].
\end{cases}
$$}
So that
$$S\big(W^{\overline{E}}_{\bullet,\bullet};O\big)\le \frac{2}{(3-d \lambda )^2}\Big(\int_0^{3-d \lambda } \frac{v^2}{8} dv+\int_{3-d \lambda }^{2(3-d \lambda )} \frac{(v- 2(3 - 4\lambda))^2}{8} dv\Big)=\frac{3-d \lambda }{6}
$$
or
$$S\big(W^{\overline{E}}_{\bullet,\bullet};O\big)\le \frac{2}{(3-d \lambda )^2}\Big(\int_0^{3-d \lambda } \frac{v^2}{8} dv+\int_{3-d \lambda }^{2(3-d \lambda )} \frac{ (v- 6 + 8\lambda) (6 - 8\lambda - 3v)}{8} dv\Big)=\frac{3-d \lambda }{3}
$$
We have
$$
\delta_P(\DP^2,\lambda C)\geqslant\mathrm{min}\Bigg\{\frac{3-4 \lambda }{3-d \lambda },\inf_{O\in\overline{E}}\frac{A_{\overline{E},\Delta_{\overline{E}}}(O)}{S\big(W^{\overline{E}}_{\bullet,\bullet};O\big)}\Bigg\},
$$
where $\Delta_{\overline{E}}=\frac{1}{2}\overline{P}+\lambda\overline{Q}_1+\lambda\overline{Q}_2$ where $\overline{Q}_1+\overline{Q}_2=\overline{C}|_{\overline{E}}$  (resp. $\Delta_{\overline{E}}=\frac{1}{2}\overline{P}+\lambda\overline{Q}_1+\lambda\overline{Q}_2+ \lambda\overline{Q}_L$ where $\overline{Q}_1+\overline{Q}_2+\overline{Q}_L=\overline{C}|_{\overline{E}}$ ).
So that
$$
\frac{A_{\overline{E},\Delta_{\overline{E}}}(O)}{S(W_{\bullet,\bullet}^{\overline{E}};O)}=
\left\{\aligned
&\frac{3}{3-d \lambda }\ \mathrm{if}\ O=\overline{E}\cap\overline{L} \text{ or if } O=\overline{P},\\
&\frac{6(1-\lambda)}{3-d \lambda }\ \mathrm{if}\  O\in\{\overline{Q}_1,\overline{Q}_2\}\text{ (resp. $O\in\{\overline{Q}_1,\overline{Q}_2,\overline{Q}_L\}$)},\\
&\frac{6}{3-d \lambda }\ \mathrm{otherwise}.
\endaligned
\right.
$$
Thus $\delta_P(\DP^2,\lambda C)=\frac{3-4 \lambda }{3-d \lambda }$.
 \end{proof}
   \noindent This proves the following theorem:
    \begin{theorem}
  Suppose $C_3\subset \DP^2$ is a cubic curve with  $A_3$ singularity. Then  for $\lambda\in\big[0,\frac{3}{4}\big]$ we have:
 $$\delta(\DP^2,\lambda C_3)=\frac{3-4 \lambda }{3-3 \lambda }.$$
 \end{theorem}
 \begin{theorem}
  Suppose $C_4\subset \DP^2$ is a quartic curve with at most $A_3$ singularities and at least one $A_3$ singularity. Then  for $\lambda\in\big[0,\frac{3}{4}\big]$ we have:
 $$\delta(\DP^2,\lambda C_4)=1.$$
 \end{theorem}
\section{Curves with $A_4$ singularities}
 \begin{lemma}
 Suppose $C\subset \DP^2$ is a curve of degree $d$ with $A_4$ singularity at point $P\in C$  on $C$. Then
 $$\delta_P(\DP^2,\lambda C)=\frac{6}{13}\cdot \frac{7-10\lambda}{3-d\lambda}\text{ for }\lambda\in \Big[\frac{3}{8}, \frac{7}{10}\Big]\text{ and }\delta_P(\DP^2,\lambda C)\ge\frac{3}{2(3-d\lambda)}\text{ for }\lambda\in \Big[0,\frac{3}{8}\Big].$$
 \end{lemma}
  \begin{proof}
Let $L$ be a tangent line at point $P$.  Let $\pi_1:S^1\to S$ be the~blow up of the~point $P$,
with the exceptional divisor $E_1^1$ and $L^1$, $C^1$ are strict transforms of $L$ and $C$ respectively; $\pi_2: S^2\to S^1$ be the~blow up of the~point $E_1^1\cap C^1\cap L^1$,
with the exceptional divisor $E_2^2$ and $L^2$, $C^2$, $E_1^2$ are strict transforms of $L^1$, $C^1$, $E_1^1$ respectively; $\pi_3: S^3\to S^2$ be the~blow up of the~point $E_2^2\cap C^2$,
with the exceptional divisor $E_3^3$ and $L^3$, $C^3$, $E_1^3$, $E_2^3$ are strict transforms of $L^2$, $C^2$, $E_1^2$, $E_2^2$ respectively; $\pi_4: S^4\to S^3$ be the~blow up of the~point $E_3^3\cap E_2^2\cap C^3$,
with the exceptional divisor $E_3^3$ and $L^4$, $C^4$, $E_1^4$, $E_2^4$, $E_3^4$ are strict transforms of $L^3$, $C^3$, $E_1^3$, $E_2^3$, $E_3^2$ respectively; let $\theta:S^4\to\overline{S}$ be  the~contraction of  the~curves  $E_1^4$, $E_2^4$, $E_3^4$
and $\sigma$ is the~birational contraction of $\overline{E}=\theta(E)$. We denote the strict transforms of $C$ and $L$ on $\overline{S}$ by $\overline{C}$ and $\overline{L}$. Let $\overline{E}$ be the exceptional divisor of $\sigma$. Note that $\overline{E}$ contains two singular points $\overline{P}_{1,2}$ and  $\overline{P}_3$   which are $\frac{1}{5}(1,2)$ and  $\frac{1}{2}(1,1)$ singular points respectively.
\begin{center}
 \includegraphics[width=12cm]{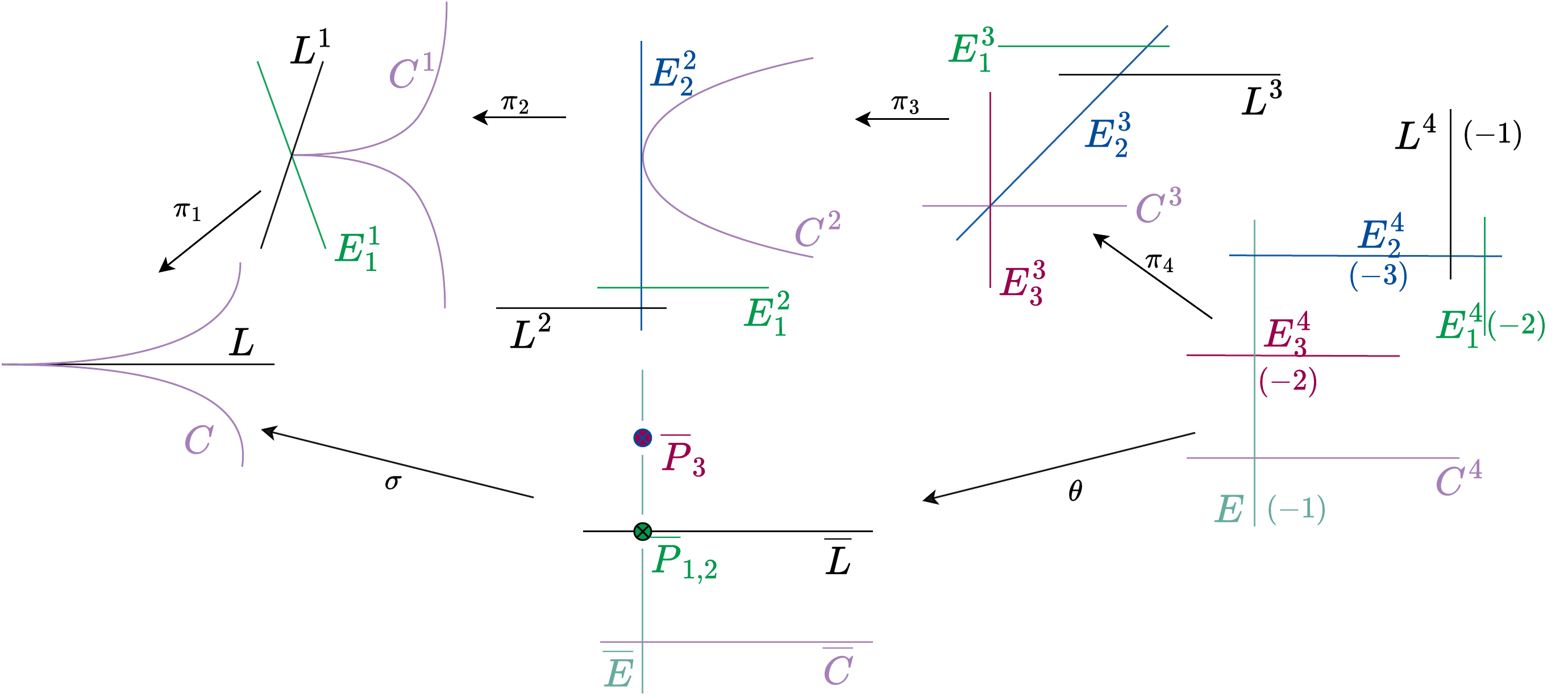}
 \end{center}
The intersections on $\overline{S}$ are given by:
\begin{center}
    \renewcommand{\arraystretch}{1.4}
    \begin{tabular}{|c|c|c|}
    \hline
         & $\overline{E}$ & $\overline{L}$ \\
    \hline
       $\overline{E}$  & $-\frac{1}{10}$ & $\frac{2}{5}$ \\
    \hline
       $\overline{L}$  & $\frac{2}{5}$ & $-\frac{3}{5}$ \\
    \hline
    \end{tabular}
\end{center}
 We have $\sigma^*(L)=\overline{L}+4\overline{E}$, $\sigma^*(C)=\overline{C}+10\overline{E}$, $\sigma^*(K_{\DP^2})=K_{\overline{S}}-6\overline{E}$. Thus, $A_{(\DP^2,\lambda C)}(\overline{E})=7-10\lambda$.
\noindent The Zariski decomposition of the divisor  $-\sigma^*(K_{\DP^2}+\lambda C)-v\overline{E}$ is given by:
\begin{align*}
&&P(v)=
\begin{cases}
-\sigma^*(K_{\DP^2}+\lambda C)-v\overline{E}\text{ if }v\in\big[0,\frac{5}{2}(3-d\lambda)\big],\\
-\sigma^*(K_{\DP^2}+\lambda C)-v\overline{E}-\frac{2}{3}\big(v-\frac{5}{2}(3-d\lambda)\big)\overline{L}\text{ if }v\in\big[\frac{5}{2}(3-d\lambda), 4(3-d\lambda)\big],\\
\end{cases}\\&&
N(v)=
\begin{cases}
0\text{ if }v\in\big[0,\frac{5}{2}(3-d\lambda)\big],\\
\frac{2}{3}\big(v-\frac{5}{2}(3-d\lambda)\big)\overline{L}\text{ if }v\in\big[\frac{5}{2}(3-d\lambda),4(3-d\lambda)\big].
\end{cases}
\end{align*}
Then
$$
P(v)^2=
\begin{cases}
(3-d\lambda)^2 -\frac{v^2}{10}\text{ if }v\in\big[0,\frac{5}{2}(3-d\lambda)\big],\\
\frac{(v - 12 +4d\lambda)^2}{6}\text{ if }v\in\big[\frac{5}{2}(3-d\lambda),4(3-d\lambda)\big],
\end{cases}
\text{ and }
P(v)\cdot \overline{E}=
\begin{cases}
\frac{v}{10}\text{ if }v\in\big[0,\frac{5}{2}(3-d\lambda)\big],\\
\frac{2}{3}\big(3-d\lambda -\frac{v}{4}\big)\text{ if }v\in\big[\frac{5}{2}(3-d\lambda),4(3-d\lambda)\big],
\end{cases}
$$
Thus
$$S_S(\overline{E})=\frac{1}{(3-d\lambda)^2}\Big(\int_0^{5/2(3-d\lambda)} (3-d\lambda)^2 -\frac{v^2}{10} dv+\int_{5/2(3-d\lambda)}^{4(3-d\lambda)} \frac{(v - 12 +16\lambda)^2}{6} dv\Big)=\frac{13(3-d\lambda)}{6}$$
so that $\delta_P(\DP^2,\lambda C)\le \frac{6}{13}\frac{7-10\lambda}{3-d\lambda}$. For every $O\in \overline{E}$, we get if $O\in \overline{E}\backslash \overline{L}$ or if $O\in \overline{E}\cap \overline{L}$:
$$h(v)\le 
\begin{cases}
\frac{v^2}{200}\text{ if }v\in\big[0,\frac{5}{2}(3-d\lambda)\big],\\
\frac{(v - 12 +4d\lambda)^2}{72}\text{ if }v\in\big[\frac{5}{2}(3-d\lambda),4(3-d\lambda)\big],
\end{cases}
\text{or }
h(v)\le \begin{cases}
\frac{v^2}{200}\text{ if }v\in\big[0,\frac{5}{2}(3-d\lambda)\big],\\
\frac{ (v - 12 +4d\lambda) (60 -11v - 20d\lambda )}{360}\text{ if }v\in\big[\frac{5}{2}(3-d\lambda),4(3-d\lambda)\big].
\end{cases}
$$
So that
$$S\big(W^{\overline{E}}_{\bullet,\bullet};O\big)\le \frac{2}{(3-d\lambda)^2}\Big(\int_0^{5/2(3-d\lambda)} \frac{v^2}{200} dv+\int_{5/2(3-d\lambda)}^{4(3-d\lambda)} \frac{(v - 12 +4d\lambda)^2}{72} dv\Big)=\frac{3-d\lambda}{12}
$$
or
$$S\big(W^{\overline{E}}_{\bullet,\bullet};O\big)\le \frac{2}{(3-d\lambda)^2}\Big(\int_0^{5/2(3-d\lambda)} \frac{v^2}{200} dv+\int_{5/2(3-d\lambda)}^{4(3-d\lambda)} \frac{ (v - 12 +4d\lambda) (60 -11v - 20d\lambda )}{360} dv\Big)=\frac{2(3-d\lambda)}{15}
$$
We have
$$
\delta_P(\DP^2,\lambda C)\geqslant\mathrm{min}\Bigg\{\frac{6}{13}\cdot \frac{7-10\lambda}{3-d\lambda},\inf_{O\in\overline{E}}\frac{A_{\overline{E},\Delta_{\overline{E}}}(O)}{S\big(W^{\overline{E}}_{\bullet,\bullet};O\big)}\Bigg\},
$$
where $\Delta_{\overline{E}}=\frac{4}{5}\overline{P}_{1,2}+\frac{1}{2}\overline{P}_{3}$. 
So that
$$
\frac{A_{\overline{E},\Delta_{\overline{E}}}(O)}{S(W_{\bullet,\bullet}^{\overline{E}};O)}=
\left\{\aligned
&\frac{3}{2(3-d\lambda)}\ \mathrm{if}\ O=\overline{P}_{1,2},\\
&\frac{6}{3-d\lambda}\ \mathrm{if}\ O=\overline{P}_{3},\\
&\frac{12}{3-d\lambda}\ \mathrm{otherwise}.
\endaligned
\right.
$$
Thus $$\delta_P(\DP^2,\lambda C)=\frac{6}{13}\cdot \frac{7-10\lambda}{3-d\lambda}\text{ for }\lambda\in \Big[\frac{3}{8}, \frac{7}{10}\Big]\text{ and }\delta_P(\DP^2,\lambda C)\ge\frac{3}{2(3-d\lambda)}\text{ for }\lambda\in \Big[0,\frac{3}{8}\Big].$$
 \end{proof}
 \noindent This proves the following theorem:
 \begin{theorem}
   Suppose $C\subset \DP^2$ is a curve of degree  with at most $A_4$ singularities and at least one $A_4$ singularity. Then  we have:
$$\delta(\DP^2,\lambda C)=\frac{6}{13}\cdot \frac{7-10\lambda}{3-4\lambda}\text{ for }\lambda\in \Big[\frac{3}{8}, \frac{7}{10}\Big]\text{ and }\delta(\DP^2,\lambda C)\ge \frac{3}{2(3-4\lambda)} \text{ for }\lambda\in \Big[0,\frac{3}{8}\Big].$$
 \end{theorem}
  \section{Curves with $A_5$ singularities}
   \begin{lemma}
 Let $C\subset \DP^2$ be a plane curve of degree $d$ with an $A_5$ singularity at point $P\in C$  on $C$, such that the tangent line at point $P$ to $C$ is not a component of $C$. Then
$$\delta_P(\DP^2,\lambda C)=\frac{6}{7}\cdot \frac{4-6\lambda}{3-d \lambda }\text{ for }\lambda\in \Big[\frac{3}{8}, \min \Big\{\frac{2}{3}, \frac{3}{d}\Big\}\Big]\text{ and }\delta_P(\DP^2,\lambda C)\ge\frac{3}{2(3-d \lambda )}\text{ for }\lambda\in \Big[0,\frac{3}{8}\Big].$$
 \end{lemma}
 \begin{proof}
Let $L$ be a tangent line at point $P$.  Let $\pi_1:S^1\to S$ be the~blow up of the~point $P$,
with the exceptional divisor $E_1^1$ and $L^1$, $C^1$ are strict transforms of $L$ and $C$ respectively; $\pi_2: S^2\to S^1$ be the~blow up of the~point $E_1^1\cap C^1$,
with the exceptional divisor $E_2^2$ and $L^2$, $C^2$, $E_1^2$ are strict transforms of $L^1$, $C^1$, $E_1^1$ respectively; $\pi_3: S^3\to S^2$ be the~blow up of the~point  $E_2^2\cap C^2$,
with the exceptional divisor $E$ and $L^3$, $C^3$, $E_1^3$, $E_2^3$ are strict transforms of $L^2$, $C^2$, $E_1^2$, $E_2^2$ respectively; let $\theta:S^3\to\overline{S}$ be  the~contraction of  the~curves  $E_1^3$, $E_2^3$
and $\sigma$ is the~birational contraction of $\overline{E}=\theta(E)$. We denote the strict transforms of $C$ and $L$ on $\overline{S}$ by $\overline{C}$ and $\overline{L}$. Let $\overline{E}$ be the exceptional divisor of $\sigma$. Note that $\overline{E}$ contains one singular point $\overline{P}$    which is a $\frac{1}{3}(1,2)$ singular point.
\begin{center}
 \includegraphics[width=16cm]{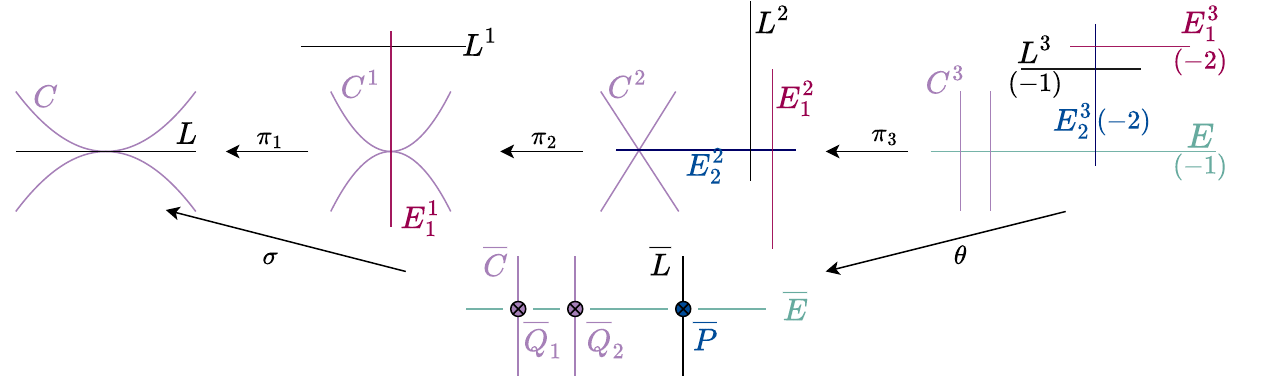}
 \end{center}
The intersections on $\overline{S}$ are given by:
\begin{center}
    \renewcommand{\arraystretch}{1.4}
    \begin{tabular}{|c|c|c|}
    \hline
         & $\overline{E}$ & $\overline{L}$ \\
    \hline
       $\overline{E}$  & $-\frac{1}{3}$ & $\frac{2}{3}$ \\
    \hline
       $\overline{L}$  & $\frac{2}{3}$ & $-\frac{1}{3}$ \\
    \hline
    \end{tabular}
\end{center}
 We have $\sigma^*(L)=\overline{L}+2\overline{E}$, $\sigma^*(C)=\overline{C}+6\overline{E}$, $\sigma^*(K_{\DP^2})=K_{\overline{S}}-3\overline{E}$. Thus, $A_{(\DP^2,\lambda C)}(\overline{E})=4-6\lambda$.
\noindent The Zariski decomposition of the divisor  $-\sigma^*(K_{\DP^2}+\lambda C)-v\overline{E}$ is given by:
\noindent The Zariski decomposition of the divisor  $-\sigma^*(K_{\DP^2}+\lambda C)-v\overline{E}$ is given by:
\begin{align*}
&&P(v)=
\begin{cases}
-\sigma^*(K_{\DP^2}+\lambda C)-v\overline{E}\text{ if }v\in\big[0,\frac{3}{2}(3-d \lambda )\big],\\
-\sigma^*(K_{\DP^2}+\lambda C)-v\overline{E}-2\big(v-\frac{3}{2}(3-d \lambda )\big)\overline{L}\text{ if }v\in\big[\frac{3}{2}(3-d \lambda ), 2(3-d \lambda )\big],\\
\end{cases}\\&&
N(v)=
\begin{cases}
0\text{ if }v\in\big[0,\frac{3}{2}(3-d \lambda )\big],\\
2\big(v-\frac{3}{2}(3-d \lambda )\big)\overline{L}\text{ if }v\in\big[\frac{3}{2}(3-d \lambda ),2(3-d \lambda )\big].
\end{cases}
\end{align*}
Then
$$
P(v)^2=
\begin{cases}
(3-d \lambda )^2 -\frac{v^2}{3}\text{ if }v\in\big[0,\frac{3}{2}(3-d \lambda )\big],\\
(v - 6 +8\lambda )^2\text{ if }v\in\big[\frac{3}{2}(3-d \lambda ),2(3-d \lambda )\big],
\end{cases}$$
\text{ and }$$
P(v)\cdot \overline{E}=
\begin{cases}
\frac{v}{3}\text{ if }v\in\big[0,\frac{3}{2}(3-d \lambda )\big],\\
\frac{1}{2}(3-d \lambda  -\frac{v}{3})\text{ if }v\in\big[\frac{3}{2}(3-d \lambda ),2(3-d \lambda )\big],
\end{cases}
$$
Thus
$$S_{(\DP^2, \lambda C)}(\overline{E})=\frac{1}{(3-d \lambda )^2}\Big(\int_0^{3/2(3-d \lambda )} (3-d \lambda )^2 -\frac{v^2}{3} dv+\int_{3/2(3-d \lambda )}^{2(3-d \lambda )} (v - 6 +8\lambda)^2  dv\Big)=\frac{7(3-d \lambda )}{6}$$
so that $\delta_P(\DP^2,\lambda C)\le \frac{6}{7}\frac{4-6\lambda}{3-d \lambda }$. For every $O\in \overline{E}$, we get if $O\in \overline{E}\backslash \overline{L}$ or if $O\in \overline{E}\cap \overline{L}$:
\begin{align*}
&&h(v)\le 
\begin{cases}
\frac{v^2}{18}\text{ if }v\in\big[0,\frac{3}{2}(3-d \lambda )\big],\\
\frac{(v - 2(3-d \lambda ))^2}{2}\text{ if }v\in\big[\frac{3}{2}(3-d \lambda ),2(3-d \lambda )\big],
\end{cases}\\
&\text{or}&
h(v)\le \begin{cases}
\frac{v^2}{18}\text{ if }v\in\big[0,\frac{3}{2}(3-d \lambda )\big],\\
\frac{ (v - 2(3-d \lambda ) ) (6(3-d \lambda ) -5v}{6}\text{ if }v\in\big[\frac{3}{2}(3-d \lambda ),2(3-d \lambda )\big].
\end{cases}
\end{align*}
So that
$$S\big(W^{\overline{E}}_{\bullet,\bullet};O\big)\le \frac{2}{(3-d \lambda )^2}\Big(\int_0^{3/2(3-d \lambda )} \frac{v^2}{18} dv+\int_{3/2(3-d \lambda )}^{2(3-d \lambda )} \frac{(v - 2(3-d \lambda ))^2}{72} dv\Big)=\frac{3-d \lambda }{6}
$$
or
$$S\big(W^{\overline{E}}_{\bullet,\bullet};O\big)\le \frac{2}{(3-d \lambda )^2}\Big(\int_0^{3/2(3-d \lambda )} \frac{v^2}{18} dv+\int_{3/2(3-d \lambda )}^{2(3-d \lambda )} \frac{ (v - 2(3-d \lambda )) (6(3-d \lambda ) -5v )}{6} dv\Big)=\frac{2(3-d \lambda )}{9}
$$
We have
$$
\delta_P(\DP^2,\lambda C)\geqslant\mathrm{min}\Bigg\{\frac{6}{7}\cdot\frac{4-6\lambda}{3-d \lambda } ,\inf_{O\in\overline{E}}\frac{A_{\overline{E},\Delta_{\overline{E}}}(O)}{S\big(W^{\overline{E}}_{\bullet,\bullet};O\big)}\Bigg\},
$$
where $\Delta_{\overline{E}}=\frac{2}{3}\overline{P}+\lambda\overline{Q}_1+\lambda\overline{Q}_2$ where $\overline{Q}_1+\overline{Q}_2=\overline{C}|_{\overline{E}}$. 
So that
$$
\frac{A_{\overline{E},\Delta_{\overline{E}}}(O)}{S(W_{\bullet,\bullet}^{\overline{E}};O)}=
\left\{\aligned
&\frac{3}{2(3-d \lambda )}\ \mathrm{if}\ O=\overline{P},\\
&\frac{6(1-\lambda)}{3-d \lambda }\ \mathrm{if}\ O\in\{\overline{Q}_1,\overline{Q}_2\},\\
&\frac{6}{3-d \lambda }\ \mathrm{otherwise}.
\endaligned
\right.
$$
Thus $$\delta_P(\DP^2,\lambda C)=\frac{6}{7}\cdot \frac{4-6\lambda}{3-d \lambda }\text{ for }\lambda\in \Big[\frac{3}{8}, \min \Big\{\frac{2}{3}, \frac{3}{d}\Big\}\Big]\text{ and }\delta_P(\DP^2,\lambda C)\ge\frac{3}{2(3-d \lambda )}\text{ for }\lambda\in \Big[0,\frac{3}{8}\Big].$$
 \end{proof}
 \noindent This proves the following theorem:
 \begin{theorem}
   Suppose $C_4\subset \DP^2$ is a quartic curve with  $A_5$ singularity such that the tangent line at
point $P$ to $C_4$ is not a component of $C_4$. Then  we have:
$$\delta_P(\DP^2,\lambda C_4)=\frac{6}{7}\cdot\frac{4-6\lambda}{3-d \lambda }\text{ for }\lambda\in \Big[\frac{3}{8}, \frac{2}{3}\Big]\text{ and }\delta_P(\DP^2,\lambda C_4)\ge\frac{3}{2(3-d \lambda )}\text{ for }\lambda\in \Big[0,\frac{3}{8}\Big].$$
 \end{theorem}

   \begin{lemma}
 Let $C\subset \DP^2$ be a quartic curve with $A_5$ singularity at point $P\in C$  on $C$, such that the tangent line at point $P$ to $C$ is a component of $C$. Then
$$\delta_P(\DP^2,\lambda C)=\frac{3}{4}\cdot\frac{4-6\lambda}{3-d \lambda }\text{ for }\lambda\in \Big[0,\min \Big\{\frac{2}{3}, \frac{3}{d}\Big\}\Big].$$
 \end{lemma}
 \begin{proof}
Let $L$ be a tangent line at point $P$ which is a component of $C$ and $\mathcal{C}\cup L=C$.  Let $\pi_1:S^1\to S$ be the~blow up of the~point $P$,
with the exceptional divisor $E_1^1$ and $L^1$, $C^1$ are strict transforms of $L$ and $C$ respectively; $\pi_2: S^2\to S^1$ be the~blow up of the~point $E_1^1\cap C^1\cap L^1$,
with the exceptional divisor $E_2^2$ and $L^2$, $C^2$, $E_1^2$ are strict transforms of $L^1$, $C^1$, $E_1^1$ respectively; $\pi_3: S^3\to S^2$ be the~blow up of the~point  $E_2^2\cap C^2\cap L^2$,
with the exceptional divisor $E$ and $L^3$, $C^3$, $E_1^3$, $E_2^3$ are strict transforms of $L^2$, $C^2$, $E_1^2$, $E_2^2$ respectively; let $\theta:S^3\to\overline{S}$ be  the~contraction of  the~curves  $E_1^3$, $E_2^3$
and $\sigma$ is the~birational contraction of $\overline{E}=\theta(E)$. We denote the strict transforms of $C$ and $L$ on $\overline{S}$ by $\overline{C}$ and $\overline{L}$. Let $\overline{E}$ be the exceptional divisor of $\sigma$. Note that $\overline{E}$ contains one singular point $\overline{P}$    which is a $\frac{1}{3}(1,2)$ singular point.
\begin{center}
 \includegraphics[width=16cm]{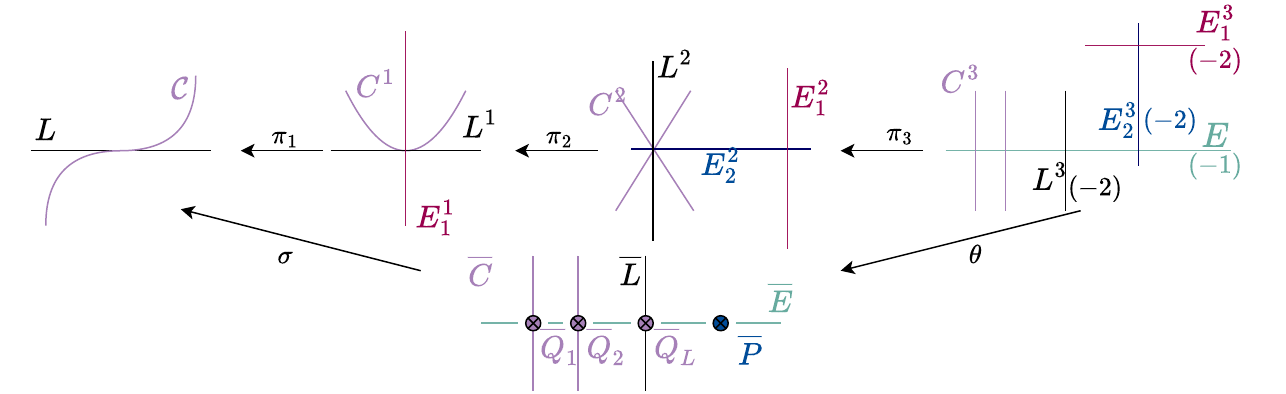}
 \end{center}
The intersections on $\overline{S}$ are given by:
\begin{center}
    \renewcommand{\arraystretch}{1.4}
    \begin{tabular}{|c|c|c|}
    \hline
         & $\overline{E}$ & $\overline{L}$ \\
    \hline
       $\overline{E}$  & $-\frac{1}{3}$ & $1$ \\
    \hline
       $\overline{L}$  & $1$ & $-2$ \\
    \hline
    \end{tabular}
\end{center}
 We have $\sigma^*(L)=\overline{L}+3\overline{E}$, $\sigma^*(C)=\overline{C}+\overline{L}+6\overline{E}$, $\sigma^*(K_{\DP^2})=K_{\overline{S}}-3\overline{E}$. Thus, $A_{(\DP^2,\lambda C)}(\overline{E})=4-6\lambda$.
 \noindent The Zariski decomposition of the divisor  $-\sigma^*(K_{\DP^2}+\lambda C)-v\overline{E}$ is given by:
\begin{align*}
&&P(v)=
\begin{cases}
-\sigma^*(K_{\DP^2}+\lambda C)-v\overline{E}\text{ if }v\in[0,3-d \lambda ],\\
-\sigma^*(K_{\DP^2}+\lambda C)-v\overline{E}-\frac{1}{2}\big(v-(3-d \lambda )\big)\overline{L}\text{ if }v\in[3-d \lambda , 3(3-d \lambda )],\\
\end{cases}\\&&
N(v)=
\begin{cases}
0\text{ if }v\in[0,3-d \lambda ],\\
\frac{1}{2}\big(v-(3-d \lambda )\big)\overline{L}\text{ if }v\in[3-d \lambda ,3(3-d \lambda )].
\end{cases}
\end{align*}
Then
$$
P(v)^2=
\begin{cases}
(3-d \lambda )^2 -\frac{v^2}{3}\text{ if }v\in[0,3-d \lambda ],\\
\frac{(v - 3(3-d \lambda ))^2}{6}\text{ if }v\in[3-d \lambda ,3(3-d \lambda )],
\end{cases}
\text{ and }
P(v)\cdot \overline{E}=
\begin{cases}
\frac{v}{3}\text{ if }v\in[0,3-d \lambda ],\\
\frac{1}{2}\big(3-d \lambda  -\frac{v}{3}\big)\text{ if }v\in[3-d \lambda ,3(3-d \lambda )],
\end{cases}
$$
Thus
$$S_{(\DP^2, \lambda C)}(\overline{E})=\frac{1}{(3-d \lambda )^2}\Big(\int_0^{3-d \lambda } (3-d \lambda )^2 -\frac{v^2}{3} dv+\int_{3-d \lambda }^{3(3-d \lambda )} \frac{(v - 3(3-d \lambda ))^2}{6} dv\Big)=\frac{4(3-d \lambda )}{3}$$
so that $\delta_P(\DP^2,\lambda C)\le \frac{3}{4}\cdot\frac{4-6\lambda}{3-d \lambda }$. For every $O\in \overline{E}$, we get if $O\in \overline{E}\backslash \overline{L}$ or if $O\in \overline{E}\cap \overline{L}$:
$$h(v)\le 
\begin{cases}
\frac{v^2}{18}\text{ if }v\in[0,3-d \lambda ],\\
\frac{(v - 9 +12\lambda)^2}{72}\text{ if }v\in[3-d \lambda , 3(3-d \lambda )],
\end{cases}
\text{or }
h(v)\le \begin{cases}
\frac{v^2}{18}\text{ if }v\in[0,3-d \lambda ],\\
\frac{ (v - 3(3-d \lambda )) (3(3-d \lambda ) -5v)}{72}\text{ if }v\in[3-d \lambda , 3(3-d \lambda )].
\end{cases}
$$
So that
$$S\big(W^{\overline{E}}_{\bullet,\bullet};O\big)\le \frac{2}{(3-d \lambda )^2}\Big(\int_0^{3-d \lambda } \frac{v^2}{18} dv+\int_{3-d \lambda }^{3(3-d \lambda )} \frac{(v - 3(3-d \lambda ))^2}{72} dv\Big)=\frac{3-d \lambda }{9}
$$
or
$$S\big(W^{\overline{E}}_{\bullet,\bullet};O\big)\le \frac{2}{(3-d \lambda )^2}\Big(\int_0^{3-d \lambda } \frac{v^2}{18} dv+\int_{3-d \lambda }^{3(3-d \lambda )} \frac{ (v - 3(3-d \lambda )) (3(3-d \lambda ) -5v )}{72} dv\Big)=\frac{3-d \lambda }{3}
$$
We have
$$
\delta_P(\DP^2,\lambda C)\geqslant\mathrm{min}\Bigg\{\frac{3}{4}\cdot\frac{4-6\lambda}{3-d \lambda },\inf_{O\in\overline{E}}\frac{A_{\overline{E},\Delta_{\overline{E}}}(O)}{S\big(W^{\overline{E}}_{\bullet,\bullet};O\big)}\Bigg\},
$$
where $\Delta_{\overline{E}}=\frac{2}{3}\overline{P}+\lambda\overline{Q}_1+\lambda\overline{Q}_2+\lambda\overline{Q}_L$ where $\overline{Q}_1+\overline{Q}_2+\overline{Q}_L=\overline{C}|_{\overline{E}}$.  
So that
$$
\frac{A_{\overline{E},\Delta_{\overline{E}}}(O)}{S(W_{\bullet,\bullet}^{\overline{E}};O)}=
\left\{\aligned
&\frac{3}{3-d \lambda }\ \mathrm{if}\ O=\overline{E}\cap\overline{L}\text{ or if } O=\overline{P},\\
&\frac{9(1-\lambda)}{3-d \lambda }\ \mathrm{if}\ O\in\{\overline{Q}_1,\overline{Q}_2,\overline{Q}_L\},\\
&\frac{9}{3-d \lambda }\ \mathrm{otherwise}.
\endaligned
\right.
$$
Thus $\delta_P(\DP^2,\lambda C)=\frac{3}{4}\cdot\frac{4-6\lambda}{3-d \lambda }$.
 \end{proof}
  \noindent This proves the following theorem:
 \begin{theorem}
   Suppose $C_4\subset \DP^2$ is a quartic curve with  $A_5$ singularity such that the tangent line at
point $P$ to $C_4$ is a component of $C_4$. Then  we have:
$$\delta(\DP^2,\lambda C_4)=\frac{3}{4}\cdot\frac{4-6\lambda}{3-4 \lambda }\text{ for }\lambda\in \Big[0, \frac{2}{3}\Big].$$
 \end{theorem}

  \section{Curves with $A_6$ singularities}
 \begin{lemma}
 Suppose $C\subset \DP^2$ is a  curve of degree $d$ with $A_6$ singularity at point $P\in C$  on $C$. Then
 $$\delta_P(\DP^2,\lambda C)=\frac{2}{5}\cdot\frac{9-14\lambda}{3-d\lambda}\text{ for }\lambda\in \Big[\frac{3}{8}, \frac{9}{14}\Big]\text{ and }\delta_P(\DP^2,\lambda C)\ge\frac{3}{2(3-d\lambda)}\text{ for }\lambda\in \Big[0,\frac{3}{8}\Big].$$
 \end{lemma}
  \begin{proof}
Let $L$ be a tangent line at point $P$.  Let $\pi_1:S^1\to S$ be the~blow up of the~point $P$,
with the exceptional divisor $E_1^1$ and $L^1$, $C^1$ are strict transforms of $L$ and $C$ respectively; $\pi_2: S^2\to S^1$ be the~blow up of the~point $E_1^1\cap C^1\cap L^1$,
with the exceptional divisor $E_2^2$ and $L^2$, $C^2$, $E_1^2$ are strict transforms of $L^1$, $C^1$, $E_1^1$ respectively; $\pi_3: S^3\to S^2$ be the~blow up of the~point $E_2^2\cap C^2$,
with the exceptional divisor $E_3^3$ and $L^3$, $C^3$, $E_1^3$, $E_2^3$ are strict transforms of $L^2$, $C^2$, $E_1^2$, $E_2^2$ respectively; $\pi_4: S^5\to S^3$ be the~blow up of the~point $E_3^3\cap  C^3$,
with the exceptional divisor $E_4^5$ and $L^4$, $C^4$, $E_1^4$, $E_2^4$, $E_3^4$ are strict transforms of $L^3$, $C^3$, $E_1^3$, $E_2^3$, $E_3^2$ respectively;  $\pi_5: S^4\to S^4$ be the~blow up of the~point $E_4^4\cap E_3^4\cap   C^4$,
with the exceptional divisor $E$ and $L^5$, $C^5$, $E_1^5$, $E_2^5$, $E_3^5$ , $E_4^5$ are strict transforms of $L^4$, $C^4$, $E_1^4$, $E_2^4$, $E_3^4$, $E_4^4$ respectively; let $\theta:S^5\to\overline{S}$ be  the~contraction of  the~curves  $E_1^5$, $E_2^5$, $E_3^5$, $E_4^5$,
and $\sigma$ is the~birational contraction of $\overline{E}=\theta(E)$. We denote the strict transforms of $C$ and $L$ on $\overline{S}$ by $\overline{C}$ and $\overline{L}$. Let $\overline{E}$ be the exceptional divisor of $\sigma$. Note that $\overline{E}$ contains two singular points $\overline{P}_{1,2,3}$ and  $\overline{P}_4$   which are $\frac{1}{7}(1,3)$ and  $\frac{1}{2}(1,1)$ singular points respectively.
\begin{center}
 \includegraphics[width=16cm]{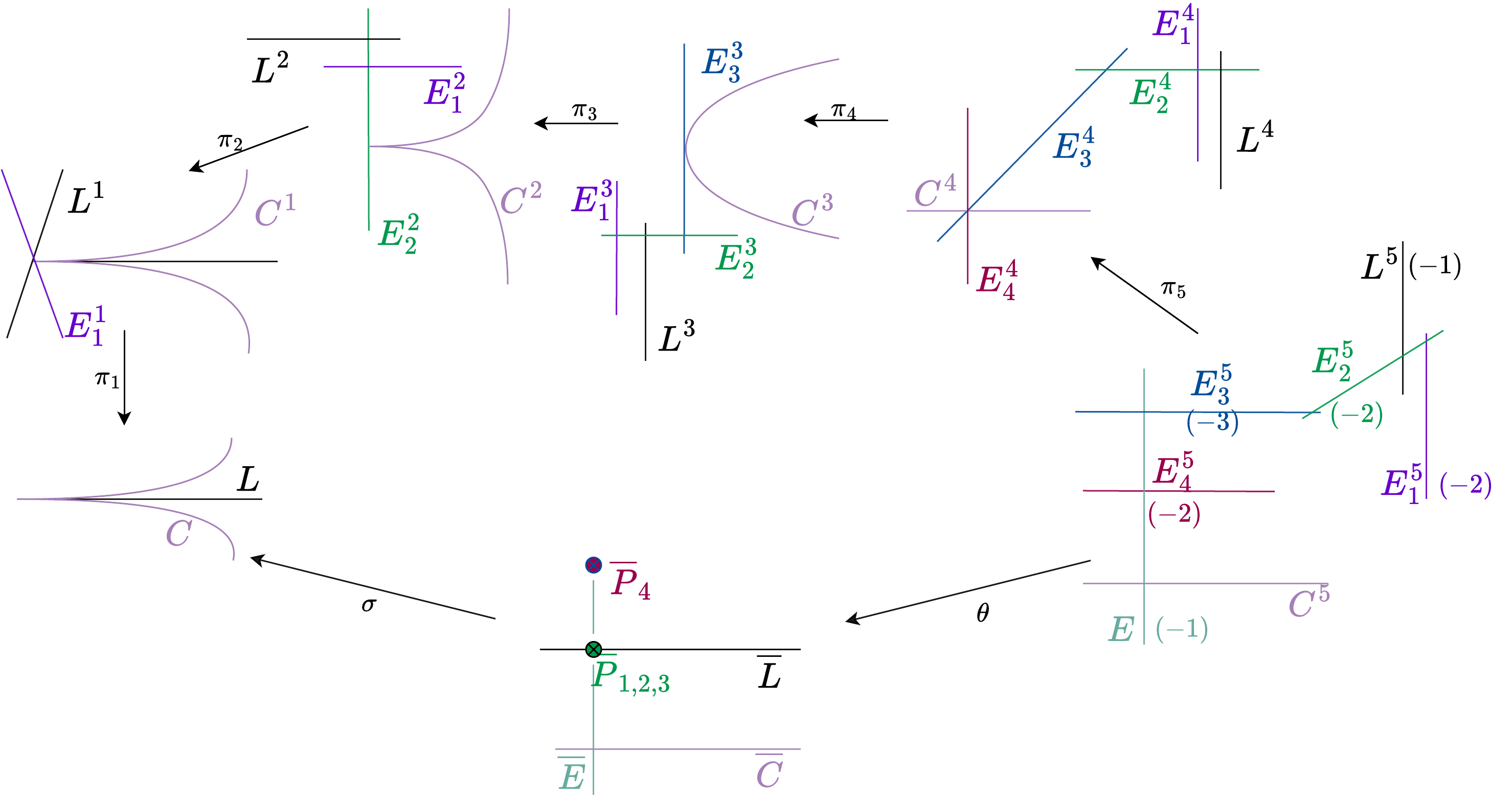}
 \end{center}
The intersections on $\overline{S}$ are given by:
\begin{center}
    \renewcommand{\arraystretch}{1.4}
    \begin{tabular}{|c|c|c|}
    \hline
         & $\overline{E}$ & $\overline{L}$ \\
    \hline
       $\overline{E}$  & $-\frac{1}{14}$ & $\frac{2}{7}$ \\
    \hline
       $\overline{L}$  & $\frac{2}{7}$ & $-\frac{1}{7}$ \\
    \hline
    \end{tabular}
\end{center}
 We have $\sigma^*(L)=\overline{L}+4\overline{E}$, $\sigma^*(C)=\overline{C}+14\overline{E}$, $\sigma^*(K_{\DP^2})=K_{\overline{S}}-8\overline{E}$. Thus, $A_{(\DP^2,\lambda C)}(\overline{E})=9-14\lambda$.
\noindent The Zariski decomposition of the divisor  $-\sigma^*(K_{\DP^2}+\lambda C)-v\overline{E}$ is given by:
\begin{align*}
&&P(v)=
\begin{cases}
-\sigma^*(K_{\DP^2}+\lambda C)-v\overline{E}\text{ if }v\in\big[0,\frac{7}{2}(3-d\lambda)\big],\\
-\sigma^*(K_{\DP^2}+\lambda C)-v\overline{E}-2\big(v-\frac{7}{2}(3-d\lambda)\big)\overline{L}\text{ if }v\in\big[\frac{7}{2}(3-d\lambda), 4(3-d\lambda)\big],\\
\end{cases}\\&&
N(v)=
\begin{cases}
0\text{ if }v\in\big[0,\frac{7}{2}(3-d\lambda)\big],\\
2\big(v-\frac{7}{2}(3-d\lambda)\big)\overline{L}\text{ if }v\in\big[\frac{7}{2}(3-d\lambda),4(3-d\lambda)\big].
\end{cases}
\end{align*}
Then
$$
P(v)^2=
\begin{cases}
(3-d\lambda)^2 -\frac{v^2}{14}\text{ if }v\in\big[0,\frac{7}{2}(3-d\lambda)\big],\\
\frac{(v - 12 +16\lambda)^2}{2}\text{ if }v\in\big[\frac{7}{2}(3-d\lambda),4(3-d\lambda)\big],
\end{cases}
\text{ and }
P(v)\cdot \overline{E}=
\begin{cases}
\frac{v}{14}\text{ if }v\in\big[0,\frac{7}{2}(3-d\lambda)\big],\\
2\big(3-d\lambda -\frac{v}{4}\big)\text{ if }v\in\big[\frac{7}{2}(3-d\lambda),4(3-d\lambda)\big],
\end{cases}
$$
Thus
$$S_S(\overline{E})=\frac{1}{(3-d\lambda)^2}\Big(\int_0^{7/2(3-d\lambda)} (3-d\lambda)^2 -\frac{v^2}{14} dv+\int_{7/2(3-d\lambda)}^{4(3-d\lambda)} \frac{(v - 12 +4d\lambda)^2}{6} dv\Big)=\frac{5(3-d\lambda)}{2}$$
so that $\delta_P(\DP^2,\lambda C)\le \frac{2}{5}\frac{9-14\lambda}{3-d\lambda}$. For every $O\in \overline{E}$, we get if $O\in \overline{E}\backslash \overline{L}$ or if $O\in \overline{E}\cap \overline{L}$:
$$h(v)\le 
\begin{cases}
\frac{v^2}{392}\text{ if }v\in\big[0,\frac{5}{2}(3-d\lambda)\big],\\
\frac{(v - 12 +4d\lambda)^2}{8}\text{ if }v\in\big[\frac{5}{2}(3-d\lambda),4(3-d\lambda)\big],
\end{cases}
\text{or }
h(v)\le \begin{cases}
\frac{v^2}{392}\text{ if }v\in\big[0,\frac{7}{2}(3-d\lambda)\big],\\
\frac{ (v - 12 +4d\lambda) (84 -9v - 28d\lambda )}{56}\text{ if }v\in\big[\frac{7}{2}(3-d\lambda),4(3-d\lambda)\big].
\end{cases}
$$
So that
$$S\big(W^{\overline{E}}_{\bullet,\bullet};O\big)\le \frac{2}{(3-d\lambda)^2}\Big(\int_0^{7/2(3-d\lambda)} \frac{v^2}{392} dv+\int_{7/2(3-d\lambda)}^{4(3-d\lambda)} \frac{(v - 12 +16\lambda)^2}{72} dv\Big)=\frac{3-d\lambda}{12}
$$
or
$$S\big(W^{\overline{E}}_{\bullet,\bullet};O\big)\le \frac{2}{(3-d\lambda)^2}\Big(\int_0^{7/2(3-d\lambda)} \frac{v^2}{392} dv+\int_{7/2(3-d\lambda)}^{4(3-d\lambda)} \frac{ (v - 12 +4d\lambda) (84 -9v - 28d\lambda )}{56} dv\Big)=\frac{2(3-d\lambda)}{21}
$$
We have
$$
\delta_P(\DP^2,\lambda C)\geqslant\mathrm{min}\Bigg\{\frac{2}{5}\cdot\frac{9-14\lambda}{3-d\lambda},\inf_{O\in\overline{E}}\frac{A_{\overline{E},\Delta_{\overline{E}}}(O)}{S\big(W^{\overline{E}}_{\bullet,\bullet};O\big)}\Bigg\},
$$
where $\Delta_{\overline{E}}=\frac{6}{7}\overline{P}_{1,2,3}+\frac{1}{2}\overline{P}_{4}$. 
So that
$$
\frac{A_{\overline{E},\Delta_{\overline{E}}}(O)}{S(W_{\bullet,\bullet}^{\overline{E}};O)}=
\left\{\aligned
&\frac{3}{2(3-d\lambda)}\ \mathrm{if}\ O=\overline{P}_{1,2,3},\\
&\frac{6}{3-d\lambda}\ \mathrm{if}\ O=\overline{P}_{4},\\
&\frac{12}{3-d\lambda}\ \mathrm{otherwise}.
\endaligned
\right.
$$
Thus $$\delta_P(\DP^2,\lambda C)=\frac{2}{5}\cdot\frac{9-14\lambda}{3-d\lambda}\text{ for }\lambda\in \Big[\frac{3}{8}, \frac{7}{10}\Big]\text{ and }\delta_P(\DP^2,\lambda C)\ge\frac{3}{2(3-d\lambda)}\text{ for }\lambda\in \Big[0,\frac{3}{8}\Big].$$
 \end{proof}
 \noindent This proves the following theorem:
 \begin{theorem}
   Suppose $C\subset \DP^2$ is a quartic curve with at most $A_6$ singularities and at least one $A_6$ singularity. Then  we have:
$$\delta(\DP^2,\lambda C)=\frac{2}{5}\cdot\frac{9-14\lambda}{3-4\lambda}\text{ for }\lambda\in \Big[\frac{3}{8}, \frac{9}{14}\Big]\text{ and }\delta(\DP^2,\lambda C)\ge \frac{3}{2(3-4\lambda)} \text{ for }\lambda\in \Big[0,\frac{3}{8}\Big].$$
 \end{theorem}

\section{Curves with $A_7$ singularities}
 \begin{lemma}
 Let $C\subset \DP^2$ be a plane curve of degree $d$ with an $A_7$ singularity at point $P\in C$  on $C$. Then
 $$\delta_P(\DP^2,\lambda C)=\frac{3}{4}\cdot\frac{5-8\lambda}{3-d \lambda }\text{ for }\lambda\in \Big[\frac{3}{8},\Big\{\frac{5}{8}, \frac{3}{d}\Big\}\Big].$$
 \end{lemma}
  \begin{proof}
Let $L$ be a tangent line at point $P$.  Let $\pi_1:S^1\to S$ be the~blow up of the~point $P$,
with the exceptional divisor $E_1^1$ and $L^1$, $C^1$ are strict transforms of $L$ and $C$ respectively; $\pi_2: S^2\to S^1$ be the~blow up of the~point $E_1^1\cap C^1\cap L^1$,
with the exceptional divisor $E_2^2$ and $L^2$, $C^2$, $E_1^2$ are strict transforms of $L^1$, $C^1$, $E_1^1$ respectively; $\pi_3: S^3\to S^2$ be the~blow up of the~point $E_2^2\cap C^2$,
with the exceptional divisor $E_3^3$ and $L^3$, $C^3$, $E_1^3$, $E_2^3$ are strict transforms of $L^2$, $C^2$, $E_1^2$, $E_2^2$ respectively; $\pi_4: S^4\to S^3$ be the~blow up of the~point $E_3^3\cap C^3$,
with the exceptional divisor $E$ and $L^4$, $C^4$, $E_1^4$, $E_2^4$, $E_3^4$ are strict transforms of $L^3$, $C^3$, $E_1^3$, $E_2^3$, $E_3^2$ respectively; let $\theta:S^4\to\overline{S}$ be  the~contraction of  the~curves  $E_1^4$, $E_2^4$, $E_3^4$
and $\sigma$ is the~birational contraction of $\overline{E}=\theta(E)$. We denote the strict transforms of $C$ and $L$ on $\overline{S}$ by $\overline{C}$ and $\overline{L}$. Let $\overline{E}$ be the exceptional divisor of $\sigma$. Note that $\overline{E}$ contains a singular point $\overline{P}$   which is a $\frac{1}{4}(1,3)$  singular point.
\begin{center}
 \includegraphics[width=17cm]{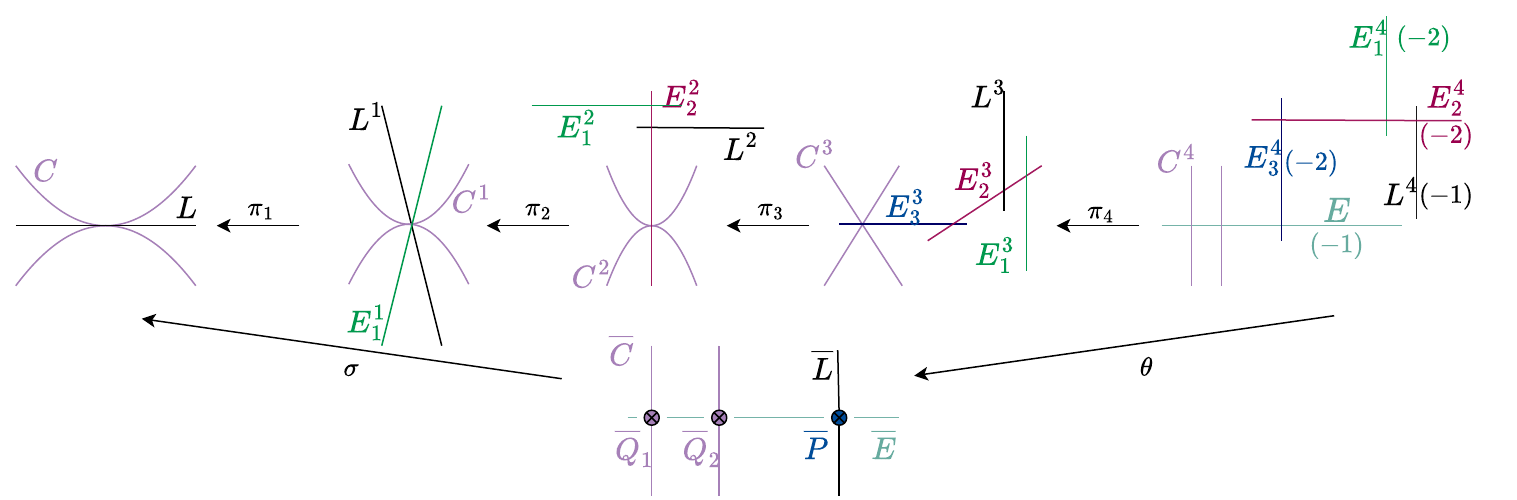}
 \end{center}
The intersections on $\overline{S}$ are given by:
\begin{center}
    \renewcommand{\arraystretch}{1.4}
    \begin{tabular}{|c|c|c|}
    \hline
         & $\overline{E}$ & $\overline{L}$ \\
    \hline
       $\overline{E}$  & $-\frac{1}{4}$ & $\frac{1}{2}$ \\
    \hline
       $\overline{L}$  & $\frac{1}{2}$ & $0$ \\
    \hline
    \end{tabular}
\end{center}
 We have $\sigma^*(L)=\overline{L}+2\overline{E}$, $\sigma^*(C)=\overline{C}+8\overline{E}$, $\sigma^*(K_{\DP^2})=K_{\overline{S}}-4\overline{E}$. Thus, $A_{(\DP^2,\lambda C)}(\overline{E})=5-8\lambda$.
\noindent The Zariski decomposition of the divisor  $-\pi^*(K_{\DP^2}+\lambda C)-vE$ is given by:
\begin{align*}
P(v)=-\pi^*(K_{\DP^2}+\lambda C)-vE \text{ and }N(v)=0\text{ if }v\in[0,2(3-d \lambda )].
\end{align*}
Then
$$
P(v)^2=\frac{(v-2(3-d \lambda ))(v+2(3-d \lambda ))}{4}{}\text{ and }P(v)\cdot E= \frac{v}{4}\text{ if }v\in[0,2(3-d \lambda )].
$$
Thus
$$S_{(\DP^2, \lambda C)}(E)=\frac{1}{(3-d \lambda )^2}\Big(\int_0^{2(3-d \lambda )} \frac{(v-2(3-d \lambda ))(v+2(3-d \lambda ))}{4} dv\Big)=\frac{4(3-d \lambda )}{3}$$
so that $\delta_P(\DP^2,\lambda C)\le \frac{3}{4}\cdot\frac{5-8\lambda}{3-d \lambda }$. For every $O\in E$ we get:
$$h(v) = \frac{v^2}{32}\text{ if }v\in[0,2(3-d \lambda )].
$$
So that
$$S\big(W^{E}_{\bullet,\bullet};O\big)= \frac{2}{(3-d \lambda )^2}\Big(\int_0^{3-d \lambda } \frac{v^2}{32} dv\Big)=\frac{3-d \lambda }{6} \le \frac{4(3-d \lambda )}{3(5-8\lambda)}\text{ for }\lambda\in \Big[\frac{3}{8},\Big\{\frac{5}{8}, \frac{3}{d}\Big\}\Big]
$$
where $\Delta_{\overline{E}}=\frac{3}{4}\overline{P}+\lambda\overline{Q}_1+\lambda\overline{Q}_2$ where $\overline{Q}_1+\overline{Q}_2=\overline{C}|_{\overline{E}}$.  
So that
$$
\frac{A_{\overline{E},\Delta_{\overline{E}}}(O)}{S(W_{\bullet,\bullet}^{\overline{E}};O)}=
\left\{\aligned
&\frac{6(1-\lambda)}{3-d \lambda }\ \mathrm{if}\ O\in\{\overline{Q}_1,\overline{Q}_2\},\\
&\frac{3}{2(3-d \lambda) }\ \mathrm{if}\ O=\overline{P},\\
&\frac{6}{3-d \lambda }\ \mathrm{otherwise}.
\endaligned
\right.
$$
Thus $\delta_P(\DP^2,\lambda C)=\frac{3}{4}\cdot \frac{5-8\lambda}{3-d \lambda }$ \text{ for }$\lambda\in \big[\frac{3}{8},\big\{\frac{5}{8}, \frac{3}{d}\big\}\big].$
 \end{proof}
  \noindent This proves the following theorem:
 \begin{theorem}
   Suppose $C_4\subset \DP^2$ is a quartic curve with  $A_7$ singularity. Then for $\lambda\in \big[\frac{3}{8},\frac{5}{8}\big].$ we have:
$$\delta(\DP^2,\lambda C_4)=\frac{3}{4}\cdot \frac{5-8\lambda}{3-4 \lambda }.$$
 \end{theorem}

 \section{Curves with $D_4$ singularities}
   \begin{lemma}
 Let $C\subset \DP^2$ be is a plane curve of degree $d$ with a $D_4$ singularity at point $P\in C$ on $C$. Then
 $$\delta_P(\DP^2,\lambda C)=\frac{3}{2}\cdot\frac{2-3\lambda}{3-d \lambda }\text{ for }\lambda\in\Big[0,\frac{2}{3}\Big].$$
 \end{lemma}
 \begin{proof}
Let $\pi:\widetilde{S}\to S$ be the~blow up of the~point $P$,
with the exceptional divisor $E$. We denote the strict transform of $C$  on $\widetilde{S}$ by $\widetilde{C}$.
\begin{center}
 \includegraphics[width=8cm]{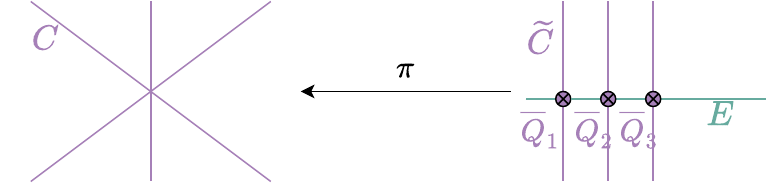}
 \end{center}
 We have $\pi^*(C)=\widetilde{C}+3E$, $\pi^*(K_{\DP^2})=K_{\widetilde{S}}-E$. Thus, $A_{(\DP^2,\lambda C)}(E)=2-3\lambda$.
\noindent The Zariski decomposition of the divisor  $-\pi^*(K_{\DP^2}+\lambda C)-vE$ is given by:
\begin{align*}
P(v)=-\pi^*(K_{\DP^2}+\lambda C)-vE \text{ and }N(v)=0\text{ if }v\in[0,3-d \lambda ].
\end{align*}
Then
$$
P(v)^2=\big(v-(3-d \lambda )\big)\big(v+(3-d \lambda )\big)\text{ and }P(v)\cdot E= v\text{ if }v\in[0,3-d \lambda ].
$$
Thus
$$S_{(\DP^2, \lambda C)}(E)=\frac{1}{(3-d \lambda )^2}\Big(\int_0^{3-d \lambda } \big(v-(3-d \lambda )\big)\big(v+(3-d \lambda )\big) dv\Big)=\frac{2(3-d \lambda )}{3}$$
so that $\delta_P(\DP^2,\lambda C)\le \frac{3}{2}\frac{2-3\lambda}{3-d \lambda }$. For every $O\in E$ we get:
$$h(v) = \frac{v^2}{2}\text{ if }v\in[0,3-d \lambda ].
$$
So that
$$S\big(W^{E}_{\bullet,\bullet};O\big)= \frac{2}{(3-d \lambda )^2}\Big(\int_0^{3-d \lambda } \frac{v^2}{2} dv\Big)=\frac{3-d \lambda }{3}\le \frac{2}{3}\cdot\frac{3-d \lambda }{2-3\lambda}
$$
We have
$$
\delta_P(\DP^2,\lambda C)\geqslant\mathrm{min}\Bigg\{  \frac{3}{2}\cdot\frac{2-3\lambda}{3-d \lambda },\inf_{O\in E}\frac{A_{E,\Delta_{E}}(O)}{S\big(W^{E}_{\bullet,\bullet};O\big)}\Bigg\},
$$
where $\Delta_{E}=\lambda Q_1+\lambda Q_2+\lambda Q_3$ where $Q_1+Q_2+Q_3=\widetilde{C}|_E$. 
So that
$$
\frac{A_{\overline{E},\Delta_{\overline{E}}}(O)}{S(W_{\bullet,\bullet}^{\overline{E}};O)}=
\left\{\aligned
&\frac{3(1-\lambda)}{3-d \lambda }\ \mathrm{if}\ O\in\{Q_1,Q_2,Q_3\},\\
&\frac{3}{3-d \lambda }\ \mathrm{otherwise}.
\endaligned
\right.
$$
Thus $\delta_P(\DP^2,\lambda C)=\frac{3}{2}\cdot\frac{2-3\lambda}{3-d \lambda }$ for $\lambda\in\big[0,\frac{2}{3}\big]$.
 \end{proof}
 \noindent This proves the following theorem:
  \begin{theorem}
  Suppose $C_3\subset \DP^2$ is a cubic curve with  $D_4$ singularity. Then  for $\lambda\in\big[0,\frac{2}{3}\big]$ we have:
 $$\delta(\DP^2,\lambda C_3)=\frac{2-3\lambda}{2- 2\lambda }.$$
 \end{theorem}
 \begin{theorem}
  Suppose $C_4\subset \DP^2$ is a quartic curve with  $D_4$ singularity. Then  for $\lambda\in\big[0,\frac{2}{3}\big]$ we have:
 $$\delta(\DP^2,\lambda C_4)=\frac{3}{2}\cdot\frac{2-3\lambda}{3-d \lambda }.$$
 \end{theorem}

  \section{Curves with $D_5$ singularities}
 \begin{lemma}
 Let $C\subset \DP^2$ be a plane curve of degree $d$ with a $D_5$ singularity at point $P\in C$  on $C$. Then
 $$\delta_P(\DP^2,\lambda C)=\frac{3}{5}\cdot\frac{5-8\lambda}{3-d \lambda }\text{ for }\lambda\in\Big[0,\min \Big\{\frac{5}{8}, \frac{3}{d}\Big\}\Big].$$
 \end{lemma}
 \begin{proof}
Let $L$ be a tangent line at point $P$.  Let $\pi_1:S^1\to S$ be the~blow up of the~point $P$,
with the exceptional divisor $E_1^1$ and $L^1$, $C^1$ are strict transforms of $L$ and $C$ respectively; $\pi_2: S^2\to S^1$ be the~blow up of the~point $E_1^1\cap C^1\cap L^1$,
with the exceptional divisor $E_2^2$ and $L^2$, $C^2$, $E_1^2$ are strict transforms of $L^1$, $C^1$, $E_1^1$ respectively; $\pi_3: S^3\to S^2$ be the~blow up of the~point $E_1^2\cap E_2^2\cap C^2$,
with the exceptional divisor $E$ and $L^3$, $C^3$, $E_1^3$, $E_2^3$ are strict transforms of $L^2$, $C^2$, $E_1^2$, $E_2^2$ respectively; let $\theta:S^3\to\overline{S}$ be  the~contraction of  the~curves  $E_1^3$, $E_2^3$
and $\sigma$ is the~birational contraction of $\overline{E}=\theta(E)$. We denote the strict transforms of $C$ and $L$ on $\overline{S}$ by $\overline{C}$ and $\overline{L}$. Let $\overline{E}$ be the exceptional divisor of $\sigma$. Note that $\overline{E}$ contains two singular points $\overline{P}_1$ and  $\overline{P}_2$   which are $\frac{1}{3}(1,1)$ and  $\frac{1}{2}(1,1)$ singular points respectively.
\begin{center}
 \includegraphics[width=16cm]{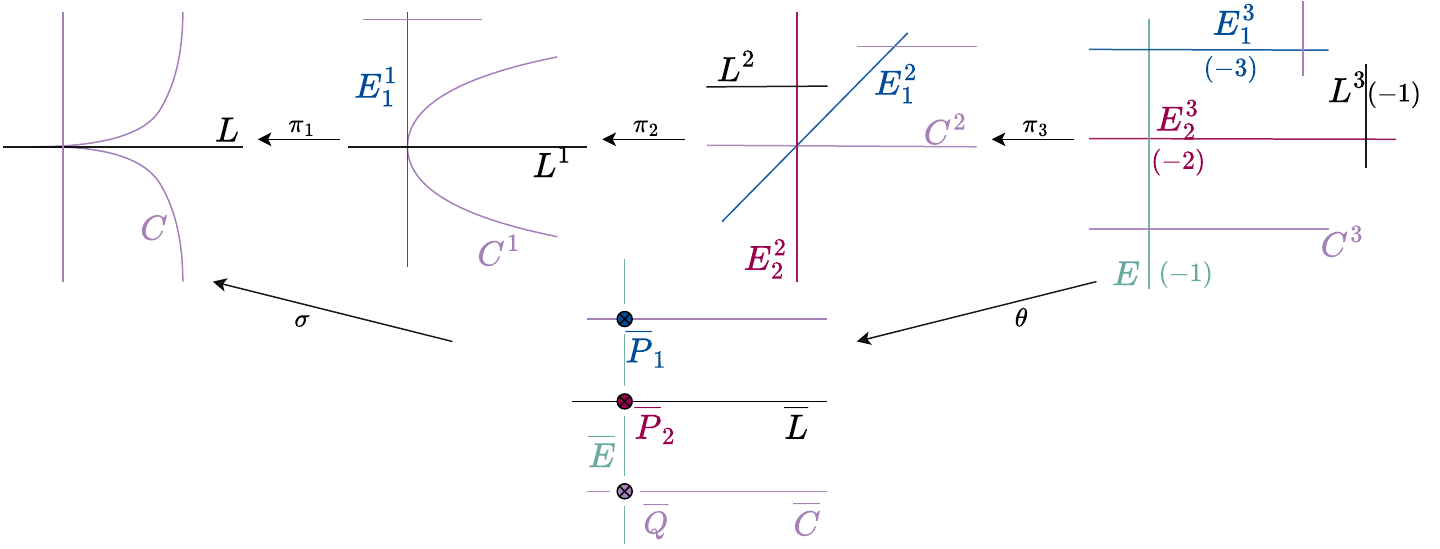}
 \end{center}
The intersections on $\overline{S}$ are given by:
\begin{center}
    \renewcommand{\arraystretch}{1.4}
    \begin{tabular}{|c|c|c|}
    \hline
         & $\overline{E}$ & $\overline{L}$ \\
    \hline
       $\overline{E}$  & $-\frac{1}{6}$ & $\frac{1}{2}$ \\
    \hline
       $\overline{L}$  & $\frac{1}{2}$ & $-\frac{1}{2}$ \\
    \hline
    \end{tabular}
\end{center}
 We have $\sigma^*(L)=\overline{L}+3\overline{E}$, $\sigma^*(C)=\overline{C}+8\overline{E}$, $\sigma^*(K_{\DP^2})=K_{\overline{S}}-4\overline{E}$. Thus, $A_{(\DP^2,\lambda C)}(\overline{E})=5-8\lambda$.
\noindent The Zariski decomposition of the divisor  $-\sigma^*(K_{\DP^2}+\lambda C)-v\overline{E}$ is given by:
\begin{align*}
&&P(v)=
\begin{cases}
-\sigma^*(K_{\DP^2}+\lambda C)-v\overline{E}\text{ if }v\in[0,2(3-d \lambda )],\\
-\sigma^*(K_{\DP^2}+\lambda C)-v\overline{E}-2\big(v-(3-d \lambda )\big)\overline{L}\text{ if }v\in[2(3-d \lambda ), 3(3-d \lambda )],\\
\end{cases}\\&&
N(v)=
\begin{cases}
0\text{ if }v\in[0,2(3-d \lambda )],\\
2\big(v-(3-d \lambda )\big)\overline{L}\text{ if }v\in[2(3-d \lambda ),3(3-d \lambda )].
\end{cases}
\end{align*}
Then
{\small $$
P(v)^2=
\begin{cases}
(3-d \lambda )^2 -\frac{v^2}{6}\text{ if }v\in[0,2(3-4)\lambda],\\
\frac{(v -3(3-d \lambda ))^2}{3}\text{ if }v\in[2(3-d \lambda ),3(3-d \lambda )],
\end{cases}
\text{and }
P(v)\cdot \overline{E}=
\begin{cases}
\frac{v}{6}\text{ if }v\in[0,2(3-d \lambda )],\\
3-d \lambda  -\frac{v}{3}\text{ if }v\in[2(3-d \lambda ),3(3-d \lambda )],
\end{cases}
$$}
Thus
$$S_{(\DP^2, \lambda C)}(\overline{E})=\frac{1}{(3-d \lambda )^2}\Big(\int_0^{2(3-d \lambda )} (3-d \lambda )^2 -\frac{v^2}{6} dv+\int_{2(3-d \lambda )}^{3(3-d \lambda )} \frac{(v - 3(3-d \lambda ))^2}{3} dv\Big)=\frac{5(3-d \lambda )}{3}$$
so that $\delta_P(\DP^2,\lambda C)\le \frac{3}{5}\cdot\frac{5-6\lambda}{3-d \lambda }$. For every $O\in \overline{E}$, we get if $O\in \overline{E}\backslash \overline{L}$ or if $O\in \overline{E}\cap \overline{L}$:
{\small $$h(v)\le 
\begin{cases}
\frac{v^2}{72}\text{ if }v\in[0,2(3-d \lambda )],\\
\frac{(v - 3(3-d \lambda ))^2}{18}\text{ if }v\in[2(3-d \lambda ), 3(3-d \lambda )],
\end{cases}
\text{or }
h(v)\le \begin{cases}
\frac{v^2}{72}\text{ if }v\in[0,2(3-d \lambda )],\\
\frac{ (v - 3(3-d \lambda )) (3(3-d \lambda ) -5v)}{18}\text{ if }v\in[2(3-d \lambda ), 3(3-d \lambda )].
\end{cases}
$$}
So that
$$S\big(W^{\overline{E}}_{\bullet,\bullet};O\big)\le \frac{2}{(3-d \lambda )^2}\Big(\int_0^{2(3-d \lambda )} \frac{v^2}{72} dv+\int_{2(3-d \lambda )}^{3(3-d \lambda )} \frac{(v - 3(3-d \lambda ))^2}{18} dv\Big)=\frac{3-d \lambda }{9}
$$
or
$$S\big(W^{\overline{E}}_{\bullet,\bullet};O\big)\le \frac{2}{(3-d \lambda )^2}\Big(\int_0^{2(3-d \lambda )} \frac{v^2}{72} dv+\int_{2(3-d \lambda )}^{3(3-d \lambda )} \frac{ (v -3(3-d \lambda )) (3(3-d \lambda ) -5v  )}{18} dv\Big)=\frac{3-d \lambda }{6}
$$
We have
$$
\delta_P(\DP^2,\lambda C)\geqslant\mathrm{min}\Bigg\{\frac{3}{5}\cdot\frac{5-8\lambda}{3-d \lambda },\inf_{O\in\overline{E}}\frac{A_{\overline{E},\Delta_{\overline{E}}}(O)}{S\big(W^{\overline{E}}_{\bullet,\bullet};O\big)}\Bigg\},
$$
where $\Delta_{\overline{E}}=\frac{2+\lambda}{3}\overline{P}_1+\frac{1}{2}\overline{P}_2+\lambda\overline{Q}$ where $\overline{Q}+\overline{P}_1=\overline{C}|_{\overline{E}}$.  
So that
$$
\frac{A_{\overline{E},\Delta_{\overline{E}}}(O)}{S(W_{\bullet,\bullet}^{\overline{E}};O)}=
\left\{\aligned
&\frac{3(1-\lambda)}{3-d \lambda }\ \mathrm{if}\ O=\overline{P}_1,\\
&\frac{9(1-\lambda)}{3-d \lambda }\ \mathrm{if}\ O=\overline{Q},\\
&\frac{3}{3-d \lambda }\ \mathrm{if}\ O=\overline{P}_2,\\
&\frac{9}{3-d \lambda }\ \mathrm{otherwise}.
\endaligned
\right.
$$
Thus $\delta_P(\DP^2,\lambda C)=\frac{3}{5}\cdot\frac{5-8\lambda}{3-d \lambda }$.
 \end{proof}

  \noindent This proves the following theorem:
 \begin{theorem}
  Suppose $C_4\subset \DP^2$ is a quartic curve with  $D_5$ singularity. Then  for $\lambda\in\big[0,\frac{5}{8}\big]$ we have:
 $$\delta(\DP^2,\lambda C_4)=\frac{3}{5}\cdot\frac{5-8\lambda}{3-4 \lambda }.$$
 \end{theorem}

 \section{Curves with $D_6$ singularities}
 \begin{lemma}
 Let $C\subset \DP^2$ be a plane curve of degree $d$ with a $D_6$ singularity at point $P\in C$  on $C$ such that the tangent line $L$ to $C$ at point $P$ is a component of $C$. Then
 $$\delta_P(\DP^2,\lambda C)=\frac{3-5\lambda}{3-d \lambda }\text{ for }\lambda\in\Big[0,\min \Big\{\frac{3}{5}, \frac{3}{d}\Big\}\Big].$$
 \end{lemma}
 \begin{proof}
 The tangent line $L$ to $C$ at point $P$ is a component of $C$. Let $C=\mathcal{C}+L$. Let $\pi_1:S^1\to S$ be the~blow up of the~point $P$,
with the exceptional divisor $E_1^1$ and $L^1$, $C^1$ are strict transforms of $L$ and $\mathcal{C}$ respectively; $\pi_2: S^2\to S^1$ be the~blow up of the~point $E_1^1\cap C^1\cap L^1$,
with the exceptional divisor $E$ and $L^2$, $C^2$ are strict transforms of $L^1$ and $C^1$ respectively; let $\theta:S^2\to\overline{S}$ be  the~contraction of  the~curve $E_1^2$,
and $\sigma$ is the~birational contraction of $\overline{E}=\theta(E)$. We denote the strict transforms of $C$ and $L$ on $\overline{S}$ by $\overline{C}$ and $\overline{L}$. Let $\overline{E}$ be the exceptional divisor of $\sigma$. Note that $\overline{E}$ contains one singular point $\overline{P}$  which is $\frac{1}{2}(1,1)$  singular point.
\begin{center}
 \includegraphics[width=12cm]{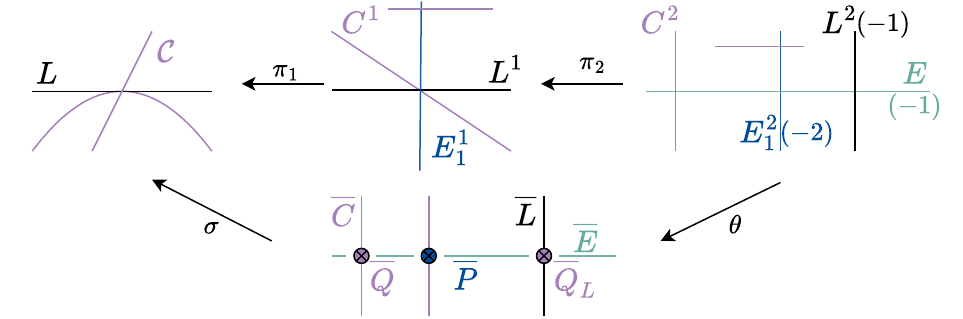}
 \end{center}
In both cases the intersections on $\overline{S}$ are given by:
\begin{center}
    \renewcommand{\arraystretch}{1.4}
    \begin{tabular}{|c|c|c|}
    \hline
         & $\overline{E}$ & $\overline{L}$ \\
    \hline
       $\overline{E}$  & $-\frac{1}{2}$ & $1$ \\
    \hline
       $\overline{L}$  & $1$ & $-1$ \\
    \hline
    \end{tabular}
\end{center}
 We have $\sigma^*(L)=\overline{L}+2\overline{E}$, $\sigma^*(C)=\overline{C}+\overline{L}+5\overline{E}$, $\sigma^*(K_{\DP^2})=K_{\overline{S}}-2\overline{E}$. Thus, $A_{(\DP^2,\lambda C)}(\overline{E})=3-5\lambda$.
\noindent The Zariski decomposition of the divisor  $-\sigma^*(K_{\DP^2}+\lambda C)-v\overline{E}$ is given by:
\begin{align*}
&&P(v)=
\begin{cases}
-\sigma^*(K_{\DP^2}+\lambda C)-v\overline{E}\text{ if }v\in[0,3-d \lambda ],\\
-\sigma^*(K_{\DP^2}+\lambda C)-v\overline{E}-\big(v-(3-d \lambda )\big)\overline{L}\text{ if }v\in[3-d \lambda , 2(3-d \lambda )],\\
\end{cases}\\&&
N(v)=
\begin{cases}
0\text{ if }v\in[0,3-d \lambda ],\\
\big(v-(3-d \lambda )\big)\overline{L}\text{ if }v\in[3-d \lambda ,2(3-d \lambda )].
\end{cases}
\end{align*}
Then
$$
P(v)^2=
\begin{cases}
(3-d \lambda )^2 -\frac{v^2}{2}\text{ if }v\in[0,3-d \lambda ],\\
\frac{(v- 2(3-d \lambda ))^2}{2}\text{ if }v\in[3-d \lambda ,2(3-d \lambda )],
\end{cases}
\text{ and }
P(v)\cdot \overline{E}=
\begin{cases}
\frac{v}{2}\text{ if }v\in[0,3-d \lambda ],\\
3-d \lambda  -\frac{v}{2}\text{ if }v\in[3-d \lambda ,2(3-d \lambda )],
\end{cases}
$$
Thus
$$S_{(\DP^2, \lambda C)}(\overline{E})=\frac{1}{(3-d \lambda )^2}\Big(\int_0^{3-d \lambda } (3-d \lambda )^2 -\frac{v^2}{2} dv+\int_{3-d \lambda }^{2(3-d \lambda )} \frac{(v- 2(3-d \lambda))^2}{2} dv\Big)=3-d \lambda $$
so that $\delta_P(\DP^2,\lambda C)\le \frac{3-5\lambda}{3-d \lambda }$. For every $O\in \overline{E}$, we get if $O\in \overline{E}\backslash \overline{L}$ or if $O\in \overline{E}\cap \overline{L}$:
$$h(v)\le 
\begin{cases}
\frac{v^2}{8}\text{ if }v\in[0,3-d \lambda ],\\
\frac{(v- 2(3-d \lambda))^2}{8}\text{ if }v\in[3-d \lambda , 2(3-d \lambda )],
\end{cases}
\text{or }
h(v)\le \begin{cases}
\frac{v^2}{8}\text{ if }v\in[0,3-d \lambda ],\\
\frac{ (v- 2(3-d \lambda) )(2(3-d \lambda) - 3v)}{8}\text{ if }v\in[3-d \lambda , 2(3-d \lambda )].
\end{cases}
$$
So that
$$S\big(W^{\overline{E}}_{\bullet,\bullet};O\big)\le \frac{2}{(3-d \lambda )^2}\Big(\int_0^{3-d \lambda } \frac{v^2}{8} dv+\int_{3-d \lambda }^{2(3-d \lambda )} \frac{(v- 2(3-d \lambda))^2}{8} dv\Big)=\frac{3-d \lambda }{6}
$$
or
$$S\big(W^{\overline{E}}_{\bullet,\bullet};O\big)\le \frac{2}{(3-d \lambda )^2}\Big(\int_0^{3-d \lambda } \frac{v^2}{8} dv+\int_{3-d \lambda }^{2(3-d \lambda )} \frac{ (v- 2(3-d \lambda)) (2(3-d \lambda) - 3v)}{8} dv\Big)=\frac{3-d \lambda }{3}
$$
We have
$$
\delta_P(\DP^2,\lambda C)\geqslant\mathrm{min}\Bigg\{ \frac{3-5\lambda}{3-d \lambda },\inf_{O\in\overline{E}}\frac{A_{\overline{E},\Delta_{\overline{E}}}(O)}{S\big(W^{\overline{E}}_{\bullet,\bullet};O\big)}\Bigg\},
$$
where $\Delta_{\overline{E}}=\frac{1+\lambda}{2}\overline{P}+\lambda \overline{Q}+\lambda \overline{Q}_L$ where $\overline{Q}+ \overline{Q}_L=\overline{C}|_{\overline{E}}$.  
So that
$$
\frac{A_{\overline{E},\Delta_{\overline{E}}}(O)}{S(W_{\bullet,\bullet}^{\overline{E}};O)}=
\left\{\aligned
&\frac{3(1-\lambda)}{3-d \lambda }\ \mathrm{if}\  O\in\{\overline{P},\overline{Q}_L\},\\
&\frac{6(1-\lambda)}{3-d \lambda }\ \mathrm{if}\ O=\overline{Q},\\
&\frac{6}{3-d \lambda }\ \mathrm{otherwise}.
\endaligned
\right.
$$
Thus $\delta_P(\DP^2,\lambda C)= \frac{3-5\lambda}{3-d \lambda }$.
 \end{proof}
   \noindent This proves the following theorem:
 \begin{theorem}
  Suppose $C_4\subset \DP^2$ is a quartic curve with  $D_6$ singularity. Then  for $\lambda\in\big[0,\frac{3}{5}\big]$ we have:
 $$\delta(\DP^2,\lambda C_4)= \frac{3-5\lambda}{3-4 \lambda }.$$
 \end{theorem}

  \section{Curves with $E_6$ singularity}
 \begin{lemma}
 Let $C\subset \DP^2$ be a plane curve of degree $d$ with a $E_6$ singularity at point $P\in C$  on $C$. Then
 $$\delta_P(\DP^2,\lambda C)=\frac{3}{7}\cdot \frac{7-12\lambda}{3-d \lambda }\text{ for }\lambda\in \Big[0, \min \Big\{\frac{7}{12}, \frac{3}{d}\Big\}\Big].$$
 \end{lemma}
  \begin{proof}
Let $L$ be a tangent line at point $P$.  Let $\pi_1:S^1\to S$ be the~blow up of the~point $P$,
with the exceptional divisor $E_1^1$ and $L^1$, $C^1$ are strict transforms of $L$ and $C$ respectively; $\pi_2: S^2\to S^1$ be the~blow up of the~point $E_1^1\cap C^1\cap L^1$,
with the exceptional divisor $E_2^2$ and $L^2$, $C^2$, $E_1^2$ are strict transforms of $L^1$, $C^1$, $E_1^1$ respectively; $\pi_3: S^3\to S^2$ be the~blow up of the~point $E_2^2\cap E_1^2\cap C^2$,
with the exceptional divisor $E_3^3$ and $L^3$, $C^3$, $E_1^3$, $E_2^3$ are strict transforms of $L^2$, $C^2$, $E_1^2$, $E_2^2$ respectively; $\pi_4: S^4\to S^3$ be the~blow up of the~point $E_3^3\cap E_1^3\cap C^3$,
with the exceptional divisor $E$ and $L^4$, $C^4$, $E_1^4$, $E_2^4$, $E_3^4$ are strict transforms of $L^3$, $C^3$, $E_1^3$, $E_2^3$, $E_3^2$ respectively; let $\theta:S^4\to\overline{S}$ be  the~contraction of  the~curves  $E_1^4$, $E_2^4$, $E_3^4$
and $\sigma$ is the~birational contraction of $\overline{E}=\theta(E)$. We denote the strict transforms of $C$ and $L$ on $\overline{S}$ by $\overline{C}$ and $\overline{L}$. Let $\overline{E}$ be the exceptional divisor of $\sigma$. Note that $\overline{E}$ contains two singular points $\overline{P}_{1}$ and  $\overline{P}_{2,3}$   which are $\frac{1}{4}(1,1)$ and  $\frac{1}{3}(1,2)$ singular points respectively.
\begin{center}
 \includegraphics[width=16cm]{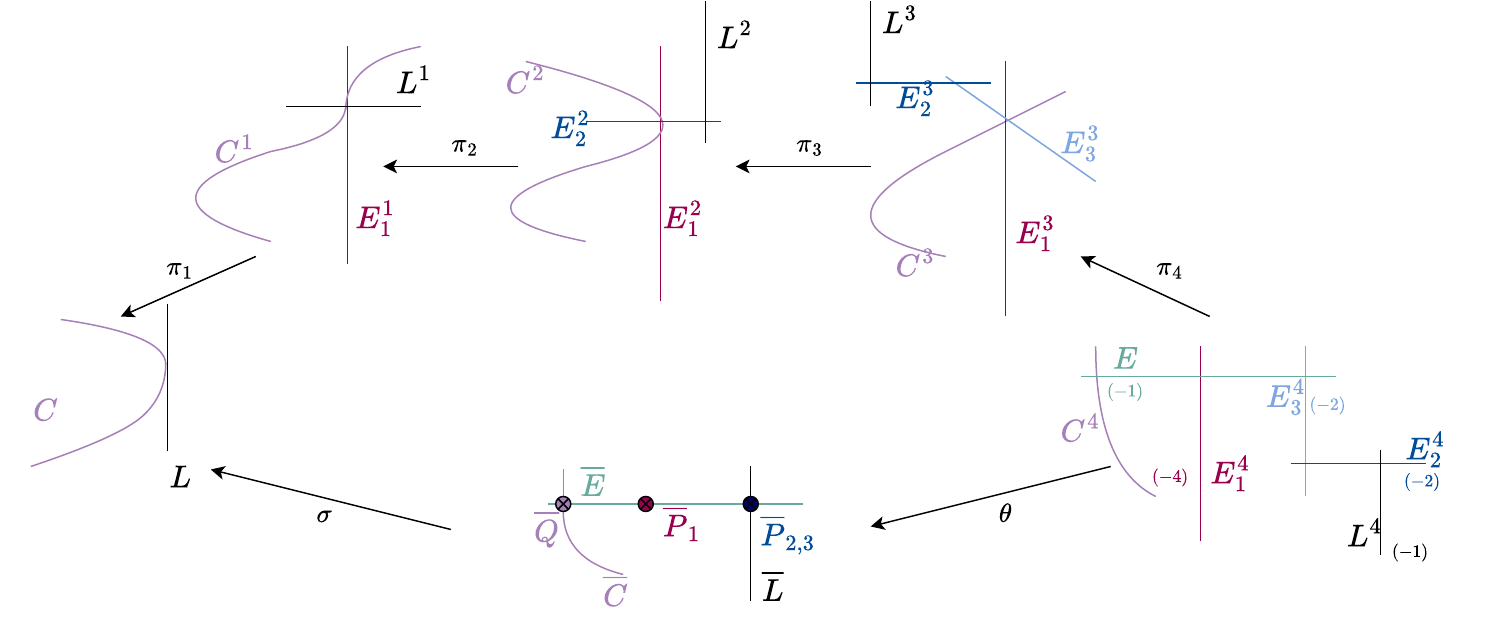}
 \end{center}
The intersections on $\overline{S}$ are given by:
\begin{center}
    \renewcommand{\arraystretch}{1.4}
    \begin{tabular}{|c|c|c|}
    \hline
         & $\overline{E}$ & $\overline{L}$ \\
    \hline
       $\overline{E}$  & $-\frac{1}{12}$ & $\frac{1}{3}$ \\
    \hline
       $\overline{L}$  & $\frac{1}{3}$ & $-\frac{1}{3}$ \\
    \hline
    \end{tabular}
\end{center}
 We have $\sigma^*(L)=\overline{L}+4\overline{E}$, $\sigma^*(C)=\overline{C}+12\overline{E}$, $\sigma^*(K_{\DP^2})=K_{\overline{S}}-6\overline{E}$. Thus, $A_{(\DP^2,\lambda C)}(\overline{E})=7-12\lambda$.
\noindent The Zariski decomposition of the divisor  $-\sigma^*(K_{\DP^2}+\lambda C)-v\overline{E}$ is given by:
\begin{align*}
&&P(v)=
\begin{cases}
-\sigma^*(K_{\DP^2}+\lambda C)-v\overline{E}\text{ if }v\in\big[0,3(3-d \lambda )\big],\\
-\sigma^*(K_{\DP^2}+\lambda C)-v\overline{E}-\big(v-3(3-d \lambda )\big)\overline{L}\text{ if }v\in\big[3(3-d \lambda ), 4(3-d \lambda )\big],\\
\end{cases}\\&&
N(v)=
\begin{cases}
0\text{ if }v\in\big[0,3(3-d \lambda )\big],\\
\big(v-3(3-d \lambda )\big)\overline{L}\text{ if }v\in\big[3(3-d \lambda ),4(3-d \lambda )\big].
\end{cases}
\end{align*}
Then
{\small $$
P(v)^2=
\begin{cases}
(3-d \lambda )^2 -\frac{v^2}{12}\text{ if }v\in\big[0,3(3-d \lambda )\big],\\
\frac{(v - 4(3-d \lambda ))^2}{4}\text{ if }v\in\big[3(3-d \lambda ),4(3-d \lambda )\big],
\end{cases}
\text{ and }
P(v)\cdot \overline{E}=
\begin{cases}
\frac{v}{12}\text{ if }v\in\big[0,3(3-d \lambda )\big],\\
3-d \lambda  -\frac{v}{4}\text{ if }v\in\big[3(3-d \lambda ),4(3-d \lambda )\big],
\end{cases}
$$}
Thus
$$S_{(\DP^2, \lambda C)}(\overline{E})=\frac{1}{(3-d \lambda )^2}\Big(\int_0^{3(3-d \lambda )} (3-d \lambda )^2 -\frac{v^2}{12} dv+\int_{3(3-d \lambda )}^{4(3-d \lambda )} \frac{(v - 4(3-d \lambda ))^2}{4} dv\Big)=\frac{7(3-d \lambda )}{3}$$
so that $\delta_P(\DP^2,\lambda C)\le \frac{3}{7}\frac{7-12\lambda}{3-d \lambda }$. For every $O\in \overline{E}$, we get if $O\in \overline{E}\backslash \overline{L}$ or if $O\in \overline{E}\cap \overline{L}$:
{\small $$h(v)\le 
\begin{cases}
\frac{v^2}{288}\text{ if }v\in\big[0,3(3-d \lambda )\big],\\
\frac{(v - 12 +16\lambda)^2}{32}\text{ if }v\in\big[3(3-d \lambda ),4(3-d \lambda )\big],
\end{cases}
\text{or }
h(v)\le \begin{cases}
\frac{v^2}{288}\text{ if }v\in\big[0,3(3-d \lambda )\big],\\
\frac{ (v - 4(3-d \lambda )) (20(3-d \lambda )  -7v )}{32}\text{ if }v\in\big[3(3-d \lambda ),4(3-d \lambda )\big].
\end{cases}
$$}
So that
$$S\big(W^{\overline{E}}_{\bullet,\bullet};O\big)\le \frac{2}{(3-d \lambda )^2}\Big(\int_0^{3(3-d \lambda )} \frac{v^2}{288} dv+\int_{3(3-d \lambda )}^{4(3-d \lambda )} \frac{(v -4(3-d \lambda ))^2}{32} dv\Big)=\frac{3-d \lambda }{12}
$$
or
$$S\big(W^{\overline{E}}_{\bullet,\bullet};O\big)\le \frac{2}{(3-d \lambda )^2}\Big(\int_0^{3(3-d \lambda )} \frac{v^2}{288} dv+\int_{3(3-d \lambda )}^{4(3-d \lambda )} \frac{ (v - 4(3-d \lambda )) (20(3-d \lambda ) -7v )}{32} dv\Big)=\frac{3-d \lambda }{9}
$$
We have
$$
\delta_P(\DP^2,\lambda C)\geqslant\mathrm{min}\Bigg\{\frac{3}{7}\cdot \frac{7-12\lambda}{3-d \lambda },\inf_{O\in\overline{E}}\frac{A_{\overline{E},\Delta_{\overline{E}}}(O)}{S\big(W^{\overline{E}}_{\bullet,\bullet};O\big)}\Bigg\},
$$
where $\Delta_{\overline{E}}=\frac{3}{4}\overline{P}_{1}+\frac{2}{3}\overline{P}_{2,3}+\lambda\overline{Q}$ where $\overline{Q}=\overline{C}|_{\overline{E}}$. 
So that
$$
\frac{A_{\overline{E},\Delta_{\overline{E}}}(O)}{S(W_{\bullet,\bullet}^{\overline{E}};O)}=
\left\{\aligned
&\frac{3}{3-d \lambda }\ \mathrm{if}\ O\in\{\overline{P}_{1}, \overline{P}_{2,3}\},\\
&\frac{12(1-\lambda)}{3-d \lambda }\ \mathrm{if}\ O=\overline{Q},\\
&\frac{12}{3-d \lambda }\ \mathrm{otherwise}.
\endaligned
\right.
$$
Thus $$\delta_P(\DP^2,\lambda C)=\frac{3}{7}\cdot \frac{7-12\lambda}{3-d \lambda }\text{ for }\lambda\in \Big[0, \frac{7}{12}\Big].$$
 \end{proof}
 \noindent This proves the following theorem:
 \begin{theorem}
   Suppose $C_4\subset \DP^2$ is a quartic curve with  $E_6$ singularity. Then  we have:
$$\delta(\DP^2,\lambda C_4)=\frac{3}{7}\cdot \frac{7-12\lambda}{3-4 \lambda }\text{ for }\lambda\in \Big[0, \frac{7}{12}\Big].$$
 \end{theorem}

 \section{Curves with $E_7$ singularity}
 \begin{lemma}
 Let $C\subset \DP^2$ be a plane curve of degree $d$ with a $E_7$ singularity at point $P\in C$  on $C$. Then
 $$\delta_P(\DP^2,\lambda C)=\frac{3}{5}\cdot\frac{5-9\lambda}{3-d \lambda }\text{ for }\lambda\in \Big[0, \min \Big\{\frac{5}{9}, \frac{3}{d}\Big\}\Big].$$
 \end{lemma}
 \begin{proof}
Let $L$ be a tangent line at point $P$. Note that $L$ is a component of $C$ and $C=\mathcal{C}+L$.  Let $\pi_1:S^1\to S$ be the~blow up of the~point $P$,
with the exceptional divisor $E_1^1$ and $L^1$, $C^1$ are strict transforms of $L$ and $C$ respectively; $\pi_2: S^2\to S^1$ be the~blow up of the~point $E_1^1\cap C^1\cap L^1$,
with the exceptional divisor $E_2^2$ and $L^2$, $C^2$, $E_1^2$ are strict transforms of $L^1$, $C^1$, $E_1^1$ respectively; $\pi_3: S^3\to S^2$ be the~blow up of the~point $E_1^2\cap E_2^2\cap C^2$,
with the exceptional divisor $E$ and $L^3$, $C^3$, $E_1^3$, $E_2^3$ are strict transforms of $L^2$, $C^2$, $E_1^2$, $E_2^2$ respectively; let $\theta:S^3\to\overline{S}$ be  the~contraction of  the~curves  $E_1^3$, $E_2^3$
and $\sigma$ is the~birational contraction of $\overline{E}=\theta(E)$. We denote the strict transforms of $C$ and $L$ on $\overline{S}$ by $\overline{C}$ and $\overline{L}$. Let $\overline{E}$ be the exceptional divisor of $\sigma$. Note that $\overline{E}$ contains two singular points $\overline{P}_1$ and  $\overline{P}_2$   which are $\frac{1}{3}(1,1)$ and  $\frac{1}{2}(1,1)$ singular points respectively.
\begin{center}
 \includegraphics[width=16cm]{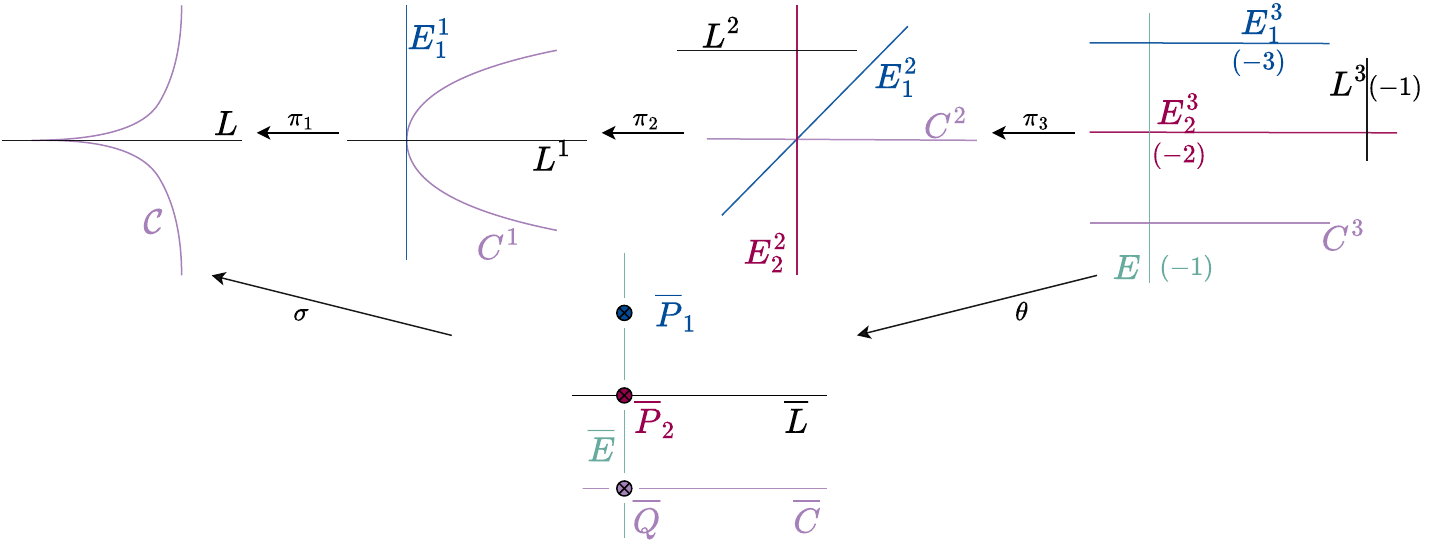}
 \end{center}
The intersections on $\overline{S}$ are given by:
\begin{center}
    \renewcommand{\arraystretch}{1.4}
    \begin{tabular}{|c|c|c|}
    \hline
         & $\overline{E}$ & $\overline{L}$ \\
    \hline
       $\overline{E}$  & $-\frac{1}{6}$ & $\frac{1}{2}$ \\
    \hline
       $\overline{L}$  & $\frac{1}{2}$ & $-\frac{1}{2}$ \\
    \hline
    \end{tabular}
\end{center}
 We have $\sigma^*(L)=\overline{L}+3\overline{E}$, $\sigma^*(C)=\overline{C}+\overline{L}+9\overline{E}$, $\sigma^*(K_{\DP^2})=K_{\overline{S}}-4\overline{E}$. Thus, $A_{(\DP^2,\lambda C)}(\overline{E})=5-9\lambda$.
\noindent The Zariski decomposition of the divisor  $-\sigma^*(K_{\DP^2}+\lambda C)-v\overline{E}$ is given by:
\begin{align*}
&&P(v)=
\begin{cases}
-\sigma^*(K_{\DP^2}+\lambda C)-v\overline{E}\text{ if }v\in[0,2(3-d \lambda )],\\
-\sigma^*(K_{\DP^2}+\lambda C)-v\overline{E}-2\big(v-(3-d \lambda )\big)\overline{L}\text{ if }v\in[2(3-d \lambda ), 3(3-d \lambda )],\\
\end{cases}\\&&
N(v)=
\begin{cases}
0\text{ if }v\in[0,2(3-d \lambda )],\\
2\big(v-(3-d \lambda )\big)\overline{L}\text{ if }v\in[2(3-d \lambda ),3(3-d \lambda )].
\end{cases}
\end{align*}
Then
$$
P(v)^2=
\begin{cases}
(3-d \lambda )^2 -\frac{v^2}{6}\text{ if }v\in[0,2(3-4)\lambda],\\
\frac{(v - 3(3-d \lambda ) )^2}{3}\text{ if }v\in[2(3-d \lambda ),3(3-d \lambda )],
\end{cases}
\text{and }
P(v)\cdot \overline{E}=
\begin{cases}
\frac{v}{6}\text{ if }v\in[0,2(3-d \lambda )],\\
3-d \lambda  -\frac{v}{3}\text{ if }v\in[2(3-d \lambda ),3(3-d \lambda )],
\end{cases}
$$
Thus
$$S_{(\DP^2, \lambda C)}(\overline{E})=\frac{1}{(3-d \lambda )^2}\Big(\int_0^{2(3-d \lambda )} (3-d \lambda )^2 -\frac{v^2}{6} dv+\int_{2(3-d \lambda )}^{3(3-d \lambda )} \frac{(v - 3(3-d \lambda ) )^2}{3} dv\Big)=\frac{5(3-d \lambda )}{3}$$
so that $\delta_P(\DP^2,\lambda C)\le \frac{3}{5}\cdot\frac{5-9\lambda}{3-d \lambda }$. For every $O\in \overline{E}$, we get if $O\in \overline{E}\backslash \overline{L}$ or if $O\in \overline{E}\cap \overline{L}$:
\begin{align*}
&&h(v)\le 
\begin{cases}
\frac{v^2}{72}\text{ if }v\in[0,2(3-d \lambda )],\\
\frac{(v - 3(3-d \lambda ))^2}{18}\text{ if }v\in[2(3-d \lambda ), 3(3-d \lambda )],
\end{cases}
\\&\text{or }&
h(v)\le \begin{cases}
\frac{v^2}{72}\text{ if }v\in[0,2(3-d \lambda )],\\
\frac{ (v - 3(3-d \lambda )) (3(3-d \lambda )-5v  )}{18}\text{ if }v\in[2(3-d \lambda ), 3(3-d \lambda )].
\end{cases}
\end{align*}
So that
$$S\big(W^{\overline{E}}_{\bullet,\bullet};O\big)\le \frac{2}{(3-d \lambda )^2}\Big(\int_0^{2(3-d \lambda )} \frac{v^2}{72} dv+\int_{2(3-d \lambda )}^{3(3-d \lambda )} \frac{(v - 3(3-d \lambda ))^2}{18} dv\Big)=\frac{3-d \lambda }{9}
$$
or
$$S\big(W^{\overline{E}}_{\bullet,\bullet};O\big)\le \frac{2}{(3-d \lambda )^2}\Big(\int_0^{2(3-d \lambda )} \frac{v^2}{72} dv+\int_{2(3-d \lambda )}^{3(3-d \lambda )} \frac{ (v -3(3-d \lambda )) (3(3-d \lambda ) -5v  )}{18} dv\Big)=\frac{3-d \lambda }{6}
$$
We have
$$
\delta_P(\DP^2,\lambda C)\geqslant\mathrm{min}\Bigg\{\frac{3}{5}\cdot\frac{5-9\lambda}{3-d \lambda },\inf_{O\in\overline{E}}\frac{A_{\overline{E},\Delta_{\overline{E}}}(O)}{S\big(W^{\overline{E}}_{\bullet,\bullet};O\big)}\Bigg\},
$$
where $\Delta_{\overline{E}}=\frac{2}{3}\overline{P}_1+\frac{1+\lambda}{2}\overline{P}_2+\lambda\overline{Q}$ where $\overline{Q}=\overline{C}|_{\overline{E}}$. 
So that
$$
\frac{A_{\overline{E},\Delta_{\overline{E}}}(O)}{S(W_{\bullet,\bullet}^{\overline{E}};O)}=
\left\{\aligned
&\frac{3}{3-d \lambda }\ \mathrm{if}\ O=\overline{P}_1,\\
&\frac{3(1-\lambda)}{3-d \lambda }\ \mathrm{if}\ O=\overline{P}_2,\\
&\frac{9(1-\lambda)}{3-d \lambda }\ \mathrm{if}\ O=\overline{Q},\\
&\frac{9}{3-d \lambda }\ \mathrm{otherwise}.
\endaligned
\right.
$$
Thus $\delta_P(\DP^2,\lambda C)=\frac{3}{5}\cdot\frac{5-9\lambda}{3-d \lambda }$.
 \end{proof}
  \noindent This proves the following theorem:
 \begin{theorem}
   Suppose $C_4\subset \DP^2$ is a quartic curve with  $E_7$ singularity. Then  we have:
$$\delta(\DP^2,\lambda C_4)=\frac{3}{5}\cdot\frac{5-9\lambda}{3-4 \lambda }\text{ for }\lambda\in \Big[0, \frac{5}{9}\Big].$$
 \end{theorem}

 \section{Four concurrent lines}
   \begin{lemma}
 Let $C\subset \DP^2$ be a plane curve of degree $d$ which has a singularity at point $P\in C$ on $C$ such that it locally looks as  four lines intersecting in one point. Then
 $$\delta_P(\DP^2,\lambda C)=\frac{3-6\lambda}{3-d \lambda }\text{ for }\lambda\in\Big[0,\frac{1}{2}\Big].$$
 \end{lemma}
 \begin{proof}
Let $\pi:\widetilde{S}\to S$ be the~blow up of the~point $P$,
with the exceptional divisor $E$. We denote the strict transform of $C$  on $\widetilde{S}$ by $\widetilde{C}$.
\begin{center}
 \includegraphics[width=8cm]{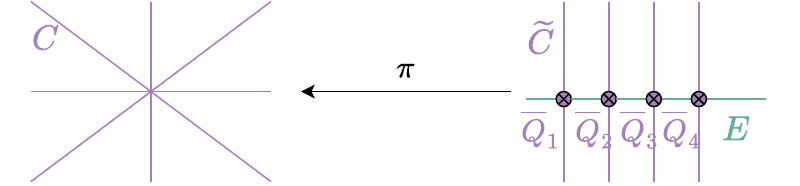}
 \end{center}
 We have $\pi^*(C)=\widetilde{C}+4E$, $\pi^*(K_{\DP^2})=K_{\widetilde{S}}-E$. Thus, $A_{(\DP^2,\lambda C)}(E)=2-4\lambda$.
\noindent The Zariski decomposition of the divisor  $-\pi^*(K_{\DP^2}+\lambda C)-vE$ is given by:
\begin{align*}
P(v)=-\pi^*(K_{\DP^2}+\lambda C)-vE \text{ and }N(v)=0\text{ if }v\in[0,3-d \lambda ].
\end{align*}
Then
$$
P(v)^2=\big(v-(3-d \lambda )\big)\big(v+(3-d \lambda )\big)\text{ and }P(v)\cdot E= v\text{ if }v\in[0,3-d \lambda ].
$$
Thus
$$S_{(\DP^2, \lambda C)}(E)=\frac{1}{(3-d \lambda )^2}\Big(\int_0^{3-d \lambda } \big(v-(3-d \lambda )\big)\big(v+(3-d \lambda )\big) dv\Big)=\frac{2(3-d \lambda )}{3}$$
so that $\delta_P(\DP^2,\lambda C)\le \frac{3}{2}\cdot \frac{2-4\lambda}{3-d \lambda } =  \frac{3-6\lambda}{3-d \lambda }$. For every $O\in E$ we get:
$$h(v) = \frac{v^2}{2}\text{ if }v\in[0,3-d \lambda ].
$$
So that
$$S\big(W^{E}_{\bullet,\bullet};O\big)= \frac{2}{(3-d \lambda )^2}\Big(\int_0^{3-d \lambda } \frac{v^2}{2} dv\Big)=\frac{3-d \lambda }{3}\le \frac{2}{3}\cdot\frac{3-d \lambda }{2-3\lambda}
$$
We have
$$
\delta_P(\DP^2,\lambda C)\geqslant\mathrm{min}\Bigg\{  \frac{3-6\lambda}{3-d \lambda },\inf_{O\in E}\frac{A_{E,\Delta_{E}}(O)}{S\big(W^{E}_{\bullet,\bullet};O\big)}\Bigg\},
$$
where $\Delta_{E}=\lambda Q_1+\lambda Q_2+\lambda Q_3+\lambda Q_4$ where $Q_1+Q_2+Q_3+Q_4=\widetilde{C}|_E$. 
So that
$$
\frac{A_{\overline{E},\Delta_{\overline{E}}}(O)}{S(W_{\bullet,\bullet}^{\overline{E}};O)}=
\left\{\aligned
&\frac{3(1-\lambda)}{3-d \lambda }\ \mathrm{if}\ O\in\{Q_1,Q_2,Q_3,Q_4\},\\
&\frac{3}{3-d \lambda }\ \mathrm{otherwise}.
\endaligned
\right.
$$
Thus $\delta_P(\DP^2,\lambda C)=\frac{3-6\lambda}{3-d \lambda }$ for $\lambda\in\big[0,\frac{1}{2}\big]$.
 \end{proof}
   \noindent This proves the following theorem:
 \begin{theorem}
   Suppose $C_4\subset \DP^2$ is a quartic curve with  a "four line" singularity. Then  we have:
$$\delta(\DP^2,\lambda C_4)=\frac{3-6\lambda}{3-d \lambda }\text{ for }\lambda\in \Big[0, \frac{1}{2}\Big].$$
 \end{theorem}

\section{Double line}
    \begin{lemma} \label{doubleline-deg4}
  Let $C\subset \DP^2$ be a double line. Then
 $$\delta_P(\DP^2,\lambda C)=\frac{3(1-2\lambda)}{3- 2 \lambda }\text{ for }\lambda\in\Big[0,\frac{1}{2}\Big].$$
 \end{lemma}
 \begin{proof}
Note that $\delta_P(\DP^2,\lambda C)\le \frac{A_{(\DP^2,\lambda C)}(\widetilde{L})}{S_{(\DP^2,\lambda C)}(\widetilde{L})}=\frac{3(1-2\lambda)}{3- 2 \lambda }$ where $\widetilde{L}$ is a strict transform of a line component of $C$. Let $\pi:\widetilde{S}\to S$ be the~blow up of the~point $P$,
with the exceptional divisor $E$.
\begin{center}
 \includegraphics[width=8cm]{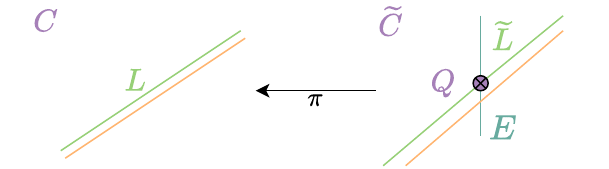}
 \end{center}
We have $\sigma^*(L)=\widetilde{L}+E$, $\pi^*(C)=\widetilde{C}+2E$, $\pi^*(K_{\DP^2})=K_{\widetilde{S}}-E$. Thus, $A_{(\DP^2,\lambda C)}(E)=2-2\lambda$.  The Zariski decomposition of the divisor  $-\pi^*(K_{\DP^2}+\lambda C)-vE$ is given by:
\begin{align*}
P(v)=-\pi^*(K_{\DP^2}+\lambda C)-vE \text{ and }N(v)=0\text{ if }v\in[0,3- 2 \lambda ].
\end{align*}
Then
$$
P(v)^2=(2\lambda - 3- v )(v-(3- 2\lambda))\text{ and }P(v)\cdot E=  v \text{ if }v\in[0,3- 2 \lambda ].
$$
Thus
$$S_{(\DP^2,\lambda C)}(E)=\frac{1}{(3- 2 \lambda )^2}\Big(\int_0^{3- 2 \lambda } (2\lambda - 3- v )(v-(3- 2\lambda)) dv\Big)=\frac{2(3- 2\lambda)}{3}$$
so that $\delta_P(\DP^2,\lambda C)\le \frac{3}{2}\cdot \frac{2-2\lambda}{3- 2 \lambda } = \frac{3-3\lambda}{3- 2 \lambda }$. 
For every $O\in E$ we get:
$$h(v) = \frac{v^2}{2}\text{ if }v\in[0,3- 2\lambda ].
$$
So that
$$S\big(W^{E}_{\bullet,\bullet};O\big)= \frac{2}{(3- 2\lambda )^2}\Big(\int_0^{3- 2\lambda } \frac{v^2}{2} dv\Big)=\frac{3- 2 \lambda}{3}
$$
We have
$$
\delta_P(\DP^2,\lambda C)\geqslant\mathrm{min}\Bigg\{ \frac{3-3\lambda}{3- 2 \lambda } ,\inf_{O\in E}\frac{A_{E,\Delta_{E}}(O)}{S\big(W^{E}_{\bullet,\bullet};O\big)}\Bigg\},
$$
where $\Delta_{E}=2\lambda Q$ where $Q=\widetilde{L}|_E$. 
So that
$$
\frac{A_{\overline{E},\Delta_{\overline{E}}}(O)}{S(W_{\bullet,\bullet}^{\overline{E}};O)}=
\left\{\aligned
&\frac{3(1-2\lambda)}{3- 2 \lambda }\ \mathrm{if}\ O=Q,\\
&\frac{3}{3- 2 \lambda }\ \mathrm{otherwise}.
\endaligned
\right.
$$
Thus $\delta_P(\DP^2,\lambda C)=\frac{3(1-2\lambda)}{3- 2 \lambda }$ for $\lambda\in\big[0,\frac{1}{2}\big]$.     
 \end{proof}
 \noindent This proves the following theorem:
 \begin{theorem}
   Suppose $C_2\subset \DP^2$ is a double line. Then  we have:
$$\delta(\DP^2,\lambda C_2)= \frac{3(1-2\lambda)}{3- 2 \lambda }\text{ for }\lambda\in \Big[0, \frac{1}{3}\Big].$$
 \end{theorem} 
 \section{Double line and another line}
  \begin{lemma}
 Let $C\subset \DP^2$ be a double line and another line, $P\in C$ be a smooth point. Then
  $$\delta_P(\DP^2,\lambda C)=1\text{ for }\lambda\in\big[0, 1\Big].$$
 \end{lemma}
 \begin{proof}
     Follows from Lemma \ref{line}.
 \end{proof}
  
 \begin{lemma}
 Let $C\subset \DP^2$ be a double line and another line, $P\in C$ be a smooth point  on a double line $L$. Then
  $$\delta_P(\DP^2,\lambda C)=\frac{1-2\lambda}{1-  \lambda }\text{ for }\lambda\in\Big[0,\frac{1}{2}\Big].$$
 \end{lemma}
 \begin{proof}
Note that $\delta_P(\DP^2,\lambda C)\le \frac{A_{(\DP^2,\lambda C)}(\widetilde{L})}{S_{(\DP^2,\lambda C)}(\widetilde{L})}=\frac{1-2\lambda}{1-  \lambda }$ where $\widetilde{L}$ is a strict transform of a line component of $C$. Let $\pi:\widetilde{S}\to S$ be the~blow up of the~point $P$,
with the exceptional divisor $E$.
\begin{center}
 \includegraphics[width=8cm]{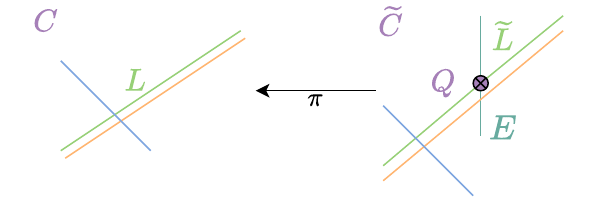}
 \end{center}
We have $\sigma^*(L)=\widetilde{L}+E$, $\pi^*(C)=\widetilde{C}+2E$, $\pi^*(K_{\DP^2})=K_{\widetilde{S}}-E$. Thus, $A_{(\DP^2,\lambda C)}(E)=2-2\lambda$.  The Zariski decomposition of the divisor  $-\pi^*(K_{\DP^2}+\lambda C)-vE$ is given by:
\begin{align*}
P(v)=-\pi^*(K_{\DP^2}+\lambda C)-vE \text{ and }N(v)=0\text{ if }v\in[0,3- 3 \lambda ].
\end{align*}
Then
$$
P(v)^2=(3\lambda - 3- v )(v-(3- 3\lambda))\text{ and }P(v)\cdot E=  v \text{ if }v\in[0,3- 3 \lambda ].
$$
Thus
$$S_{(\DP^2,\lambda C)}(E)=\frac{1}{(3- 3 \lambda )^2}\Big(\int_0^{3- 3 \lambda } (3\lambda - 3- v )(v-(3- 3\lambda)) dv\Big)=\frac{2(3- 3\lambda)}{3}$$
so that $\delta_P(\DP^2,\lambda C)\le \frac{3}{2}\cdot \frac{2-2\lambda}{3- 3 \lambda } = 1$. 
For every $O\in E$ we get:
$$h(v) = \frac{v^2}{2}\text{ if }v\in[0,3- 3\lambda ].
$$
So that
$$S\big(W^{E}_{\bullet,\bullet};O\big)= \frac{2}{(3- 3\lambda )^2}\Big(\int_0^{3- 3\lambda } \frac{v^2}{2} dv\Big)=\frac{3- 3 \lambda}{3}
$$
We have
$$
\delta_P(\DP^2,\lambda C)\geqslant\mathrm{min}\Bigg\{ 1 ,\inf_{O\in E}\frac{A_{E,\Delta_{E}}(O)}{S\big(W^{E}_{\bullet,\bullet};O\big)}\Bigg\},
$$
where $\Delta_{E}=2\lambda Q$ where $Q=\widetilde{L}|_E$. 
So that
$$
\frac{A_{\overline{E},\Delta_{\overline{E}}}(O)}{S(W_{\bullet,\bullet}^{\overline{E}};O)}=
\left\{\aligned
&\frac{3(1-2\lambda)}{3- 3 \lambda }\ \mathrm{if}\ O=Q,\\
&\frac{3}{3- 3 \lambda }\ \mathrm{otherwise}.
\endaligned
\right.
$$
Thus $\delta_P(\DP^2,\lambda C)=\frac{1-2\lambda}{1-  \lambda }$ for $\lambda\in\big[0,\frac{1}{2}\big]$.     
 \end{proof}
 \begin{lemma}
 Let $C\subset \DP^2$ be a double line and another line, $P\in C$ be a singular point  on a double line $L$. Then
  $$\delta_P(\DP^2,\lambda C)=\frac{1-2\lambda}{1-  \lambda }\text{ for }\lambda\in\Big[0,\frac{1}{2}\Big].$$
 \end{lemma}

  \begin{proof}
Note that $\delta_P(\DP^2,\lambda C)\le \frac{A_{(\DP^2,\lambda C)}(\widetilde{L})}{S_{(\DP^2,\lambda C)}(\widetilde{L})}=\frac{1-2\lambda}{1- \lambda }$ where $\widetilde{L}$ is a strict transform of a double line component of $C$.  Let $\pi:\widetilde{S}\to S$ be the~blow up of the~point $P$,
with the exceptional divisor $E$. We denote the strict transform of $C$  on $\widetilde{S}$ by $\widetilde{C}$.
\begin{center}
 \includegraphics[width=9cm]{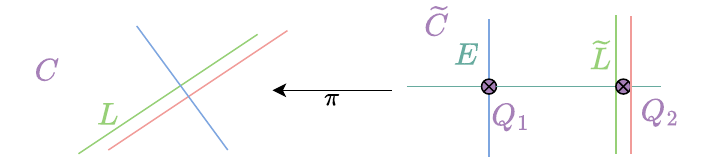}
 \end{center}
We have $\pi^*(C)=\widetilde{C}+3E$, $\pi^*(K_{\DP^2})=K_{\widetilde{S}}-E$. Thus, $A_{(\DP^2,\lambda C)}(E)=2-3\lambda$.
\noindent The Zariski decomposition of the divisor  $-\pi^*(K_{\DP^2}+\lambda C)-vE$ is given by:
\begin{align*}
P(v)=-\pi^*(K_{\DP^2}+\lambda C)-vE \text{ and }N(v)=0\text{ if }v\in[0,3- 3\lambda ].
\end{align*}
Then
$$
P(v)^2=\big(v-(3- 3\lambda )\big)\big(v+(3- 3\lambda )\big)\text{ and }P(v)\cdot E= v\text{ if }v\in[0,3- 3\lambda ].
$$
Thus
$$S_{(\DP^2, \lambda C)}(E)=\frac{1}{(3- 3\lambda )^2}\Big(\int_0^{3- 3\lambda } \big(v-(3- 3\lambda )\big)\big(v+(3- 3\lambda )\big) dv\Big)=\frac{2(3- 3\lambda )}{3}$$
so that $\delta_P(\DP^2,\lambda C)\le \frac{3}{2}\cdot \frac{2-3\lambda}{3- 3\lambda } $. For every $O\in E$ we get:
$$h(v) = \frac{v^2}{2}\text{ if }v\in[0,3- 3\lambda ].
$$
So that
$$S\big(W^{E}_{\bullet,\bullet};O\big)= \frac{2}{(3- 3\lambda )^2}\Big(\int_0^{3- 3\lambda } \frac{v^2}{2} dv\Big)=\frac{3- 3\lambda }{3}$$
We have
$$
\delta_P(\DP^2,\lambda C)\geqslant\mathrm{min}\Bigg\{ \frac{3}{2}\cdot \frac{2-3\lambda}{3- 3\lambda } ,\inf_{O\in E}\frac{A_{E,\Delta_{E}}(O)}{S\big(W^{E}_{\bullet,\bullet};O\big)}\Bigg\},
$$
where $\Delta_{E}=\lambda Q_1+2\lambda Q_2$ where $ Q_1+ 2Q_2=\widetilde{C}|_E$. 
So that
$$
\frac{A_{\overline{E},\Delta_{\overline{E}}}(O)}{S(W_{\bullet,\bullet}^{\overline{E}};O)}=
\left\{\aligned
&\frac{3(1-\lambda)}{3- 3\lambda }\ \mathrm{if}\ O=Q_1,\\
&\frac{3(1-2\lambda)}{3- 3\lambda }\ \mathrm{if}\ O=Q_2,\\
&\frac{3}{3- 3\lambda }\ \mathrm{otherwise}.
\endaligned
\right.
$$
Thus $\delta_P(\DP^2,\lambda C)= \frac{3(1-2\lambda)}{3- 3\lambda }$.
 \end{proof}
\noindent This proves the following theorem:
 \begin{theorem}
   Suppose $C_3\subset \DP^2$ is a cubic curve which is  a union of a double line  and another line. Then  we have:
$$\delta(\DP^2,\lambda C_3)= \frac{1-2\lambda}{1- \lambda }\text{ for }\lambda\in \Big[0, \frac{1}{2}\Big].$$
 \end{theorem}

 \section{Triple line}
 
   \begin{lemma} 
  Let $C\subset \DP^2$  be a triple line, $P\in C$ be a  point on $C$. Then
 $$\delta_P(\DP^2,\lambda C)=\frac{1-3\lambda}{1- \lambda }\text{ for }\lambda\in\Big[0,\frac{1}{3}\Big].$$
 \end{lemma}
 \begin{proof}
Note that $\delta_P(\DP^2,\lambda C)\le \frac{A_{(\DP^2,\lambda C)}(\widetilde{L})}{S_{(\DP^2,\lambda C)}(\widetilde{L})}=\frac{1-3\lambda}{1- \lambda }$ where $\widetilde{L}$ is a strict transform of a triple line component of $C$. Let $\pi:\widetilde{S}\to S$ be the~blow up of the~point $P$,
with the exceptional divisor $E$.
\begin{center}
 \includegraphics[width=8cm]{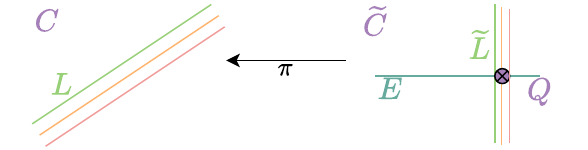}
 \end{center}
We have $\sigma^*(L)=\widetilde{L}+E$, $\pi^*(C)=\widetilde{C}+3E$, $\pi^*(K_{\DP^2})=K_{\widetilde{S}}-E$. Thus, $A_{(\DP^2,\lambda C)}(E)=2-3\lambda$.  The Zariski decomposition of the divisor  $-\pi^*(K_{\DP^2}+\lambda C)-vE$ is given by:
\begin{align*}
P(v)=-\pi^*(K_{\DP^2}+\lambda C)-vE \text{ and }N(v)=0\text{ if }v\in[0,3- 3 \lambda ].
\end{align*}
Then
$$
P(v)^2=(3\lambda - 3- v )(v-(3- 3\lambda))\text{ and }P(v)\cdot E=  v \text{ if }v\in[0,3- 3 \lambda ].
$$
Thus
$$S_{(\DP^2,\lambda C)}(E)=\frac{1}{(3- 3 \lambda )^2}\Big(\int_0^{3- 3 \lambda } (3\lambda - 3- v )(v-(3- 3\lambda)) dv\Big)=\frac{2(3- 3\lambda)}{3}$$
so that $\delta_P(\DP^2,\lambda C)\le \frac{3}{2}\cdot \frac{2-3\lambda}{3- 3 \lambda }$. 
For every $O\in E$ we get:
$$h(v) = \frac{v^2}{2}\text{ if }v\in[0,3- 3\lambda ].
$$
So that
$$S\big(W^{E}_{\bullet,\bullet};O\big)= \frac{2}{(3- 3\lambda )^2}\Big(\int_0^{3- 3\lambda } \frac{v^2}{2} dv\Big)=\frac{3- 3 \lambda}{3}
$$
We have
$$
\delta_P(\DP^2,\lambda C)\geqslant\mathrm{min}\Bigg\{ \frac{3}{2}\cdot \frac{2-3\lambda}{3- 3 \lambda } ,\inf_{O\in E}\frac{A_{E,\Delta_{E}}(O)}{S\big(W^{E}_{\bullet,\bullet};O\big)}\Bigg\},
$$
where $\Delta_{E}=3\lambda Q$ where $Q=\widetilde{L}|_E$. 
So that
$$
\frac{A_{\overline{E},\Delta_{\overline{E}}}(O)}{S(W_{\bullet,\bullet}^{\overline{E}};O)}=
\left\{\aligned
&\frac{3(1-3\lambda)}{3- 3 \lambda }\ \mathrm{if}\ O=Q,\\
&\frac{3}{3- 3 \lambda }\ \mathrm{otherwise}.
\endaligned
\right.
$$
Thus $\delta_P(\DP^2,\lambda C)=\frac{1-3\lambda}{1- \lambda }$ for $\lambda\in\big[0,\frac{1}{3}\big]$.     
 \end{proof}
 \noindent This proves the following theorem:
 \begin{theorem}
   Suppose $C_3\subset \DP^2$ is a cubic curve which is   a triple line. Then  we have:
$$\delta(\DP^2,\lambda C_3)= \frac{1-3\lambda}{1- \lambda }\text{ for }\lambda\in \Big[0, \frac{1}{3}\Big].$$
 \end{theorem}
 
 \section{Double conic}
  \begin{lemma}
 Let $C\subset \DP^2$ be a  double conic, $P\in C$ be a smooth point on $C$, and $L$ be the tangent line $L$ at this point. Then
 $$\delta_P(\DP^2,\lambda C)=1\text{ for }\lambda\in\Big[0,\frac{3}{8}\Big].$$
 \end{lemma}
 \begin{proof}
Let $\pi_1:S^1\to S$ be the~blow up of the~point $P$,
with the exceptional divisor $E_1^1$ and $L^1$, $C^1$ are strict transforms of $L$ and $C$ respectively; $\pi_2: S^2\to S^1$ be the~blow up of the~point $E_1^1\cap C^1\cap L^1$,
with the exceptional divisor $E$ and $L^2$, $C^2$ are strict transforms of $L^1$ and $C^1$ respectively; let $\theta:S^2\to\overline{S}$ be  the~contraction of  the~curve $E_1^2$,
and $\sigma$ is the~birational contraction of $\overline{E}=\theta(E)$. We denote the strict transforms of $C$ and $L$ on $\overline{S}$ by $\overline{C}$ and $\overline{L}$. Let $\overline{E}$ be the exceptional divisor of $\sigma$. Note that $\overline{E}$ contains one singular point $\overline{P}$  which is $\frac{1}{2}(1,1)$  singular point.
\begin{center}
 \includegraphics[width=15cm]{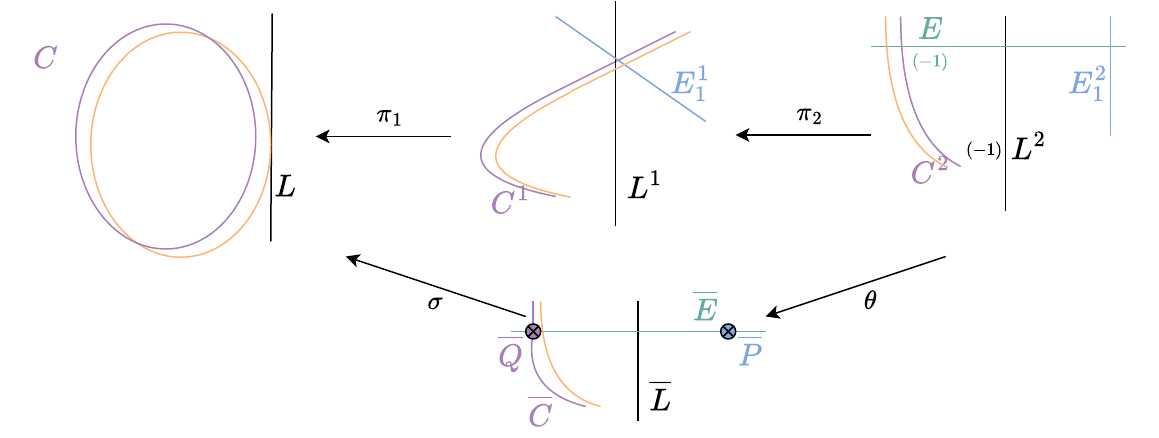}
 \end{center}
The intersections on $\overline{S}$ are given by:
\begin{center}
    \renewcommand{\arraystretch}{1.4}
    \begin{tabular}{|c|c|c|}
    \hline
         & $\overline{E}$ & $\overline{L}$ \\
    \hline
       $\overline{E}$  & $-\frac{1}{2}$ & $1$ \\
    \hline
       $\overline{L}$  & $1$ & $-1$ \\
    \hline
    \end{tabular}
\end{center}
 We have $\sigma^*(L)=\overline{L}+2\overline{E}$, $\sigma^*(C)=\overline{C}+4\overline{E}$, $\sigma^*(K_{\DP^2})=K_{\overline{S}}-4\overline{E}$. Thus, $A_{(\DP^2,\lambda C)}(\overline{E})=3-4\lambda$.
\noindent The Zariski decomposition of the divisor  $-\sigma^*(K_{\DP^2}+\lambda C)-v\overline{E}$ is given by:
\begin{align*}
&&P(v)=
\begin{cases}
-\sigma^*(K_{\DP^2}+\lambda C)-v\overline{E}\text{ if }v\in[0,3-4\lambda],\\
-\sigma^*(K_{\DP^2}+\lambda C)-v\overline{E}-\big(v-(3-4\lambda)\big)\overline{L}\text{ if }v\in[3-4\lambda, 2(3-4\lambda)],\\
\end{cases}\\&&
N(v)=
\begin{cases}
0\text{ if }v\in[0,3-4\lambda],\\
\big(v-(3-4\lambda)\big)\overline{L}\text{ if }v\in[3-4\lambda,2(3-4\lambda)].
\end{cases}
\end{align*}
Then
$$
P(v)^2=
\begin{cases}
(3-4\lambda)^2 -\frac{v^2}{2}\text{ if }v\in[0,3-4\lambda],\\
\frac{(v- 6 + 8\lambda )^2}{2}\text{ if }v\in[3-4\lambda,2(3-4\lambda)],
\end{cases}
\text{ and }
P(v)\cdot \overline{E}=
\begin{cases}
\frac{v}{2}\text{ if }v\in[0,3-4\lambda],\\
3-4\lambda -\frac{v}{2}\text{ if }v\in[3-4\lambda,2(3-4\lambda)],
\end{cases}
$$
Thus
$$S_S(\overline{E})=\frac{1}{(3-4\lambda)^2}\Big(\int_0^{3-4\lambda} (3-4\lambda)^2 -\frac{v^2}{2} dv+\int_{3-4\lambda}^{2(3-4\lambda)} \frac{(v- 6 + 8\lambda)^2}{2} dv\Big)=3-4\lambda$$
so that $\delta_P(\DP^2,\lambda C)\le \frac{3-4\lambda}{3-4\lambda}=1$. For every $O\in \overline{E}$, we get if $O\in \overline{E}\backslash \overline{L}$ or if $O\in \overline{E}\cap \overline{L}$:
$$h(v)\le 
\begin{cases}
\frac{v^2}{8}\text{ if }v\in[0,3-4\lambda],\\
\frac{(v- 6 + 8\lambda)^2}{8}\text{ if }v\in[3-4\lambda, 2(3-4\lambda)],
\end{cases}
\text{or }
h(v)\le \begin{cases}
\frac{v^2}{8}\text{ if }v\in[0,3-4\lambda],\\
\frac{ (v- 6 + 8\lambda) (6 - 8\lambda - 3v)}{8}\text{ if }v\in[3-4\lambda, 2(3-4\lambda)].
\end{cases}
$$
So that
$$S\big(W^{\overline{E}}_{\bullet,\bullet};O\big)\le \frac{2}{(3-4\lambda)^2}\Big(\int_0^{3-4\lambda} \frac{v^2}{8} dv+\int_{3-4\lambda}^{2(3-4\lambda)} \frac{(v- 6 + 8\lambda)^2}{8} dv\Big)=\frac{3-4\lambda}{6}
$$
or
$$S\big(W^{\overline{E}}_{\bullet,\bullet};O\big)\le \frac{2}{(3-4\lambda)^2}\Big(\int_0^{3-4\lambda} \frac{v^2}{8} dv+\int_{3-4\lambda}^{2(3-4\lambda)} \frac{ (v- 6 + 8\lambda) (6 - 8\lambda - 3v)}{8} dv\Big)=\frac{3-4\lambda}{3}
$$
We have
$$
\delta_P(\DP^2,\lambda C)\geqslant\mathrm{min}\Bigg\{1,\inf_{O\in\overline{E}}\frac{A_{\overline{E},\Delta_{\overline{E}}}(O)}{S\big(W^{\overline{E}}_{\bullet,\bullet};O\big)}\Bigg\},
$$
where $\Delta_{\overline{E}}=\frac{1}{2}\overline{P}+2\lambda\overline{Q}$, where $\overline{Q}=\overline{E}\cap \overline{C}$. 
So that
$$
\frac{A_{\overline{E},\Delta_{\overline{E}}}(O)}{S(W_{\bullet,\bullet}^{\overline{E}};O)}=
\left\{\aligned
&\frac{3}{3-4\lambda}\ \mathrm{if}\ O=\overline{E}\cap\overline{L} \text{ or if } O=\overline{P},\\
&\frac{6(1-2\lambda)}{3-4\lambda}\ \mathrm{if}\ O=\overline{Q},\\
&\frac{6}{3-4\lambda}\ \mathrm{otherwise}.
\endaligned
\right.
$$
Thus $\delta_P(\DP^2,\lambda C)=1$ for $\lambda\in \big[0,\frac{3}{8}\big]$.
 \end{proof}

  \section{Conic and a double chord}
  \begin{lemma}
 Let $C\subset \DP^2$ be a conic and a double chord, $P\in C$ be a smooth point  on a conic component. Then
  $$\delta_P(\DP^2,\lambda C)=\frac{3-2\lambda}{3-4 \lambda}\text{ for }\lambda\in\Big[0,\frac{3}{4}\Big].$$
 \end{lemma}
 \begin{proof}
Follows from Lemma \ref{qc-smooth-2}.
 \end{proof}
  \begin{lemma} \label{doubleline-deg4}
  Let $C\subset \DP^2$ be a union of a conic and a double chord, $P\in C$ be a smooth point  on a chord component $L$. Then
 $$\delta_P(\DP^2,\lambda C)=\frac{3(1-2\lambda)}{3- 4 \lambda }\text{ for }\lambda\in\Big[0,\frac{1}{2}\Big].$$
 \end{lemma}
 \begin{proof}
Note that $\delta_P(\DP^2,\lambda C)\le \frac{A_{(\DP^2,\lambda C)}(\widetilde{L})}{S_{(\DP^2,\lambda C)}(\widetilde{L})}=\frac{3(1-2\lambda)}{3- 4 \lambda }$ where $\widetilde{L}$ is a strict transform of a line component of $C$. Let $\pi:\widetilde{S}\to S$ be the~blow up of the~point $P$,
with the exceptional divisor $E$.
\begin{center}
 \includegraphics[width=8cm]{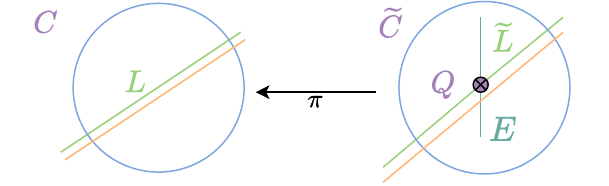}
 \end{center}
We have $\sigma^*(L)=\widetilde{L}+E$, $\pi^*(C)=\widetilde{C}+2E$, $\pi^*(K_{\DP^2})=K_{\widetilde{S}}-E$. Thus, $A_{(\DP^2,\lambda C)}(E)=2-2\lambda$.  The Zariski decomposition of the divisor  $-\pi^*(K_{\DP^2}+\lambda C)-vE$ is given by:
\begin{align*}
P(v)=-\pi^*(K_{\DP^2}+\lambda C)-vE \text{ and }N(v)=0\text{ if }v\in[0,3- 4 \lambda ].
\end{align*}
Then
$$
P(v)^2=(4\lambda - 3- v )(v-(3- 4\lambda))\text{ and }P(v)\cdot E=  v \text{ if }v\in[0,3- 4 \lambda ].
$$
Thus
$$S_{(\DP^2,\lambda C)}(E)=\frac{1}{(3- 4 \lambda )^2}\Big(\int_0^{3-4 \lambda } (4\lambda - 3- v )(v-(3- 4\lambda)) dv\Big)=\frac{2(3-4\lambda)}{3}$$
so that $\delta_P(\DP^2,\lambda C)\le \frac{3}{2}\cdot \frac{2-2\lambda}{3-4 \lambda } = \frac{3-3\lambda}{3-4 \lambda }$. 
For every $O\in E$ we get:
$$h(v) = \frac{v^2}{2}\text{ if }v\in[0,3-4\lambda ].
$$
So that
$$S\big(W^{E}_{\bullet,\bullet};O\big)= \frac{2}{(3-4\lambda )^2}\Big(\int_0^{3-4\lambda } \frac{v^2}{2} dv\Big)=\frac{3-4 \lambda}{3}
$$
We have
$$
\delta_P(\DP^2,\lambda C)\geqslant\mathrm{min}\Bigg\{ \frac{3-3\lambda}{3-4 \lambda } ,\inf_{O\in E}\frac{A_{E,\Delta_{E}}(O)}{S\big(W^{E}_{\bullet,\bullet};O\big)}\Bigg\},
$$
where $\Delta_{E}=2\lambda Q$ where $Q=\widetilde{L}|_E$. 
So that
$$
\frac{A_{\overline{E},\Delta_{\overline{E}}}(O)}{S(W_{\bullet,\bullet}^{\overline{E}};O)}=
\left\{\aligned
&\frac{3(1-2\lambda)}{3-4 \lambda }\ \mathrm{if}\ O=Q,\\
&\frac{3}{3-4 \lambda }\ \mathrm{otherwise}.
\endaligned
\right.
$$
Thus $\delta_P(\DP^2,\lambda C)=\frac{3(1-2\lambda)}{3- 4 \lambda }$ for $\lambda\in\big[0,\frac{1}{2}\big]$.     
 \end{proof}

  \begin{lemma}
  Let $C\subset \DP^2$ be a union of a conic and a double chord, $P\in C$ be a singular point on $C$. Then
 $$\delta_P(\DP^2,\lambda C)=\frac{3(1-2\lambda)}{3- 4 \lambda }\text{ for }\lambda\in\Big[0,\frac{1}{2}\Big].$$
 \end{lemma}
 \begin{proof}
 Note that $\delta_P(\DP^2,\lambda C)\le \frac{A_{(\DP^2,\lambda C)}(\widetilde{L})}{S_{(\DP^2,\lambda C)}(\widetilde{L})}=\frac{3(1-2\lambda)}{3- 4 \lambda }$ where $\widetilde{L}$ is a strict transform of a line component of $C$. Let $\pi:\widetilde{S}\to S$ be the~blow up of the~point $P$,
with the exceptional divisor $E$.
\begin{center}
 \includegraphics[width=10cm]{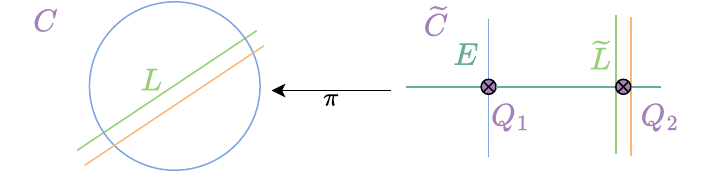}
 \end{center}
We have $\sigma^*(L)=\widetilde{L}+E$, $\pi^*(C)=\widetilde{C}+3E$, $\pi^*(K_{\DP^2})=K_{\widetilde{S}}-E$. Thus, $A_{(\DP^2,\lambda C)}(E)=2-3\lambda$.  The Zariski decomposition of the divisor  $-\pi^*(K_{\DP^2}+\lambda C)-vE$ is given by:
\begin{align*}
P(v)=-\pi^*(K_{\DP^2}+\lambda C)-vE \text{ and }N(v)=0\text{ if }v\in[0,3- 4 \lambda ].
\end{align*}
Then
$$
P(v)^2=(4\lambda - 3- v )(v-(3- 4\lambda))\text{ and }P(v)\cdot E=  v \text{ if }v\in[0,3- 4 \lambda ].
$$
Thus
$$S_{(\DP^2,\lambda C)}(E)=\frac{1}{(3- 4 \lambda )^2}\Big(\int_0^{3-4 \lambda } (4\lambda - 3- v )(v-(3- 4\lambda)) dv\Big)=\frac{2(3-4\lambda)}{3}$$
so that $\delta_P(\DP^2,\lambda C)\le \frac{3}{2}\cdot \frac{2-3\lambda}{3-4 \lambda }$. 
For every $O\in E$ we get:
$$h(v) = \frac{v^2}{2}\text{ if }v\in[0,3-4\lambda ].
$$
So that
$$S\big(W^{E}_{\bullet,\bullet};O\big)= \frac{2}{(3-4\lambda )^2}\Big(\int_0^{3-4\lambda } \frac{v^2}{2} dv\Big)=\frac{3-4 \lambda}{3}
$$
We have
$$
\delta_P(\DP^2,\lambda C)\geqslant\mathrm{min}\Bigg\{\frac{3}{2}\cdot \frac{2-3\lambda}{3-4 \lambda } ,\inf_{O\in E}\frac{A_{E,\Delta_{E}}(O)}{S\big(W^{E}_{\bullet,\bullet};O\big)}\Bigg\},
$$
where $\Delta_{E}=\lambda Q_1+2\lambda Q_2$ where $Q_1+2Q_2=\widetilde{C}|_E$. 
So that
$$
\frac{A_{\overline{E},\Delta_{\overline{E}}}(O)}{S(W_{\bullet,\bullet}^{\overline{E}};O)}=
\left\{\aligned
&\frac{3(1-\lambda)}{3-4 \lambda }\ \mathrm{if}\ O=Q_1,\\
&\frac{3(1-2\lambda)}{3-4 \lambda }\ \mathrm{if}\ O=Q_2,\\
&\frac{3}{3-4 \lambda }\ \mathrm{otherwise}.
\endaligned
\right.
$$
Thus $\delta_P(\DP^2,\lambda C)= \frac{3(1-2\lambda)}{3- 4 \lambda }$ for $\lambda\in\big[0,\frac{1}{2}\big]$.     
 \end{proof}
 \noindent This proves the following theorem:
 \begin{theorem}
   Suppose $C_4\subset \DP^2$ is a quartic curve which is  a union of a conic and a double chord. Then  we have:
$$\delta(\DP^2,\lambda C_4)=\frac{3(1-2\lambda)}{3- 4 \lambda }\text{ for }\lambda\in \Big[0, \frac{1}{2}\Big].$$
 \end{theorem}
  \section{Conic and a double tangent line}
  \begin{lemma}
 Let $C\subset \DP^2$ be a conic and a double tangent line, $P\in C$ be a smooth point  on a conic component. Then
  $$\delta_P(\DP^2,\lambda C)=\frac{3-2\lambda}{3-4 \lambda}\text{ for }\lambda\in\Big[0,\frac{3}{4}\Big].$$
 \end{lemma}
 \begin{proof}
Follows from Lemma \ref{qc-smooth-2}.
 \end{proof}
  \begin{lemma}
 Let $C\subset \DP^2$ be a conic and a double tangent line, $P\in C$ be a smooth point  on a double line $L$. Then
  $$\delta_P(\DP^2,\lambda C)=\frac{3(1-2\lambda)}{3-4 \lambda}\text{ for }\lambda\in\Big[0,\frac{1}{2}\Big].$$
 \end{lemma}
 \begin{proof}
Follows from Lemma \ref{doubleline-deg4}.
 \end{proof}
  \begin{lemma}
  Let $C\subset \DP^2$ be a union of a conic and a double tangent line, $P\in C$ be a singular point. Then
 $$\delta_P(\DP^2,\lambda C)=\frac{3(1-2\lambda)}{3- 4 \lambda }\text{ for }\lambda\in\Big[0,\frac{1}{2}\Big].$$
 \end{lemma}
  \begin{proof}
Note that $\delta_P(\DP^2,\lambda C)\le \frac{A_{(\DP^2,\lambda C)}(\widetilde{L})}{S_{(\DP^2,\lambda C)}(\widetilde{L})}=\frac{3(1-2\lambda)}{3- 4 \lambda }$ where $\widetilde{L}$ is a strict transform of a double line component of $C$. Let $\pi_1:S^1\to S$ be the~blow up of the~point $P$,
with the exceptional divisor $E_1^1$ and $L^1$, $C^1$ are strict transforms of $L$ and $C$ respectively; $\pi_2: S^2\to S^1$ be the~blow up of the~point $E_1^1\cap C^1\cap L^1$,
with the exceptional divisor $E$ and $L^2$, $C^2$ are strict transforms of $L^1$ and $C^1$ respectively; let $\theta:S^2\to\overline{S}$ be  the~contraction of  the~curve $E_1^2$,
and $\sigma$ is the~birational contraction of $\overline{E}=\theta(E)$. We denote the strict transforms of $C$ and $L$ on $\overline{S}$ by $\overline{C}$ and $\overline{L}$. Let $\overline{E}$ be the exceptional divisor of $\sigma$. Note that $\overline{E}$ contains one singular point $\overline{P}$  which is $\frac{1}{2}(1,1)$  singular point.
\begin{center}
 \includegraphics[width=15cm]{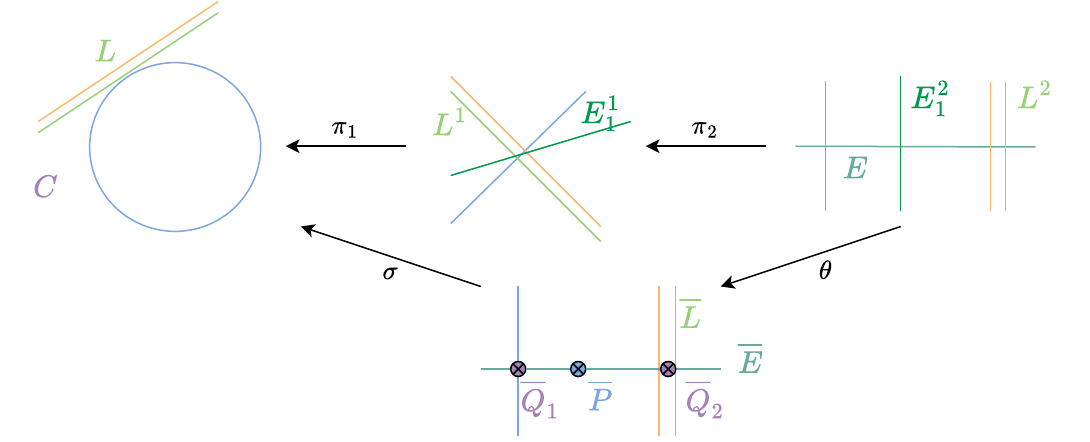}
 \end{center}
The intersections on $\overline{S}$ are given by:
\begin{center}
    \renewcommand{\arraystretch}{1.4}
    \begin{tabular}{|c|c|c|}
    \hline
         & $\overline{E}$ & $\overline{L}$ \\
    \hline
       $\overline{E}$  & $-\frac{1}{2}$ & $1$ \\
    \hline
       $\overline{L}$  & $1$ & $-1$ \\
    \hline
    \end{tabular}
\end{center}
 We have $\sigma^*(L)=\overline{L}+2\overline{E}$, $\sigma^*(C)=\overline{C}+4\overline{E}$, $\sigma^*(K_{\DP^2})=K_{\overline{S}}-2\overline{E}$. Thus, $A_{(\DP^2,\lambda C)}(\overline{E})=3-4\lambda$.
\noindent The Zariski decomposition of the divisor  $-\sigma^*(K_{\DP^2}+\lambda C)-v\overline{E}$ is given by:
\begin{align*}
&&P(v)=
\begin{cases}
-\sigma^*(K_{\DP^2}+\lambda C)-v\overline{E}\text{ if }v\in[0,3-4\lambda],\\
-\sigma^*(K_{\DP^2}+\lambda C)-v\overline{E}-\big(v-(3-4\lambda)\big)\overline{L}\text{ if }v\in[3-4\lambda, 2(3-4\lambda)],\\
\end{cases}\\&&
N(v)=
\begin{cases}
0\text{ if }v\in[0,3-4\lambda],\\
\big(v-(3-4\lambda)\big)\overline{L}\text{ if }v\in[3-4\lambda,2(3-4\lambda)].
\end{cases}
\end{align*}
Then
$$
P(v)^2=
\begin{cases}
(3-4\lambda)^2 -\frac{v^2}{2}\text{ if }v\in[0,3-4\lambda],\\
\frac{(v- 6 + 8\lambda )^2}{2}\text{ if }v\in[3-4\lambda,2(3-4\lambda)],
\end{cases}
\text{ and }
P(v)\cdot \overline{E}=
\begin{cases}
\frac{v}{2}\text{ if }v\in[0,3-4\lambda],\\
3-4\lambda -\frac{v}{2}\text{ if }v\in[3-4\lambda,2(3-4\lambda)],
\end{cases}
$$
Thus
$$S_S(\overline{E})=\frac{1}{(3-4\lambda)^2}\Big(\int_0^{3-4\lambda} (3-4\lambda)^2 -\frac{v^2}{2} dv+\int_{3-4\lambda}^{2(3-4\lambda)} \frac{(v- 6 + 8\lambda)^2}{2} dv\Big)=3-4\lambda$$
so that $\delta_P(\DP^2,\lambda C)\le 1$. For every $O\in \overline{E}$, we get if $O\in \overline{E}\backslash \overline{L}$ or if $O\in \overline{E}\cap \overline{L}$:
$$h(v)\le 
\begin{cases}
\frac{v^2}{8}\text{ if }v\in[0,3-4\lambda],\\
\frac{(v- 6 + 8\lambda)^2}{8}\text{ if }v\in[3-4\lambda, 2(3-4\lambda)],
\end{cases}
\text{ or }
h(v)\le \begin{cases}
\frac{v^2}{8}\text{ if }v\in[0,3-4\lambda],\\
\frac{ (v- 6 + 6\lambda) (6 - 8\lambda - 3v)}{8}\text{ if }v\in[3-4\lambda, 2(3-4\lambda)].
\end{cases}
$$
So that
$$S\big(W^{\overline{E}}_{\bullet,\bullet};O\big)\le \frac{2}{(3-4\lambda)^2}\Big(\int_0^{3-4\lambda} \frac{v^2}{8} dv+\int_{3-4\lambda}^{2(3-4\lambda)} \frac{(v- 6 + 6\lambda)^2}{8} dv\Big)=\frac{3-4\lambda}{6}
$$
or
$$S\big(W^{\overline{E}}_{\bullet,\bullet};O\big)\le \frac{2}{(3-4\lambda)^2}\Big(\int_0^{3-4\lambda} \frac{v^2}{8} dv+\int_{3-4\lambda}^{2(3-4\lambda)} \frac{ (v- 6 + 6\lambda) (6 - 8\lambda - 3v)}{8} dv\Big)=\frac{3-4\lambda}{3}
$$
We have
$$
\delta_P(\DP^2,\lambda C)\geqslant\mathrm{min}\Bigg\{1,\inf_{O\in\overline{E}}\frac{A_{\overline{E},\Delta_{\overline{E}}}(O)}{S\big(W^{\overline{E}}_{\bullet,\bullet};O\big)}\Bigg\},
$$
where $\Delta_{\overline{E}}=\frac{1}{2}\overline{P}+\lambda Q_1+2\lambda Q_2$ where $Q_1+2Q_2=\widetilde{C}|_E$.  
So that
$$
\frac{A_{\overline{E},\Delta_{\overline{E}}}(O)}{S(W_{\bullet,\bullet}^{\overline{E}};O)}=
\left\{\aligned
&\frac{3}{3-4\lambda}\ \mathrm{if}\ \text{ if } O=\overline{P},\\
&\frac{6(1-\lambda)}{3-4\lambda}\ \mathrm{if}\ O=\overline{Q}_1,\\
&\frac{3(1-2\lambda)}{3-4\lambda}\ \mathrm{if}\ O=\overline{Q}_2,\\
&\frac{6}{3-4\lambda}\ \mathrm{otherwise}.
\endaligned
\right.
$$
Thus $\delta_P(\DP^2,\lambda C)= \frac{3(1-2\lambda)}{3- 4 \lambda } $ for $\lambda\in\big[0,\frac{1}{2}\big]$.
 \end{proof}
 \noindent This proves the following theorem:
 \begin{theorem}
   Suppose $C_4\subset \DP^2$ is a quartic curve which is   a union of a conic and a double tangent line. Then  we have:
$$\delta(\DP^2,\lambda C_4)= \frac{3(1-2\lambda)}{3- 4 \lambda }\text{ for }\lambda\in \Big[0, \frac{1}{2}\Big].$$
 \end{theorem}
 \section{Double line and two other lines in general position}
  \begin{lemma}
 Let $C\subset \DP^2$ be a double line and two other lines in general position, $P\in C$ be a smooth point. Then
  $$\delta_P(\DP^2,\lambda C)=\frac{3(1-\lambda)}{3- 4 \lambda }\text{ for }\lambda\in\Big[0,\frac{3}{4}\Big].$$
 \end{lemma}
 \begin{proof}
     Follows from Lemma \ref{line}.
 \end{proof}
  \begin{lemma}
 Let $C\subset \DP^2$ be a double line and two other lines in general position, $P\in C$ be a $A_1$ point. Then
  $$\delta_P(\DP^2,\lambda C)=\frac{3(1-\lambda)}{3-4 \lambda }\text{ for }\lambda\in\Big[0,\frac{3}{4}\Big].$$
 \end{lemma}
 \begin{proof}
     Follows from Lemma \ref{A1-points}.
 \end{proof}
 \begin{lemma}
 Let $C\subset \DP^2$ be a double line and two other lines in general position, $P\in C$ be a smooth point  on a double line $L$. Then
  $$\delta_P(\DP^2,\lambda C)=\frac{3(1-2\lambda)}{3-4 \lambda}\text{ for }\lambda\in\Big[0,\frac{1}{2}\Big].$$
 \end{lemma}
 \begin{proof}
Follows from Lemma \ref{doubleline-deg4}.
 \end{proof}
 \begin{lemma}
 Let $C\subset \DP^2$ be a double line and two other lines in general position, $P\in C$ be a singular point  on a double line $L$. Then
  $$\delta_P(\DP^2,\lambda C)=\frac{3(1-2\lambda)}{3-4 \lambda}\text{ for }\lambda\in\Big[0,\frac{1}{2}\Big].$$
 \end{lemma}

  \begin{proof}
Let $\pi:\widetilde{S}\to S$ be the~blow up of the~point $P$,
with the exceptional divisor $E$. We denote the strict transform of $C$  on $\widetilde{S}$ by $\widetilde{C}$.
\begin{center}
 \includegraphics[width=9cm]{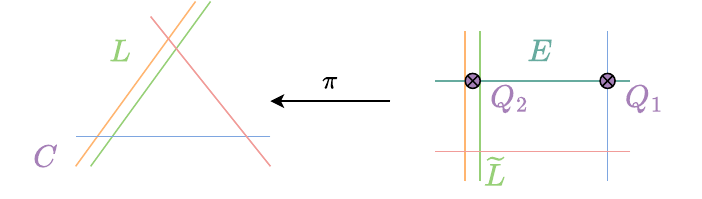}
 \end{center}
Note that $\delta_P(\DP^2,\lambda C)\le \frac{A_{(\DP^2,\lambda C)}(\widetilde{L})}{S_{(\DP^2,\lambda C)}(\widetilde{L})}=\frac{3(1-2\lambda)}{3- 4 \lambda }$ where $\widetilde{L}$ is a strict transform of a double line component of $C$. We have $\pi^*(C)=\widetilde{C}+3E$, $\pi^*(K_{\DP^2})=K_{\widetilde{S}}-E$. Thus, $A_{(\DP^2,\lambda C)}(E)=2-3\lambda$.
\noindent The Zariski decomposition of the divisor  $-\pi^*(K_{\DP^2}+\lambda C)-vE$ is given by:
\begin{align*}
P(v)=-\pi^*(K_{\DP^2}+\lambda C)-vE \text{ and }N(v)=0\text{ if }v\in[0,3- 4\lambda ].
\end{align*}
Then
$$
P(v)^2=\big(v-(3- 4\lambda )\big)\big(v+(3- 4\lambda )\big)\text{ and }P(v)\cdot E= v\text{ if }v\in[0,3- 4\lambda ].
$$
Thus
$$S_{(\DP^2, \lambda C)}(E)=\frac{1}{(3- 4\lambda )^2}\Big(\int_0^{3- 4\lambda } \big(v-(3- 4\lambda )\big)\big(v+(3- 4\lambda )\big) dv\Big)=\frac{2(3- 4\lambda )}{3}$$
so that $\delta_P(\DP^2,\lambda C)\le \frac{3}{2}\cdot \frac{2-3\lambda}{3- 4\lambda } $. For every $O\in E$ we get:
$$h(v) = \frac{v^2}{2}\text{ if }v\in[0,3- 4\lambda ].
$$
So that
$$S\big(W^{E}_{\bullet,\bullet};O\big)= \frac{2}{(3- 4\lambda )^2}\Big(\int_0^{3- 4\lambda } \frac{v^2}{2} dv\Big)=\frac{3- 4\lambda }{3}$$
We have
$$
\delta_P(\DP^2,\lambda C)\geqslant\mathrm{min}\Bigg\{ \frac{3}{2}\cdot \frac{2-3\lambda}{3- 4\lambda } ,\inf_{O\in E}\frac{A_{E,\Delta_{E}}(O)}{S\big(W^{E}_{\bullet,\bullet};O\big)}\Bigg\},
$$
where $\Delta_{E}=\lambda Q_1+2\lambda Q_2$ where $ Q_1+ 2Q_2=\widetilde{C}|_E$. 
So that
$$
\frac{A_{\overline{E},\Delta_{\overline{E}}}(O)}{S(W_{\bullet,\bullet}^{\overline{E}};O)}=
\left\{\aligned
&\frac{3(1-\lambda)}{3- 4\lambda }\ \mathrm{if}\ O=Q_1,\\
&\frac{3(1-2\lambda)}{3- 4\lambda }\ \mathrm{if}\ O=Q_2,\\
&\frac{3}{3- 4\lambda }\ \mathrm{otherwise}.
\endaligned
\right.
$$
Thus $\delta_P(\DP^2,\lambda C)= \frac{3(1-2\lambda)}{3- 4\lambda }$.
 \end{proof}
\noindent This proves the following theorem:
 \begin{theorem}
   Suppose $C_4\subset \DP^2$ is a quartic curve which is  a union of a double line and two other lines in general position. Then  we have:
$$\delta(\DP^2,\lambda C_4)= \frac{3(1-2\lambda)}{3- 4 \lambda }\text{ for }\lambda\in \Big[0, \frac{1}{2}\Big].$$
 \end{theorem}
 \section{Double line and two different lines such that all of them are concurrent}
  \begin{lemma}
 Let $C\subset \DP^2$ be a double line and two different lines such that all of them are concurrent, $P\in C$ be a smooth point. Then
  $$\delta_P(\DP^2,\lambda C)=\frac{3(1-\lambda)}{3- 4 \lambda }\text{ for }\lambda\in\Big[0,\frac{3}{4}\Big].$$
 \end{lemma}
 \begin{proof}
     Follows from Lemma \ref{line}.
 \end{proof}
 
 \begin{lemma}
 Let $C\subset \DP^2$ be a double line and two different lines such that all of them are concurrent, $P\in C$ be a smooth point  on a double line $L$. Then
  $$\delta_P(\DP^2,\lambda C)=\frac{3(1-2\lambda)}{3-4 \lambda}\text{ for }\lambda\in\Big[0,\frac{1}{2}\Big].$$
 \end{lemma}
 \begin{proof}
Follows from Lemma \ref{doubleline-deg4}.
 \end{proof}
 \begin{lemma}
 Let $C\subset \DP^2$ be a double line and two different lines such that all of them are concurrent, $P\in C$ be a singular point on $C$. Then
  $$\delta_P(\DP^2,\lambda C)=\frac{3(1-2\lambda)}{3-4 \lambda}\text{ for }\lambda\in\Big[0,\frac{1}{2}\Big].$$
 \end{lemma}
 \begin{proof}
Let $\pi:\widetilde{S}\to S$ be the~blow up of the~point $P$,
with the exceptional divisor $E$. We denote the strict transform of $C$  on $\widetilde{S}$ by $\widetilde{C}$.
\begin{center}
 \includegraphics[width=9cm]{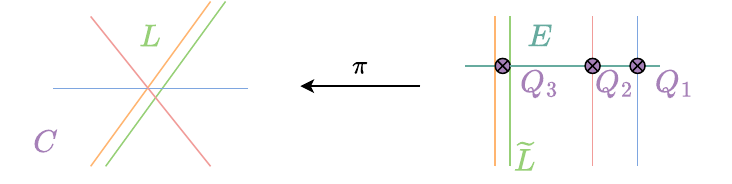}
 \end{center}
 We have $\pi^*(C)=\widetilde{C}+4E$, $\pi^*(K_{\DP^2})=K_{\widetilde{S}}-E$. Thus, $A_{(\DP^2,\lambda C)}(E)=2-4\lambda$.
\noindent The Zariski decomposition of the divisor  $-\pi^*(K_{\DP^2}+\lambda C)-vE$ is given by:
\begin{align*}
P(v)=-\pi^*(K_{\DP^2}+\lambda C)-vE \text{ and }N(v)=0\text{ if }v\in[0,3-4 \lambda ].
\end{align*}
Then
$$
P(v)^2=\big(v-(3-4 \lambda )\big)\big(v+(3-4 \lambda )\big)\text{ and }P(v)\cdot E= v\text{ if }v\in[0,3-4 \lambda ].
$$
Thus
$$S_{(\DP^2, \lambda C)}(E)=\frac{1}{(3-4 \lambda )^2}\Big(\int_0^{3-4 \lambda } \big(v-(3-4 \lambda )\big)\big(v+(3-4 \lambda )\big) dv\Big)=\frac{2(3-4 \lambda )}{3}$$
so that $\delta_P(\DP^2,\lambda C)\le \frac{3}{2}\cdot\frac{2-4\lambda}{3-4 \lambda }=\frac{3(1-2\lambda)}{3-4 \lambda }$. For every $O\in E$ we get:
$$h(v) = \frac{v^2}{2}\text{ if }v\in[0,3-4 \lambda ].
$$
So that
$$S\big(W^{E}_{\bullet,\bullet};O\big)= \frac{2}{(3-4 \lambda )^2}\Big(\int_0^{3-4 \lambda } \frac{v^2}{2} dv\Big)=\frac{3-4 \lambda }{3}\le \frac{2}{3}\cdot\frac{3-4 \lambda }{2-3\lambda}
$$
We have
$$
\delta_P(\DP^2,\lambda C)\geqslant\mathrm{min}\Bigg\{  \frac{3(1-2\lambda)}{3-4 \lambda },\inf_{O\in E}\frac{A_{E,\Delta_{E}}(O)}{S\big(W^{E}_{\bullet,\bullet};O\big)}\Bigg\},
$$
where $\Delta_{E}=\lambda Q_1+\lambda Q_2+2\lambda Q_3$ where $Q_1+Q_2+2Q_3=\widetilde{C}|_E$. 
So that
$$
\frac{A_{\overline{E},\Delta_{\overline{E}}}(O)}{S(W_{\bullet,\bullet}^{\overline{E}};O)}=
\left\{\aligned
&\frac{3(1-\lambda)}{3-4 \lambda }\ \mathrm{if}\ O\in\{Q_1,Q_2\},\\
&\frac{3(1-2\lambda)}{3-4 \lambda }\ \mathrm{if}\ O=Q_3,\\
&\frac{3}{3-4 \lambda }\ \mathrm{otherwise}.
\endaligned
\right.
$$
Thus $\delta_P(\DP^2,\lambda C)=\frac{3(1-2\lambda)}{3-4 \lambda }$ for $\lambda\in\big[0,\frac{1}{2}\big]$.
 \end{proof}
 \noindent This proves the following theorem:
  \begin{theorem}
   Suppose $C_4\subset \DP^2$ is a quartic curve which is  a union ofa double line and two concurrent lines. Then  we have:
$$\delta(\DP^2,\lambda C_4)= \frac{3(1-2\lambda)}{3- 4 \lambda }\text{ for }\lambda\in \Big[0, \frac{1}{2}\Big].$$
 \end{theorem}
  \section{Two double lines}

 \begin{lemma}
 Let $C\subset \DP^2$ be two  double lines, $P\in C$ be a smooth point  on a double line $L$. Then
  $$\delta_P(\DP^2,\lambda C)=\frac{3(1-2\lambda)}{3-4 \lambda}\text{ for }\lambda\in\Big[0,\frac{1}{2}\Big].$$
 \end{lemma}
 \begin{proof}
Follows from Lemma \ref{doubleline-deg4}.
 \end{proof}

 \begin{lemma}
 Let $C\subset \DP^2$ be two double lines, $P\in C$ be a singular point $C$ . Then
  $$\delta_P(\DP^2,\lambda C)=\frac{3(1-2\lambda)}{3-4 \lambda}\text{ for }\lambda\in\Big[0,\frac{1}{2}\Big].$$
 \end{lemma}
 \begin{proof}
Let $\pi:\widetilde{S}\to S$ be the~blow up of the~point $P$,
with the exceptional divisor $E$. We denote the strict transform of $C$  on $\widetilde{S}$ by $\widetilde{C}$.
\begin{center}
 \includegraphics[width=9cm]{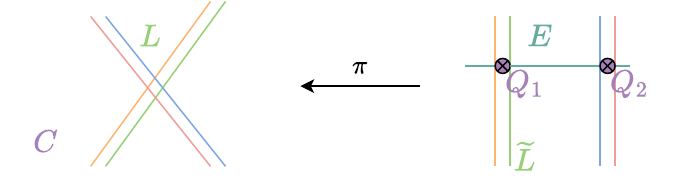}
 \end{center}
 We have $\pi^*(C)=\widetilde{C}+4E$, $\pi^*(K_{\DP^2})=K_{\widetilde{S}}-E$. Thus, $A_{(\DP^2,\lambda C)}(E)=2-4\lambda$.
\noindent The Zariski decomposition of the divisor  $-\pi^*(K_{\DP^2}+\lambda C)-vE$ is given by:
\begin{align*}
P(v)=-\pi^*(K_{\DP^2}+\lambda C)-vE \text{ and }N(v)=0\text{ if }v\in[0,3- 4 \lambda ].
\end{align*}
Then
$$
P(v)^2=\big(v-(3- 4 \lambda )\big)\big(v+(3- 4 \lambda )\big)\text{ and }P(v)\cdot E= v\text{ if }v\in[0,3- 4 \lambda ].
$$
Thus
$$S_{(\DP^2, \lambda C)}(E)=\frac{1}{(3- 4 \lambda )^2}\Big(\int_0^{3- 4 \lambda } \big(v-(3- 4 \lambda )\big)\big(v+(3- 4 \lambda )\big) dv\Big)=\frac{2(3- 4 \lambda )}{3}$$
so that $\delta_P(\DP^2,\lambda C)\le \frac{3}{2}\cdot \frac{2-4\lambda}{3- 4 \lambda }= \frac{3(1-2\lambda)}{3- 4 \lambda } $. For every $O\in E$ we get:
$$h(v) = \frac{v^2}{2}\text{ if }v\in[0,3- 4 \lambda ].
$$
So that
$$S\big(W^{E}_{\bullet,\bullet};O\big)= \frac{2}{(3- 4 \lambda )^2}\Big(\int_0^{33- 4 \lambda } \frac{v^2}{2} dv\Big)=\frac{3- 4 \lambda }{3}$$
We have
$$
\delta_P(\DP^2,\lambda C)\geqslant\mathrm{min}\Bigg\{ \frac{3(1-2\lambda)}{3- 4 \lambda } ,\inf_{O\in E}\frac{A_{E,\Delta_{E}}(O)}{S\big(W^{E}_{\bullet,\bullet};O\big)}\Bigg\},
$$
where $\Delta_{E}=\lambda 2Q_1+2\lambda Q_2$ where $ 2Q_1+ 2Q_2=\widetilde{C}|_E$. 
So that
$$
\frac{A_{\overline{E},\Delta_{\overline{E}}}(O)}{S(W_{\bullet,\bullet}^{\overline{E}};O)}=
\left\{\aligned
&\frac{3(1-2\lambda)}{3- 4 \lambda }\ \mathrm{if}\ O\in\{Q_1,Q_2\},\\
&\frac{3}{3- 4 \lambda }\ \mathrm{otherwise}.
\endaligned
\right.
$$
Thus $\delta_P(\DP^2,\lambda C)= \frac{3(1-2\lambda)}{3- 4 \lambda }$.
 \end{proof}
 \noindent This proves the following theorem:
  \begin{theorem}
   Suppose $C_4\subset \DP^2$ is a quartic curve which is  a union of two double lines. Then  we have:
$$\delta(\DP^2,\lambda C_4)= \frac{3(1-2\lambda)}{3- 4 \lambda }\text{ for }\lambda\in \Big[0, \frac{1}{2}\Big].$$
 \end{theorem}
  \section{Triple line and another line}
  \begin{lemma}
 Let $C\subset \DP^2$ be a triple line and another line, $P\in C$ be a smooth point. Then
  $$\delta_P(\DP^2,\lambda C)=\frac{3(1-\lambda)}{3- 4 \lambda }\text{ for }\lambda\in\Big[0,\frac{3}{4}\Big].$$
 \end{lemma}
 \begin{proof}
     Follows from Lemma \ref{line}.
 \end{proof}
   \begin{lemma} 
  Let $C\subset \DP^2$  be a triple line and another line, $P\in C$ be a triple line. Then
 $$\delta_P(\DP^2,\lambda C)=\frac{3(1-3\lambda)}{3- 4 \lambda }\text{ for }\lambda\in\Big[0,\frac{1}{3}\Big].$$
 \end{lemma}
 \begin{proof}
Note that $\delta_P(\DP^2,\lambda C)\le \frac{A_{(\DP^2,\lambda C)}(\widetilde{L})}{S_{(\DP^2,\lambda C)}(\widetilde{L})}=\frac{3(1-3\lambda)}{3- 4 \lambda }$ where $\widetilde{L}$ is a strict transform of a triple line component of $C$. Let $\pi:\widetilde{S}\to S$ be the~blow up of the~point $P$,
with the exceptional divisor $E$.
\begin{center}
 \includegraphics[width=8cm]{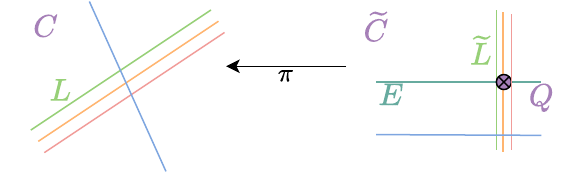}
 \end{center}
We have $\sigma^*(L)=\widetilde{L}+E$, $\pi^*(C)=\widetilde{C}+3E$, $\pi^*(K_{\DP^2})=K_{\widetilde{S}}-E$. Thus, $A_{(\DP^2,\lambda C)}(E)=2-3\lambda$.  The Zariski decomposition of the divisor  $-\pi^*(K_{\DP^2}+\lambda C)-vE$ is given by:
\begin{align*}
P(v)=-\pi^*(K_{\DP^2}+\lambda C)-vE \text{ and }N(v)=0\text{ if }v\in[0,3- 4 \lambda ].
\end{align*}
Then
$$
P(v)^2=(4\lambda - 3- v )(v-(3- 4\lambda))\text{ and }P(v)\cdot E=  v \text{ if }v\in[0,3- 4 \lambda ].
$$
Thus
$$S_{(\DP^2,\lambda C)}(E)=\frac{1}{(3- 4 \lambda )^2}\Big(\int_0^{3-4 \lambda } (4\lambda - 3- v )(v-(3- 4\lambda)) dv\Big)=\frac{2(3-4\lambda)}{3}$$
so that $\delta_P(\DP^2,\lambda C)\le \frac{3}{2}\cdot \frac{2-3\lambda}{3-4 \lambda }$. 
For every $O\in E$ we get:
$$h(v) = \frac{v^2}{2}\text{ if }v\in[0,3-4\lambda ].
$$
So that
$$S\big(W^{E}_{\bullet,\bullet};O\big)= \frac{2}{(3-4\lambda )^2}\Big(\int_0^{3-4\lambda } \frac{v^2}{2} dv\Big)=\frac{3-4 \lambda}{3}
$$
We have
$$
\delta_P(\DP^2,\lambda C)\geqslant\mathrm{min}\Bigg\{ \frac{3}{2}\cdot \frac{2-3\lambda}{3-4 \lambda } ,\inf_{O\in E}\frac{A_{E,\Delta_{E}}(O)}{S\big(W^{E}_{\bullet,\bullet};O\big)}\Bigg\},
$$
where $\Delta_{E}=3\lambda Q$ where $Q=\widetilde{L}|_E$. 
So that
$$
\frac{A_{\overline{E},\Delta_{\overline{E}}}(O)}{S(W_{\bullet,\bullet}^{\overline{E}};O)}=
\left\{\aligned
&\frac{3(1-3\lambda)}{3-4 \lambda }\ \mathrm{if}\ O=Q,\\
&\frac{3}{3-4 \lambda }\ \mathrm{otherwise}.
\endaligned
\right.
$$
Thus $\delta_P(\DP^2,\lambda C)=\frac{3(1-3\lambda)}{3- 4 \lambda }$ for $\lambda\in\big[0,\frac{1}{3}\big]$.     
 \end{proof}

 \begin{lemma}
 Let $C\subset \DP^2$ be a triple line and another line, $P\in C$ be a singular point  on a triple line. Then
  $$\delta_P(\DP^2,\lambda C)=\frac{3(1-3\lambda)}{3-4 \lambda}\text{ for }\lambda\in\Big[0,\frac{1}{3}\Big].$$
 \end{lemma}

  \begin{proof}
Note that $\delta_P(\DP^2,\lambda C)\le \frac{A_{(\DP^2,\lambda C)}(\widetilde{L})}{S_{(\DP^2,\lambda C)}(\widetilde{L})}=\frac{3(1-3\lambda)}{3- 4 \lambda }$ where $\widetilde{L}$ is a strict transform of a triple line component of $C$. Let $\pi:\widetilde{S}\to S$ be the~blow up of the~point $P$,
with the exceptional divisor $E$. We denote the strict transform of $C$  on $\widetilde{S}$ by $\widetilde{C}$.
\begin{center}
 \includegraphics[width=9cm]{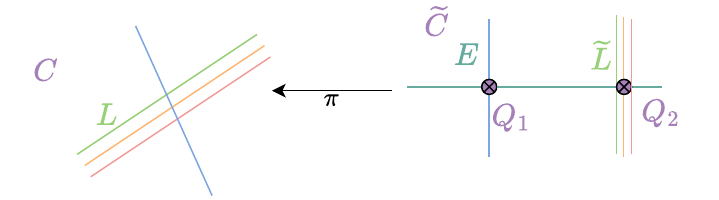}
 \end{center}
 We have $\pi^*(C)=\widetilde{C}+4E$, $\pi^*(K_{\DP^2})=K_{\widetilde{S}}-E$. Thus, $A_{(\DP^2,\lambda C)}(E)=2-4\lambda$.
\noindent The Zariski decomposition of the divisor  $-\pi^*(K_{\DP^2}+\lambda C)-vE$ is given by:
\begin{align*}
P(v)=-\pi^*(K_{\DP^2}+\lambda C)-vE \text{ and }N(v)=0\text{ if }v\in[0,3- 4 \lambda ].
\end{align*}
Then
$$
P(v)^2=\big(v-(3- 4 \lambda )\big)\big(v+(3- 4 \lambda )\big)\text{ and }P(v)\cdot E= v\text{ if }v\in[0,3- 4 \lambda ].
$$
Thus
$$S_{(\DP^2, \lambda C)}(E)=\frac{1}{(3- d \lambda )^2}\Big(\int_0^{3- 4 \lambda } \big(v-(3- 4 \lambda )\big)\big(v+(3- 4 \lambda )\big) dv\Big)=\frac{2(3- 4 \lambda )}{3}$$
so that $\delta_P(\DP^2,\lambda C)\le \frac{3}{2}\cdot \frac{2-4\lambda}{3- 4 \lambda }= \frac{3(1-2\lambda)}{3- 4 \lambda } $. For every $O\in E$ we get:
$$h(v) = \frac{v^2}{2}\text{ if }v\in[0,3- 4 \lambda ].
$$
So that
$$S\big(W^{E}_{\bullet,\bullet};O\big)= \frac{2}{(3- 4 \lambda )^2}\Big(\int_0^{3- 4 \lambda } \frac{v^2}{2} dv\Big)=\frac{3- 4 \lambda }{3}$$
We have
$$
\delta_P(\DP^2,\lambda C)\geqslant\mathrm{min}\Bigg\{ \frac{3(1-2\lambda)}{3- 4 \lambda } ,\inf_{O\in E}\frac{A_{E,\Delta_{E}}(O)}{S\big(W^{E}_{\bullet,\bullet};O\big)}\Bigg\},
$$
where $\Delta_{E}=\lambda Q_1+3 \lambda Q_2$ where $ Q_1+ 3Q_2=\widetilde{C}|_E$. 
So that
$$
\frac{A_{\overline{E},\Delta_{\overline{E}}}(O)}{S(W_{\bullet,\bullet}^{\overline{E}};O)}=
\left\{\aligned
&\frac{3(1-\lambda)}{3- 4 \lambda }\ \mathrm{if}\ O=Q_1,\\
&\frac{3(1-3\lambda)}{3- 4 \lambda }\ \mathrm{if}\ O=Q_2,\\
&\frac{3}{3- 4 \lambda }\ \mathrm{otherwise}.
\endaligned
\right.
$$
Thus $\delta_P(\DP^2,\lambda C)= \frac{3(1-\lambda)}{3- 4 \lambda }$.
 \end{proof}
 \noindent This proves the following theorem:
  \begin{theorem}
   Suppose $C_4\subset \DP^2$ is a quartic curve which is a union of a triple line and another line. Then  we have:
$$\delta(\DP^2,\lambda C_4)= \frac{3(1-3\lambda)}{3- 4 \lambda }\text{ for }\lambda\in \Big[0, \frac{1}{3}\Big].$$
 \end{theorem}

 \section{Quadruple line}
  
   \begin{lemma} 
  Let $C\subset \DP^2$  be a quadruple line, $P\in C$ be a point on $C$. Then
 $$\delta_P(\DP^2,\lambda C)=\frac{3(1-4\lambda)}{3- 4 \lambda }\text{ for }\lambda\in\Big[0,\frac{1}{4}\Big].$$
 \end{lemma}
 \begin{proof}
Note that $\delta_P(\DP^2,\lambda C)\le \frac{A_{(\DP^2,\lambda C)}(\widetilde{L})}{S_{(\DP^2,\lambda C)}(\widetilde{L})}=\frac{3(1-4\lambda)}{3- 4 \lambda }$ where $\widetilde{L}$ is a strict transform of a quadruple line component of $C$. Let $\pi:\widetilde{S}\to S$ be the~blow up of the~point $P$,
with the exceptional divisor $E$.
\begin{center}
 \includegraphics[width=8cm]{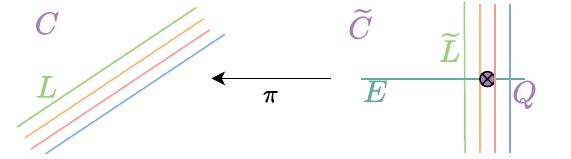}
 \end{center}
We have $\sigma^*(L)=\widetilde{L}+E$, $\pi^*(C)=\widetilde{C}+4E$, $\pi^*(K_{\DP^2})=K_{\widetilde{S}}-E$. Thus, $A_{(\DP^2,\lambda C)}(E)=2-4\lambda$.  The Zariski decomposition of the divisor  $-\pi^*(K_{\DP^2}+\lambda C)-vE$ is given by:
\begin{align*}
P(v)=-\pi^*(K_{\DP^2}+\lambda C)-vE \text{ and }N(v)=0\text{ if }v\in[0,3- 4 \lambda ].
\end{align*}
Then
$$
P(v)^2=(4\lambda - 3- v )(v-(3- 4\lambda))\text{ and }P(v)\cdot E=  v \text{ if }v\in[0,3- 4 \lambda ].
$$
Thus
$$S_{(\DP^2,\lambda C)}(E)=\frac{1}{(3- 4 \lambda )^2}\Big(\int_0^{3-4 \lambda } (4\lambda - 3- v )(v-(3- 4\lambda)) dv\Big)=\frac{2(3-4\lambda)}{3}$$
so that $\delta_P(\DP^2,\lambda C)\le \frac{3}{2}\cdot \frac{2-4\lambda}{3-4 \lambda }=\frac{3(1-2\lambda)}{3-4 \lambda }$. 
For every $O\in E$ we get:
$$h(v) = \frac{v^2}{2}\text{ if }v\in[0,3-4\lambda ].
$$
So that
$$S\big(W^{E}_{\bullet,\bullet};O\big)= \frac{2}{(3-4\lambda )^2}\Big(\int_0^{3-4\lambda } \frac{v^2}{2} dv\Big)=\frac{3-4 \lambda}{3}
$$
We have
$$
\delta_P(\DP^2,\lambda C)\geqslant\mathrm{min}\Bigg\{ \frac{3(1-2\lambda)}{3-4 \lambda } ,\inf_{O\in E}\frac{A_{E,\Delta_{E}}(O)}{S\big(W^{E}_{\bullet,\bullet};O\big)}\Bigg\},
$$
where $\Delta_{E}=4\lambda Q$ where $Q=\widetilde{L}|_E$. 
So that
$$
\frac{A_{\overline{E},\Delta_{\overline{E}}}(O)}{S(W_{\bullet,\bullet}^{\overline{E}};O)}=
\left\{\aligned
&\frac{3(1-4\lambda)}{3-4 \lambda }\ \mathrm{if}\ O=Q,\\
&\frac{3}{3-4 \lambda }\ \mathrm{otherwise}.
\endaligned
\right.
$$
Thus $\delta_P(\DP^2,\lambda C)=\frac{3(1-4\lambda)}{3- 4 \lambda }$ for $\lambda\in\big[0,\frac{1}{4}\big]$.     
 \end{proof}
 \noindent This proves the following theorem:
  \begin{theorem}
   Suppose $C_4\subset \DP^2$ is a quartic curve which is a quadruple line. Then  we have:
$$\delta(\DP^2,\lambda C_4)= \frac{3(1-4\lambda)}{3- 4 \lambda }\text{ for }\lambda\in \Big[0, \frac{1}{4}\Big].$$
 \end{theorem}

 \appendix 

\section{Cubic Plane curves with ADE singularities}

In this appendix, we provide a table with cubic plane curves with ADE singularities for convenience:
{\renewcommand{\arraystretch}{1.5} 
  \begin{longtable}[t]{ | c | c | c |  c |  } 
   \hline 
    $\mathrm{Sing}(C)$ & picture  & Equation 
& $\delta(\DP^2,\lambda C)$ \\
  \hline
\endhead 
\hline
  \multirow{2}{*}{$A_1$} & \multirow{2}{*}{\includegraphics[width=1cm]{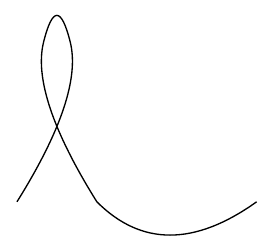}}
  &  $z^3+xy^2+\alpha y^2z+\beta yz^2 +xyz=0$ & \multirow{2}{*}{ $1$} \\
  & & $\alpha\ne 0$, $\beta^2\ne 4\alpha$, $1+\alpha-\beta \ne 0$ &  \\\hline
   \multirow{2}{*}{$2A_1$} & \multirow{2}{*}{\includegraphics[width=1cm]{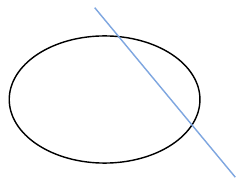}}
  &  $y(x^2+ \alpha yz +xz +xy)=0$ & \multirow{2}{*}{ $1$} \\
  & & $\alpha\ne 0$, $\alpha\ne 1$ &  \\\hline
     \multirow{2}{*}{$3A_1$} & \multirow{2}{*}{\includegraphics[width=1cm]{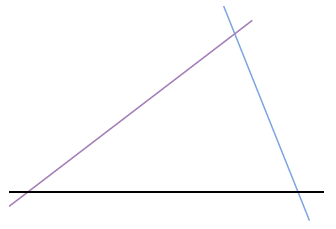}}
  &  \multirow{2}{*}{$yz(x+y+z)=0$} & \multirow{2}{*}{ $1$} \\
  & &  &  \\\hline
  \multirow{2}{*}{$A_2$} & \multirow{2}{*}{\includegraphics[width=1cm]{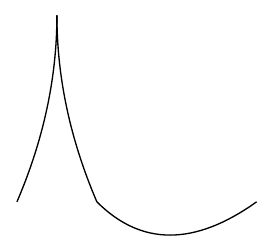}} &
 $z^3+xy^2+y^2z+\alpha yz^2=0$ & \multirow{2}{*}{\large{$\frac{5-6\lambda}{5-5 \lambda }$}}\\
 & & $\alpha\ne 4$  & \\
\hline
 \multirow{2}{*}{$A_3$} &  \multirow{2}{*}{\includegraphics[width=1cm]{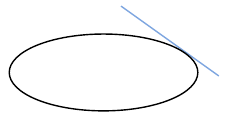}} &
    \multirow{2}{*}{$y(x^2+yz+xz)=0$}
  &  \multirow{2}{*}{$\frac{3-4\lambda}{3-3 \lambda }$} \\
  & &    & \\\hline
   \multirow{2}{*}{$D_4$} & \multirow{2}{*}{\includegraphics[width=1cm]{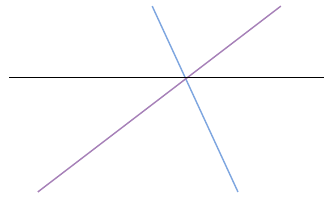}} &
 \multirow{2}{*}{$xyz=0$} & \multirow{2}{*}{\large{$\frac{2-3\lambda}{2-2 \lambda }$}}\\
 & &    & \\
\hline
  \end{longtable}

 \section{Quartic  Plane curves with ADE singularities}
Plane quartic curves we classified in \cite{Hui}.  In this appendix, we provide a table with $\delta$-invariants of quartic plane curves with ADE singularities for convenience:
{\allowdisplaybreaks\renewcommand{\arraystretch}{1.5} 
  \begin{longtable}[t]{ | c | c | c |  c | c | } 
   \hline 
    $\mathrm{Sing}(C)$ & picture  & Equation 
& $\delta(\DP^2,\lambda C)$ & $\lambda=\frac{1}{2}$\\
  \hline
\endhead 
\hline
  \multirow{2}{*}{$A_1$} & \multirow{2}{*}{\includegraphics[width=1cm]{qc_A1.pdf}}
  &  $x^2z^2+yz^3+x^2yz+\alpha y^4+\beta xy^2z+$ & \multirow{2}{*}{ \large{$\frac{3}{2}\cdot\frac{2-2\lambda}{3-4 \lambda }$}} &\multirow{2}{*}{\large{$\frac{3}{2}$}}\\
  & & $+\gamma y^3z+\delta xyz^2+\epsilon y^2z^2=0$ &  &\\
\hline
 \multirow{2}{*}{$2A_1$} &  \multirow{2}{*}{\includegraphics[width=1cm]{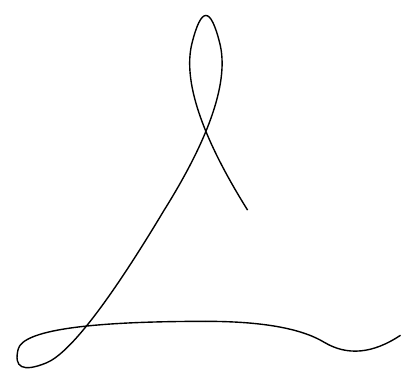}} &
 $x^2y^2+x^2z^2+y^2z^2+\lambda z^4 +\alpha xyz^2+\beta xz^3+\gamma yz^3=0$ & \multirow{2}{*}{ \large{$\frac{3}{2}\cdot\frac{2-2\lambda}{3-4 \lambda }$}} &\multirow{2}{*}{{\large$\frac{3}{2}$}}\\
  & & $(\lambda,\beta,\gamma)\ne (0,0,0)$ & &\\
\hline
\multirow{2}{*}{$3A_1$} & \multirow{2}{*}{\includegraphics[width=1cm]{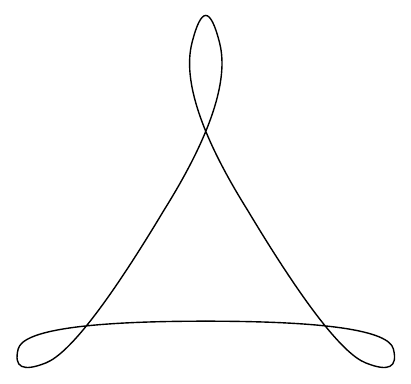}}
  & $x^2y^2+x^2z^2+y^2z^2=xyz(\alpha x+\beta y+\gamma z)$ & \multirow{2}{*}{\large{$\frac{3}{2}\cdot\frac{2-2\lambda}{3-4 \lambda }$}} & \multirow{2}{*}{\large{$\frac{3}{2}$}}\\
  & & $\alpha^2\ne 4$, $\beta^2\ne 4$, $\gamma^2\ne 4$ & & \\
\hline
 \multirow{2}{*}{$3A_1$} &  \multirow{2}{*}{\includegraphics[width=1cm]{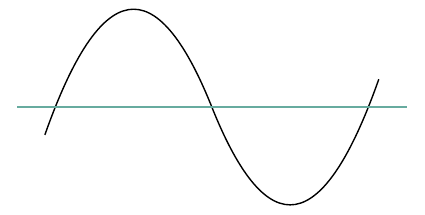}}
  & $x(x^2z+\alpha x z^2 +\beta xyz+y^3+\gamma y^2z +yz^2)=0$
  & \multirow{2}{*}{\large{$\frac{3}{2}\cdot\frac{2-2\lambda}{3-4 \lambda }$}} &   \multirow{2}{*}{\large{$\frac{3}{2}$}}\\
  & & $\gamma^2\ne 4$  & &\\
\hline
\multirow{2}{*}{$4A_1$} & \multirow{2}{*}{\includegraphics[width=1cm]{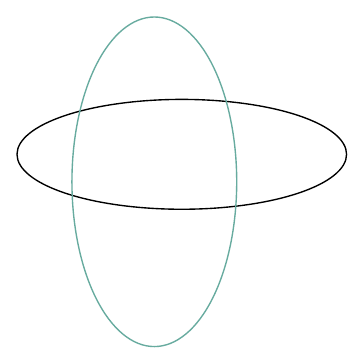}}
  &
  $(\alpha yz +\beta xz+xy)(yz+xz+xy)=0$
  & \multirow{2}{*}{\large{$\frac{3}{2}\cdot\frac{2-2\lambda}{3-4 \lambda }$}} &\multirow{2}{*}{ \large{$\frac{3}{2}$}}\\
   & & $\alpha\beta\ne 0$, $\alpha\ne 1$, $\beta\ne 1$, $\alpha\ne \beta$  & &\\
\hline
\multirow{2}{*}{$4A_1$} & \multirow{2}{*}{\includegraphics[width=1cm]{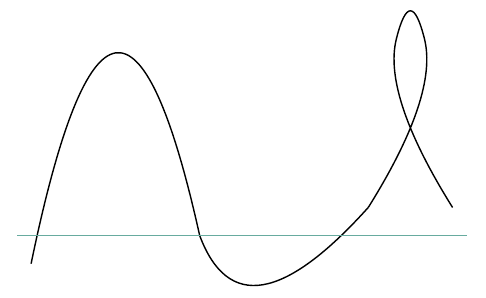}}
  & 
  $x(z^3+xy^2+\alpha y^2z+\beta yz^2 +xyz)=0$
  & \multirow{2}{*}{\large{$\frac{3}{2}\cdot\frac{2-2\lambda}{3-4\lambda }$}} & \multirow{2}{*}{\large{$\frac{3}{2}$}}\\ 
  & & $\alpha\ne 0$, $\beta^2\ne 4\alpha$, $1+\alpha-\beta \ne 0$ & & \\
  \hline
 \multirow{2}{*}{$5A_1$}  & \multirow{2}{*}{\includegraphics[width=1cm]{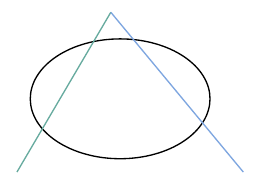}}
  & 
  $yz(x^2+ \alpha yz +xz +xy)=0$
  & \multirow{2}{*}{\large{$\frac{3}{2}\cdot\frac{2-2\lambda}{3-4\lambda }$}} &  \multirow{2}{*}{\large{$\frac{3}{2}$ }}\\
   & & $\alpha\ne 0$, $\alpha\ne 1$  & &\\
\hline
\multirow{2}{*}{$6A_1$}  & \multirow{2}{*}{\includegraphics[width=1cm]{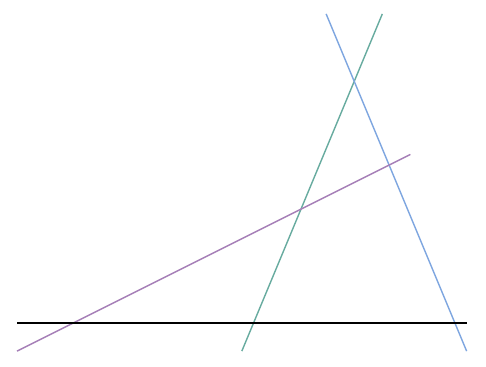}}
  & \multirow{2}{*}{$xyz(x+y+z)=0$} & \multirow{2}{*}{\large{$\frac{3}{2}\cdot\frac{2-2\lambda}{3-4 \lambda }$}} & \multirow{2}{*}{\large{$\frac{3}{2}$}}\\
   & &   & &\\
\hline
 \multirow{2}{*}{$A_2$} & \multirow{2}{*}{\includegraphics[width=1cm]{qc_A2.pdf}}
  & 
  \multirow{2}{*}{$x^2z^2+yz^3+xy^3+\alpha y^2z+ \beta xy^2z +\gamma y^2z^2+\delta xyz^2=0 $}
  & \multirow{2}{*}{\large{$\frac{3}{5}\cdot\frac{5-6\lambda}{3-4 \lambda }$}} & \multirow{2}{*}{\large{$\frac{6}{5}$}} \\
  & &   & &\\
\hline
 \multirow{2}{*}{$A_2+A_1$} & \multirow{2}{*}{\includegraphics[width=1cm]{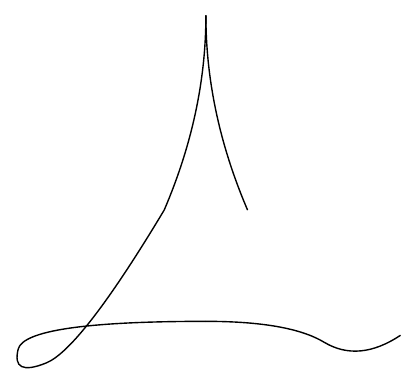}} &
 \multirow{2}{*}{$ x^2y^2+y^2z^2+\alpha xyz^2+xy^3+\gamma yz^2+\delta z^4=0$} & \multirow{2}{*}{\large{$\frac{3}{5}\cdot\frac{5-6\lambda}{3-4 \lambda }$}} &  \multirow{2}{*}{\large{$\frac{6}{5}$}}\\
 & &   & &\\
\hline
 \multirow{2}{*}{$A_2+2A_1$} & \multirow{2}{*}{\includegraphics[width=1cm]{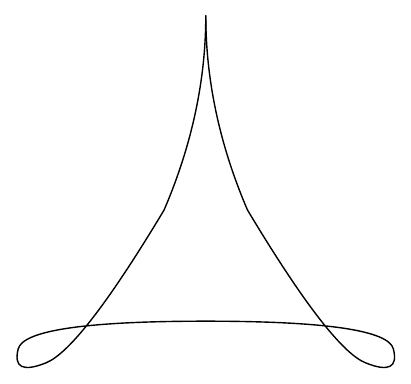}} & 
 $x^2y^2+x^2z^2+y^2z^2=xyz(2 x+\beta y+\gamma z)$ & \multirow{2}{*}{\large{$\frac{3}{5}\cdot\frac{5-6\lambda}{3-4 \lambda }$}} &  \multirow{2}{*}{\large{$\frac{6}{5}$}}\\
  & &  $\beta^2\ne 4$, $\gamma^2\ne 4$, $\beta\pm \gamma\ne 0$ & &\\
\hline
 \multirow{2}{*}{$A_2+3A_1$} & \multirow{2}{*}{\includegraphics[width=1cm]{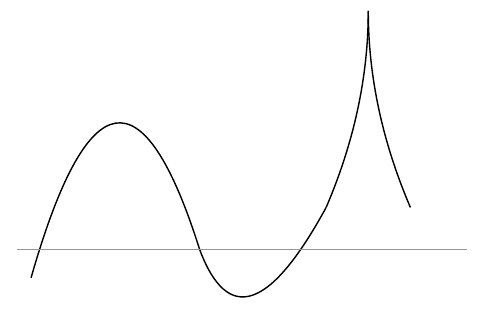}} &
 $x(z^3+xy^2+y^2z+\alpha yz^2)=0$ & \multirow{2}{*}{\large{$\frac{3}{5}\cdot\frac{5-6\lambda}{3-4 \lambda }$}}
 &  \multirow{2}{*}{\large{$\frac{6}{5}$}}\\
 & & $\alpha\ne 4$  & &\\
\hline
 \multirow{2}{*}{$2A_2$} & \multirow{2}{*}{\includegraphics[width=1cm]{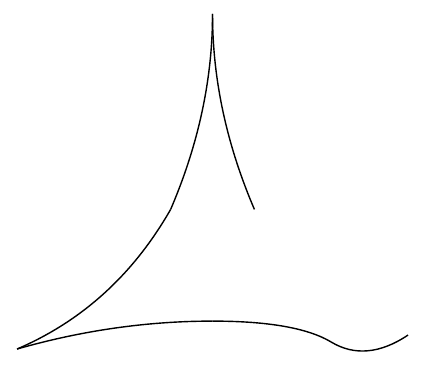}} & \multirow{2}{*}{$x^2y^2+\alpha xyz^2+xz^3+yz^3+\delta z^4=0 $} & \multirow{2}{*}{\large{$\frac{3}{5}\cdot\frac{5-6\lambda}{3-4 \lambda }$}} & \multirow{2}{*}{\large{$\frac{6}{5}$}}\\
 & &   & & \\
\hline
\multirow{2}{*}{$2A_2+A_1$} & \multirow{2}{*}{\includegraphics[width=1cm]{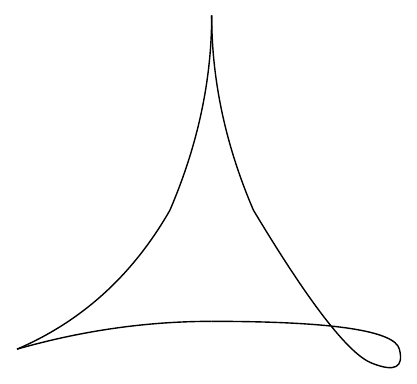}} &
$x^2y^2+x^2z^2+y^2z^2=xyz(2 x+2 y+\gamma z)$& \multirow{2}{*}{\large{$\frac{3}{5}\cdot\frac{5-6\lambda}{3-4 \lambda }$}} &  \multirow{2}{*}{\large{$\frac{6}{5}$}}\\
& &   $\gamma^2\ne 4$ & &\\
\hline
 \multirow{2}{*}{$3A_2$} & \multirow{2}{*}{\includegraphics[width=1cm]{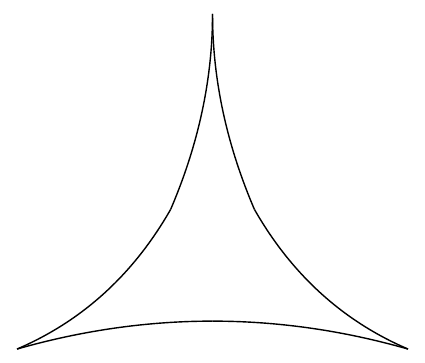}} & \multirow{2}{*}{$x^2y^2+x^2z^2+y^2z^2=xyz(2 x+2 y+2 z)$} & \multirow{2}{*}{\large{$\frac{3}{5}\cdot\frac{5-6\lambda}{3-4 \lambda }$}} &  \multirow{2}{*}{\large{$\frac{6}{5}$}}\\
 & &    & &\\
\hline
  \multirow{2}{*}{$A_3$} & \multirow{2}{*}{\includegraphics[width=1cm]{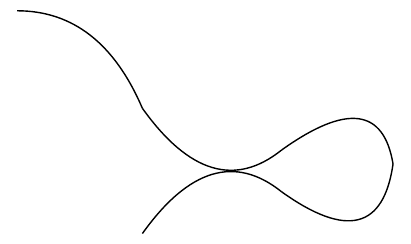}} & 
  $x^2z^2+y^4+x^3y+\alpha xy^2z+\beta x^2y^2+\gamma xy^3=0$ & \multirow{2}{*}{$1$} & \multirow{2}{*}{$1$}\\
  & &   $\alpha^2\ne 4$ & &\\
  \hline
\multirow{2}{*}{$A_3+A_1$} & \multirow{2}{*}{\includegraphics[width=1cm]{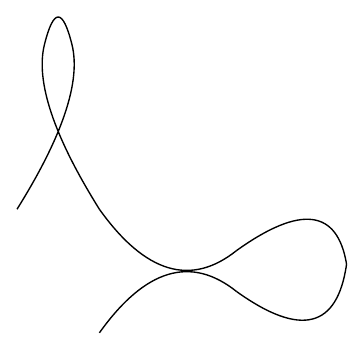}} &  
$y^4+x^2z^2+x^2yz+\alpha xy^2z+\beta xy^3=0$ & \multirow{2}{*}{$1$} & \multirow{2}{*}{$1$}\\
& &   $\alpha^2\ne 4$, $\beta^2-\alpha\beta+1\ne 0$ & & \\
\hline
 \multirow{2}{*}{$A_3+A_1$} & \multirow{2}{*}{\includegraphics[width=1cm]{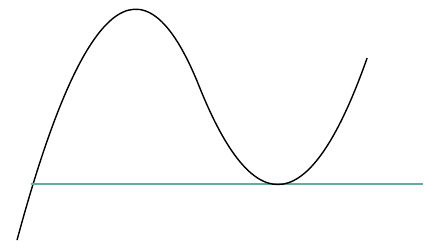}} & 
 $x(x^2z+\alpha xz^2+\beta xyz+y^3+y^2z)=0$
  & \multirow{2}{*}{$1$} & \multirow{2}{*}{$1$}\\
  & &   $\alpha\ne 0$ & &\\
\hline 
 \multirow{2}{*}{$A_3+2A_1$} & \multirow{2}{*}{\includegraphics[width=1cm]{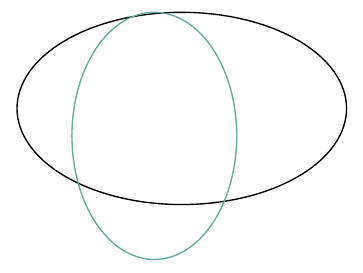}} &
 $(\alpha yz +xz+xy)(yz+xz+xy)=0 $
 & \multirow{2}{*}{$1$} & \multirow{2}{*}{$1$}\\
  & &   $\alpha\ne 0$, $\alpha\ne 1$ & &\\
\hline
 \multirow{2}{*}{$A_3+2A_1$} & \multirow{2}{*}{\includegraphics[width=1cm]{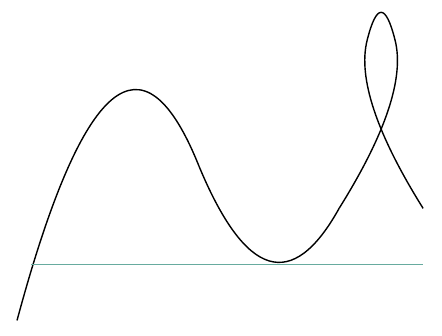}} & 
 $x(z^3+xy^2+ \alpha yz^2 + xyz)=0$
 & \multirow{2}{*}{$1$} & \multirow{2}{*}{$1$}\\
 & &  $\alpha \ne 0$,  $\alpha \ne 1$  & &\\
\hline
  \multirow{2}{*}{$A_3+3A_1$} &  \multirow{2}{*}{\includegraphics[width=1cm]{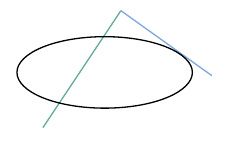}} &
    \multirow{2}{*}{$yz(x^2+yz+xz)=0$}
  &  \multirow{2}{*}{$1$} & \multirow{2}{*}{$1$}\\
  & &    & &\\
\hline
 \multirow{2}{*}{$A_3+A_2$} & \multirow{2}{*}{\includegraphics[width=1cm]{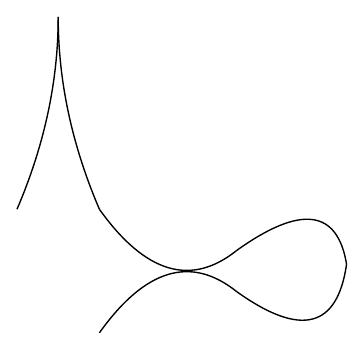}} & 
 $y^4+x^2z^2+\alpha xy^2z+ xy^3=0$ & \multirow{2}{*}{$1$} & \multirow{2}{*}{$1$}\\
 & &   $\alpha^2\ne 4$ & &\\
\hline
  \multirow{2}{*}{$A_3+A_2+A_1$} & \multirow{2}{*}{\includegraphics[width=1cm]{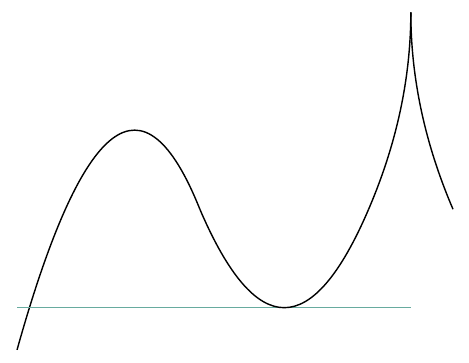}} &
  \multirow{2}{*}{$x(z^3+xy^2+yz^2)=0 $}
  & \multirow{2}{*}{$1$} & \multirow{2}{*}{$1$}\\
  & &    & & \\
\hline
\multirow{2}{*}{$2A_3$} & \multirow{2}{*}{\includegraphics[width=1cm]{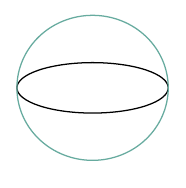}} & 
$(x^2+yz)(x^2+ \alpha yz)=0$
& \multirow{2}{*}{$1$} & \multirow{2}{*}{$1$}\\
 & &  $\alpha\ne 0$, $\alpha\ne 1$  & &\\
\hline
 \multirow{2}{*}{$2A_3+A_1$} & \multirow{2}{*}{\includegraphics[width=1cm]{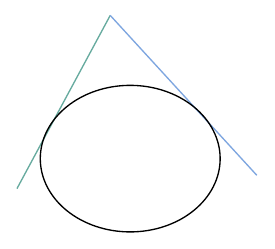}} &
 \multirow{2}{*}{$yz(x^2+yz)=0$}
 & \multirow{2}{*}{$1$} & \multirow{2}{*}{$1$}\\
   & &    & &\\
\hline
\multirow{2}{*}{$A_4$} & \multirow{2}{*}{\includegraphics[width=1cm]{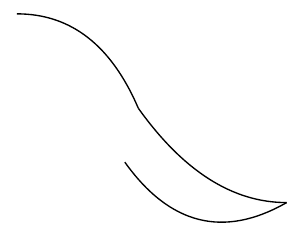}} & 
$ x^2z^2+2xy^2z+y^4+x^3y+\beta x^2y^2+\gamma xy^3=0$ & \multirow{2}{*}{\large{$\frac{6}{13}\cdot\frac{7-10\lambda}{3-4\lambda }$}} & \multirow{2}{*}{\large{$\frac{12}{13}$}}\\
& & $\gamma\ne 0$, $\beta\ne 4\gamma$ & &\\
\hline
 \multirow{2}{*}{$A_4+A_1$} & \multirow{2}{*}{\includegraphics[width=1cm]{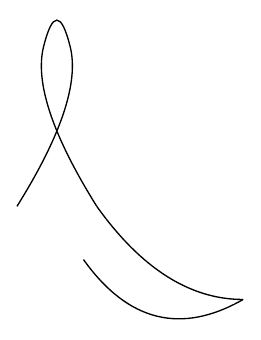}} & $(xz+y^2)^2=\alpha x^2y^2+xy^3 $ & \multirow{2}{*}{\large{$\frac{6}{13}\cdot\frac{7-10\lambda}{3-4 \lambda }$}} & \multirow{2}{*}{\large{$\frac{12}{13}$}}\\
 & &   $\alpha^2\ne 0$ & &\\
\hline
 \multirow{2}{*}{$A_4+A_2$} & \multirow{2}{*}{\includegraphics[width=1cm]{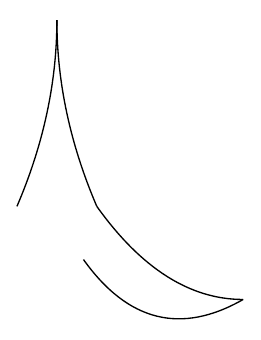}} & 
 \multirow{2}{*}{$(xz+y^2)^2=xy^3 $}& \multirow{2}{*}{\large{$\frac{6}{13}\cdot\frac{7-10\lambda}{3-4 \lambda }$}}  & \multirow{2}{*}{\large{$\frac{12}{13}$}}\\
  & &  & &\\
\hline
 \multirow{2}{*}{$A_5$} & \multirow{2}{*}{\includegraphics[width=1cm]{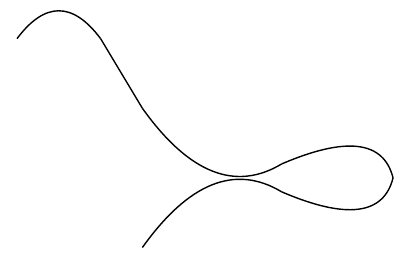}} &
 $x^2z^2+zy^2z+y^4=x^2y^2+\alpha x^3 y$
& \multirow{2}{*}{\large{$\frac{6}{7}\cdot\frac{4-6\lambda}{3-4 \lambda }$}}  &  \multirow{2}{*}{\large{$\frac{6}{7}$}}\\
 & &  $\alpha^2\ne 0$  & &\\
\hline
 \multirow{2}{*}{$A_5$} & \multirow{2}{*}{\includegraphics[width=1cm]{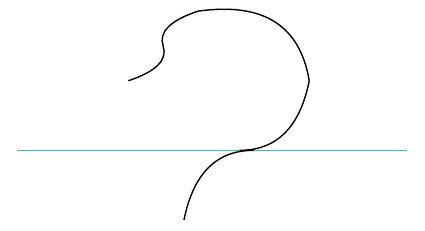}} & 
  $x(x^2z + xz^2+ \beta xyz+y^3)=0$ & \multirow{2}{*}{\large{$\frac{3}{4}\cdot\frac{4-6\lambda}{3-4 \lambda }$}} &  \multirow{2}{*}{\large{$\frac{3}{4}$}}\\
 & &    $\beta^3+27\ne 0$ & &\\
\hline
 \multirow{2}{*}{$A_5+A_1$}  & \multirow{2}{*}{\includegraphics[width=1cm]{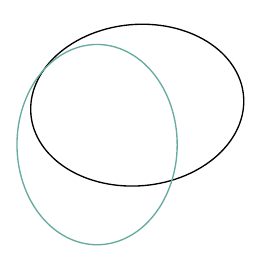}} & 
 \multirow{2}{*}{$(z^2+xy)(z^2+yz+xy)=0$} & \multirow{2}{*}{\large{$\frac{6}{7}\cdot\frac{4-6\lambda}{3-4 \lambda }$}}
 &  \multirow{2}{*}{\large{$\frac{6}{7}$}}\\
  & &    & &  \\
\hline 
  \multirow{2}{*}{$A_5+A_1$}  &  \multirow{2}{*}{\includegraphics[width=1cm]{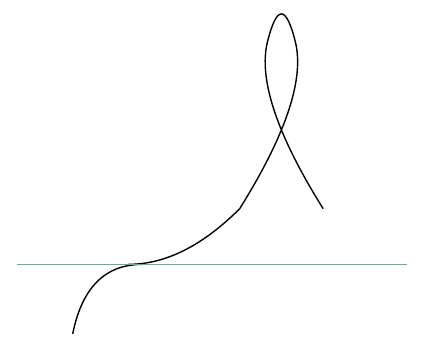}} &
  \multirow{2}{*}{ $x(z^3+xy^2+xyz)=0$}& \multirow{2}{*}{\large{$\frac{3}{4}\cdot\frac{4-6\lambda}{3-4 \lambda }$}}
  &   \multirow{2}{*}{\large{$\frac{3}{4}$}}\\
   & &    & &\\
\hline
 \multirow{2}{*}{$A_5+A_2$}  & \multirow{2}{*}{\includegraphics[width=1cm]{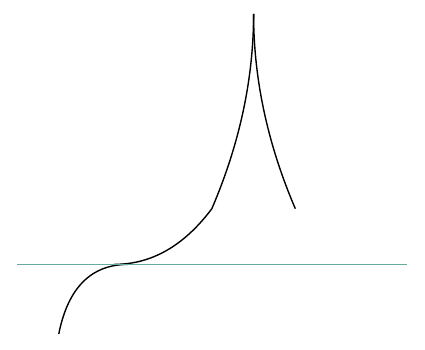}} &
  \multirow{2}{*}{$x(z^3+xy^2)=0$} & \multirow{2}{*}{\large{$\frac{3}{4}\cdot\frac{4-6\lambda}{3-4 \lambda }$}}
 &  \multirow{2}{*}{\large{$\frac{3}{4}$}}\\
  & &     & &\\
\hline
 \multirow{2}{*}{$A_6$} & \multirow{2}{*}{\includegraphics[width=1cm]{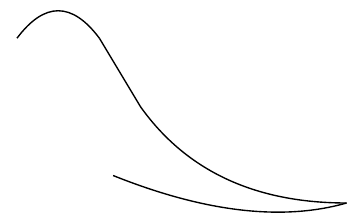}} & \multirow{2}{*}{$(xz+y^2)^2=x^3y$} & \multirow{2}{*}{\large{$\frac{2}{5}\cdot\frac{9-14\lambda}{3-4 \lambda }$}} &  \multirow{2}{*}{\large{$\frac{4}{5}$}}\\
 & &   & &\\
\hline 
 \multirow{2}{*}{$A_7$} & \multirow{2}{*}{\includegraphics[width=1cm]{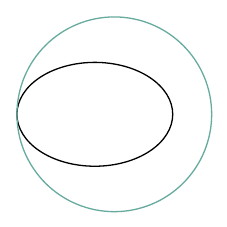}} & 
 \multirow{2}{*}{$(x^2+yz)(x^2 +yz +y^2)=0$} & \multirow{2}{*}{\large{$\frac{3}{4}\cdot\frac{5-8\lambda}{3-4\lambda }$}}
 &  \multirow{2}{*}{\large{$\frac{3}{4}$}}\\
  & &   & &\\
\hline
 \multirow{2}{*}{$D_4$} & \multirow{2}{*}{\includegraphics[width=1.2cm]{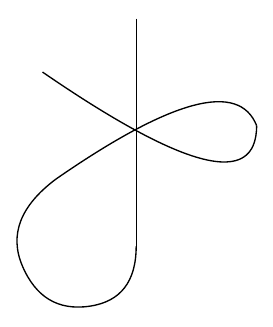}} & $xyz(x+y)=x^4+\alpha x^2y^2+\beta y^4$ & \multirow{2}{*}{\large{$\frac{3}{2}\cdot\frac{2-3\lambda}{3-4 \lambda }$}} &  \multirow{2}{*}{\large{$\frac{3}{4}$}}\\
 & &   $\beta\ne 0$, $\alpha+\beta+1\ne 0$ & &\\
\hline
 \multirow{2}{*}{$D_4+A_1$} & \multirow{2}{*}{\includegraphics[width=1cm]{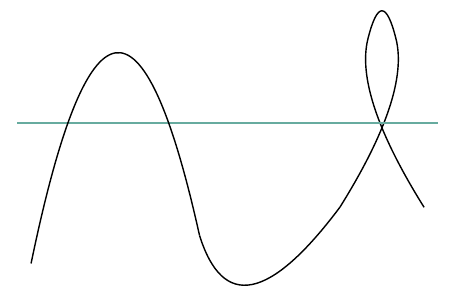}} &
 $x(x^3+\alpha x^2z +yz^2 +xyz) $ & \multirow{2}{*}{\large{$\frac{3}{2}\cdot\frac{2-3\lambda}{3-4 \lambda }$}}
 &  \multirow{2}{*}{\large{$\frac{3}{4}$}}\\
 & & $\alpha \ne 1$  & &\\
\hline
 \multirow{2}{*}{$D_4+2A_1$} & \multirow{2}{*}{\includegraphics[width=1cm]{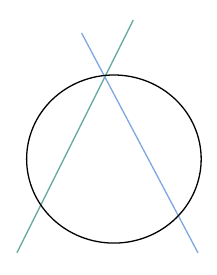}} &
  \multirow{2}{*}{$yz(yz+xz+xy)=0 $} & \multirow{2}{*}{\large{$\frac{3}{2}\cdot\frac{2-3\lambda}{3-4 \lambda }$}}
 &  \multirow{2}{*}{\large{$\frac{3}{4}$}}\\
  & &    & &\\
\hline
 \multirow{2}{*}{$D_4+3A_1$} & \multirow{2}{*}{\includegraphics[width=1cm]{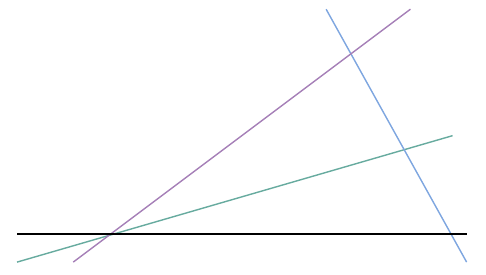}} &
 \multirow{2}{*}{$xyz(x+z)=0$} & \multirow{2}{*}{\large{$\frac{3}{2}\cdot\frac{2-3\lambda}{3-4 \lambda }$}}
 &  \multirow{2}{*}{\large{$\frac{3}{4}$}}\\
 & &    & &\\
\hline
 \multirow{2}{*}{$D_5$} & \multirow{2}{*}{\includegraphics[width=1cm]{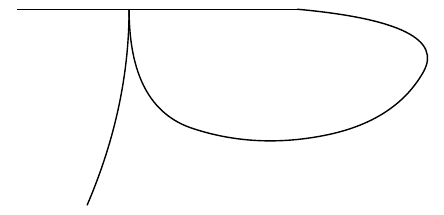}} & \multirow{2}{*}{$x^2yz=x^4+\alpha xy^3 +y^4$} & \multirow{2}{*}{\large{$\frac{3}{5}\cdot\frac{5-8\lambda}{3-4 \lambda }$}} &  \multirow{2}{*}{\large{$\frac{3}{5}$}}\\
  & &   & &\\
\hline
 \multirow{2}{*}{$D_5+A_1$} & \multirow{2}{*}{\includegraphics[width=1cm]{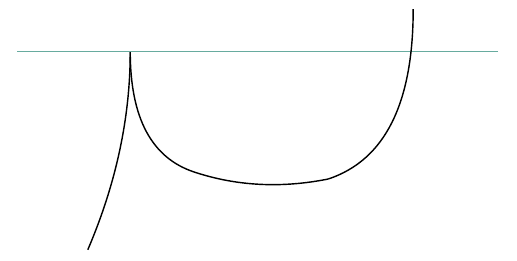}} &
 $x (x^3 +x^2z+yz^2)=0$ or  $x (x^3 +yz^2)=0$ & \multirow{2}{*}{\large{$\frac{3}{5}\cdot\frac{5-8\lambda}{3-4 \lambda }$}}
 & \multirow{2}{*}{ \large{$\frac{3}{5}$}}\\
  & & two orbits  & &\\
\hline
\multirow{2}{*}{ $D_6$} & \multirow{2}{*}{\includegraphics[width=1cm]{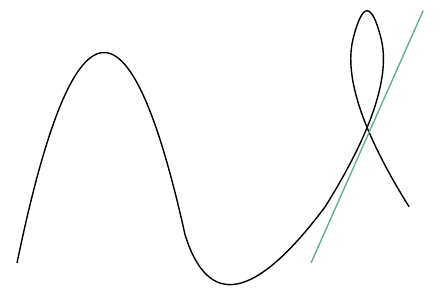}} &
 \multirow{2}{*}{$x(x^3+z^3+xyz)=0$} & \multirow{2}{*}{\large{$\frac{3-5\lambda}{3-4 \lambda }$}}
 &  \multirow{2}{*}{\large{$\frac{1}{2}$}}\\
 & &    & &\\
\hline
  \multirow{2}{*}{$D_6+A_1$} &  \multirow{2}{*}{\includegraphics[width=1cm]{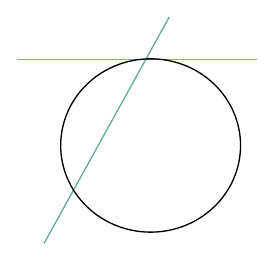}} &
  \multirow{2}{*}{$yz(y^2+xz)=0$} & \multirow{2}{*}{\large{$\frac{3-5\lambda}{3-4 \lambda }$}}
  &   \multirow{2}{*}{\large{$\frac{1}{2}$}}\\
  & &    & &\\
\hline
  \multirow{2}{*}{$E_6$} & \multirow{2}{*}{\includegraphics[width=0.5cm]{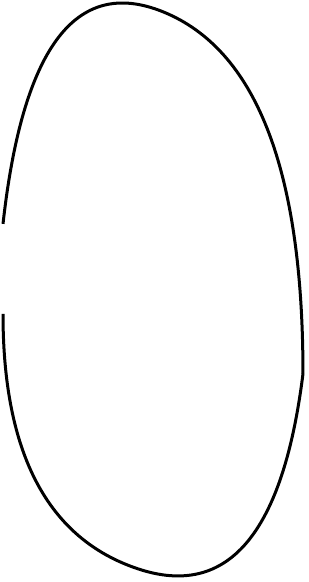}} &
  $x^3z = x^2y^2+y^4$ or $x^3z=y^4$ & \multirow{2}{*}{\large{$\frac{3}{7}\cdot \frac{7-12\lambda}{3-4 \lambda }$}} &  \multirow{2}{*}{\large{$\frac{3}{7}$}}\\
   & &   two  orbits & &\\
\hline
 \multirow{2}{*}{ $E_7$} &  \multirow{2}{*}{\includegraphics[width=1cm]{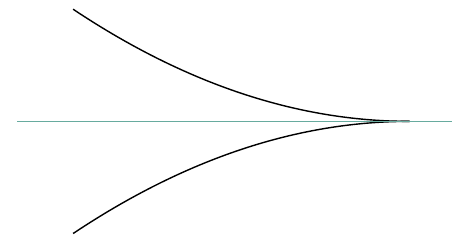}} & 
  \multirow{2}{*}{$x(z^3+x^2y+xz^2)=0$} & \multirow{2}{*}{\large{$\frac{3}{5}\cdot \frac{5-9\lambda}{3-4 \lambda }$}}
 &  \multirow{2}{*}{\large{$\frac{3}{10}$}}\\
  & &   & &\\
\hline
  \end{longtable}
}
\newpage
\section{ Reduced Quartic Curves }
In this appendix we provide the pictures which illustrate reduced quartic curves.


\begin{figure}[H]\label{Irreducible Quartic Curves}
    \centering
 \includegraphics[width=14cm]
 {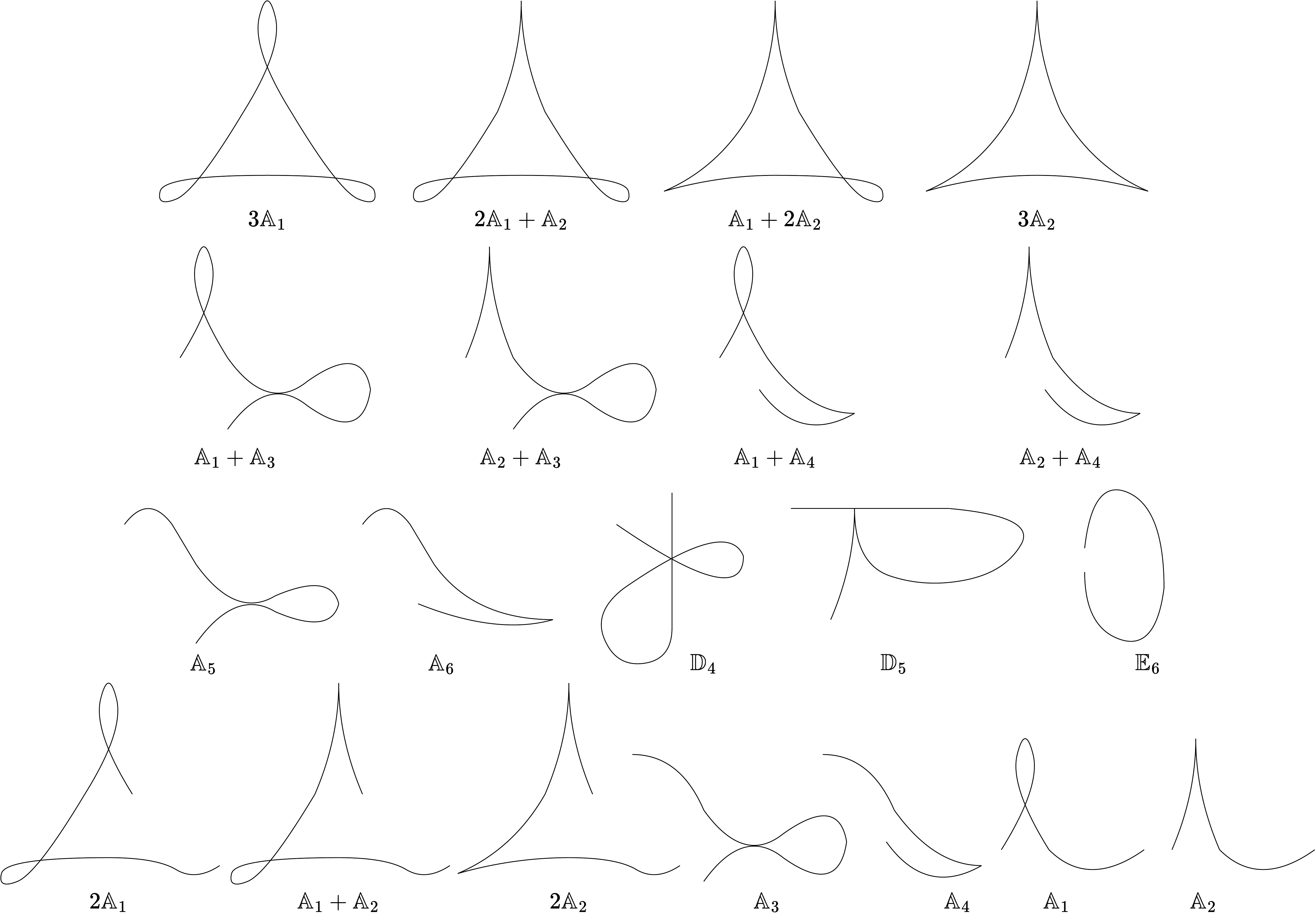}
 \caption{Irreducible Quartic Curves}
 \end{figure}
 

 \begin{figure}[H]\label{Line and Irreducible cubic}
    \centering
 \includegraphics[width=14cm]{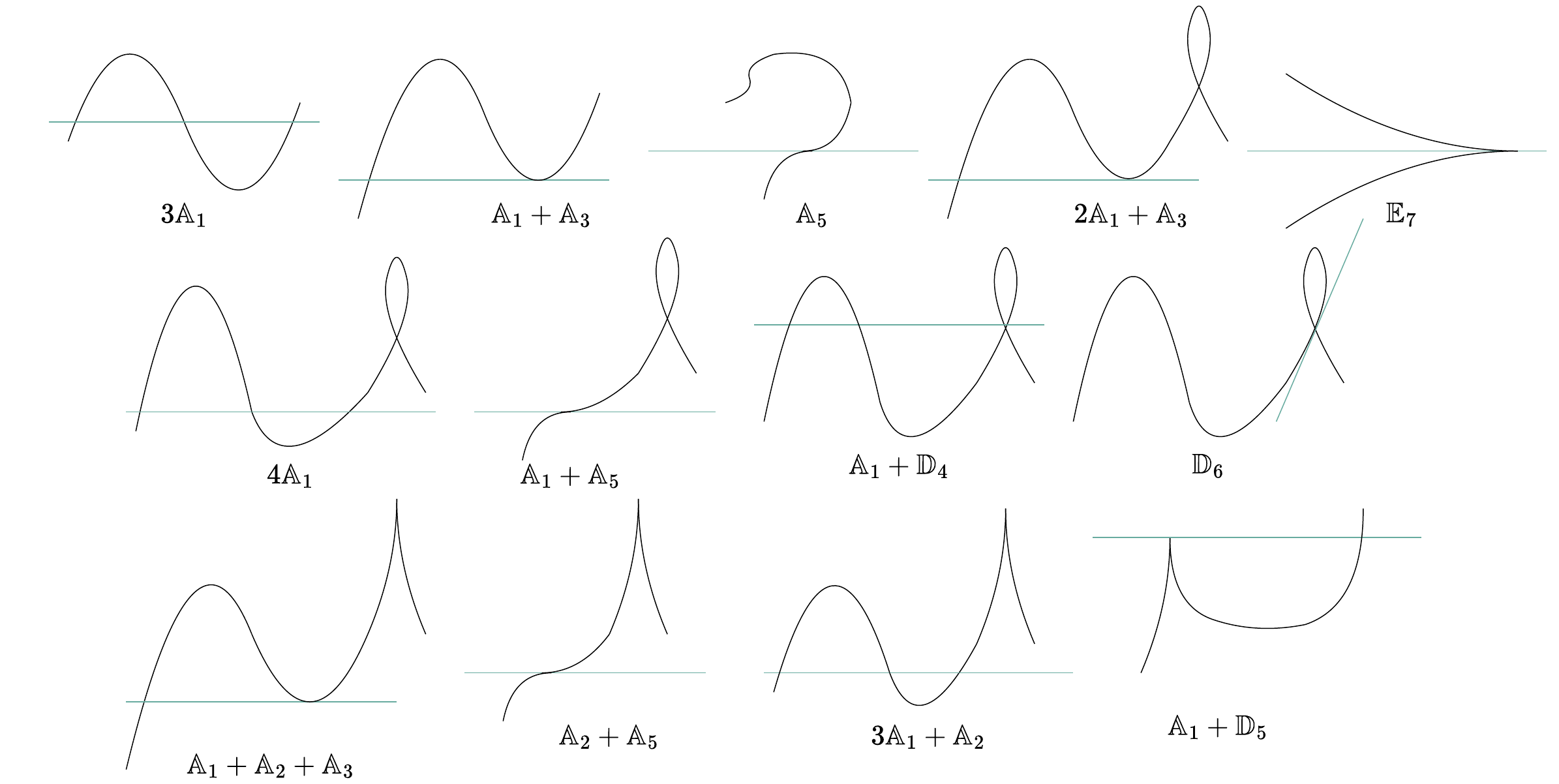}
 \caption{Line and Irreducible cubic}
 \end{figure}
 
  \begin{figure}[H]
  \label{Two Conics}
    \centering
 \includegraphics[width=14cm]{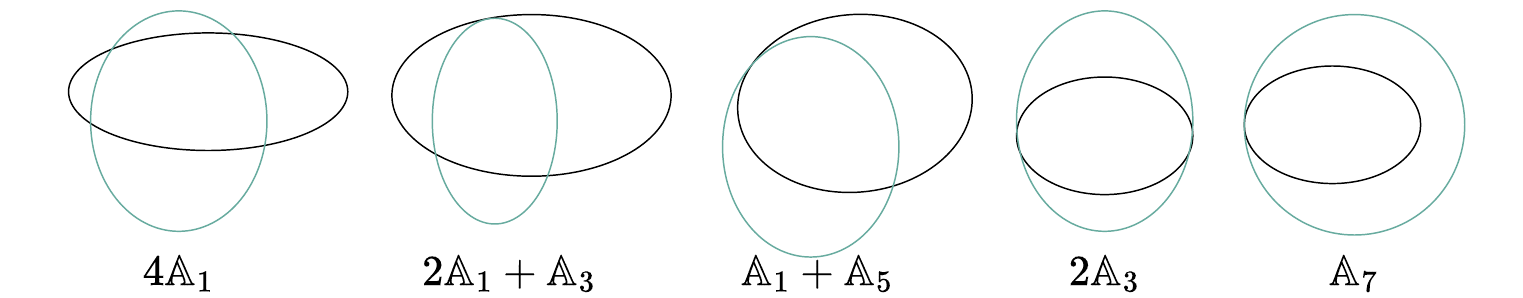}
 \caption{Two Conics}
 \end{figure}
   \begin{figure}[h!]
   \begin{center} \label{Conic and Two Lines}
 \includegraphics[width=14cm]{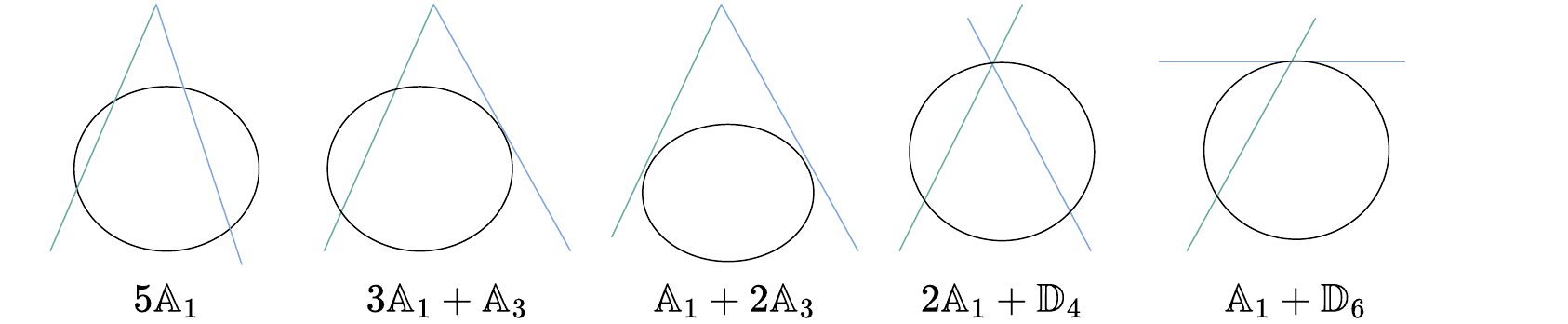}
 \caption{Conic and Two Lines}
 \end{center}
 \end{figure}
 \begin{figure}[H]
 \label{Four Lines}
   \begin{center}
 \includegraphics[width=10cm]{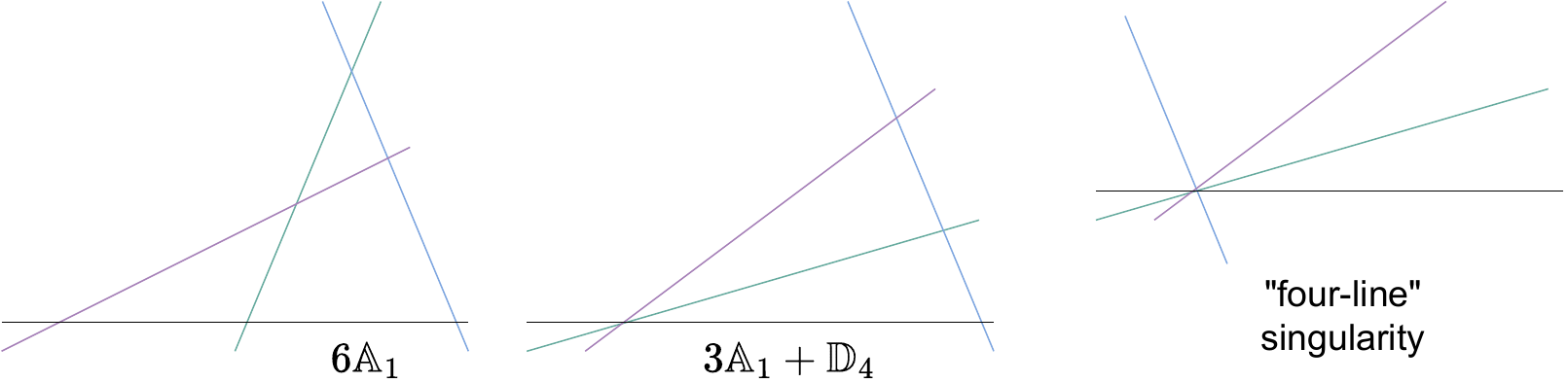}
 \caption{Four Lines}
 \end{center}
 \end{figure}

\section{Non-Reduced  Curves }
In this appendix we provide the pictures which illustrate non-reduced  curves for degree less than four.

\begin{figure}[H]
    \centering
    \begin{minipage}{.4\textwidth}
        \centering
        \includegraphics[width=2cm]{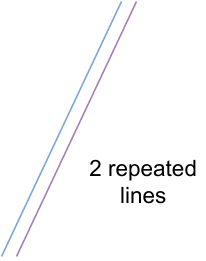}
 \caption{Non-Reduced Conic}
        \label{fig:prob1_6_2}
    \end{minipage}%
    \begin{minipage}{0.5\textwidth}
        \centering
        \includegraphics[width=5cm]{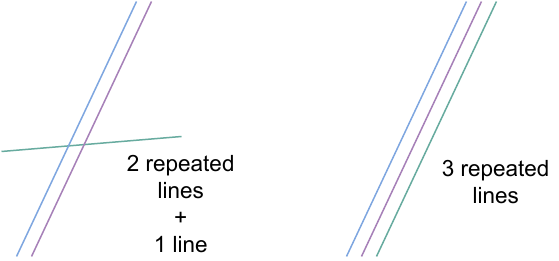}
 \caption{Non-Reduced Cubic Curves}
        \label{fig:prob1_6_1}
    \end{minipage}\\
   \begin{center}
 \includegraphics[width=18cm]{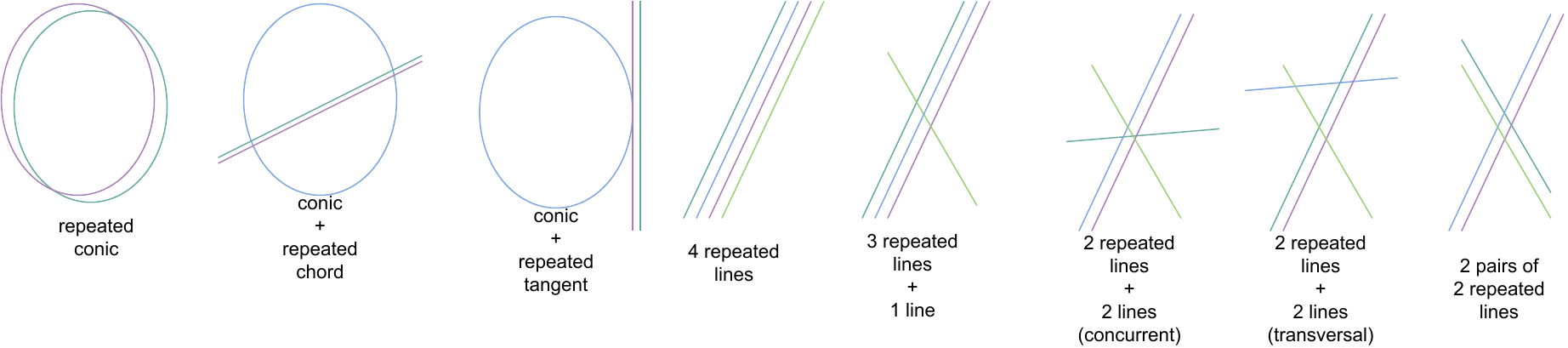}
 \caption{Non-Reduced Quartic Curves}
 \end{center}
 \end{figure}

\tableofcontents

\end{document}